\documentclass[a4paper,11pt]{amsart}
\usepackage{fullpage}
\usepackage{amsmath}
\usepackage{amssymb}
\usepackage{enumerate}
\usepackage{amsbsy}
\usepackage{amsfonts}
\usepackage{color}
\usepackage{slashed}
\usepackage{upgreek}
\usepackage{mathrsfs}  
\usepackage{tocbasic}

\newtheorem{thm}{Theorem}[section]
\newtheorem{lem}[thm]{Lemma}
\newtheorem{prop}[thm]{Proposition}

\newtheorem{defn}[thm]{Definition}
\newtheorem{rem}[thm]{Remark}

\numberwithin{equation}{section}

\begin{document}

	\title[Strichartz for half Klein-Gordon and applications to Dirac]{Strichartz estimates for the half Klein-Gordon equation on asymptotically flat backgrounds and applications to cubic Dirac equations}

	\author[S. Herr]{Sebastian Herr}
	\address{Universit\"at Bielefeld, Fakult\"at f\"ur Mathematik, Postfach 100131, 33501 Bielefeld, Germany}
	\email{herr@math.uni-bielefeld.de}
	
    \author[S. Hong]{Seokchang Hong}
    \address{Universit\"at Bielefeld, Fakult\"at f\"ur Mathematik, Postfach 100131, 33501 Bielefeld, Germany}
    \email{shong@math.uni-bielefeld.de}

	
	\begin{abstract}
		The aim of this paper is to establish the $L^2_t$-endpoint Strichartz estimate for (half) Klein-Gordon equations on a weakly asymptotically flat space-time. As an application we prove small data global well-posedness and scattering for massive cubic Dirac equations in the full subcritical range in this setting.
    
        Crucial ingredient is a parametrix contruction following the work of Metcalfe-Tataru and Xue and complements Strichartz estimates obtained by Zheng-Zhang. The proof of the global result for the cubic Dirac equation follows the strategy developed by Machihara-Nakanishi-Ozawa in the Euclidean setting.
	\end{abstract}

		\maketitle


\section{Introduction}\label{sec:intro}
\subsection{Dispersive and Strichartz estimates for wave and Klein-Gordon equations}
Consider the classical constant coefficient wave equation on $\mathbb R^{1+d}$, i.e.
\begin{align}
    \Box u = (\partial_t^2-\Delta_x)u= 0, \ u|_{t=0}=u_0, \ \partial_t u|_{t=0}=u_1,
\end{align}
where $\Delta_x=\sum_{j=1}^d\frac{\partial^2}{\partial x_j^2}$ is the usual Laplacian operator on $\mathbb R^d$, $d\geq 2$.
For initial data Fourier-supported in the annulus $\{ \frac\lambda2\le |\xi|\le2\lambda \}$, the solution $u$ satisfies the dispersive estimate
\begin{align}\label{dispersive-wave}
    \| \nabla_x u(t)\|_{L^\infty_x} \lesssim t^{-\frac{d-1}{2}} \lambda^{\frac{d+1}{2}} ( \|\nabla_{x}u_0\|_{L^1_x}+\|u_1\|_{L^1_x}),
\end{align}
which follows from the classical stationary phase analysis.
From the dispersive decay of waves, the $L^2$ conservation and Littlewood--Paley theory one obtains the family of Strichartz estimates \cite{strichartz,keeltao} 
\begin{align}\label{strichartz-wave}
    \||\nabla_x|^{-\sigma}\nabla_{x}u\|_{L^p_tL^q_x(\mathbb R^{1+d})} & \lesssim \|\nabla_{x}u_0\|_{L^2_x}+\|u_1\|_{L^2_x},
\end{align}
for wave-admissible $(\sigma,p,q)$ satisfying $2\le p,q\le\infty$ and
\begin{align}
    \sigma = \frac{d}{2}-\frac{d}{q}-\frac1p, \quad \frac2p +\frac{d-1}{q} \le \frac{d-1}{2},
\end{align}
except for the forbidden endpoint $(1,2,\infty)$ when $d=3$, in which the estimate fails to hold \cite{mg}.

This has been generalized to variable coefficient wave equations under a number of different assumptions. For instance, the dispersive  and Strichartz estimates \eqref{strichartz-wave} have been established in \cite{metataru}  under a weak asymptotic flatness condition.

Similarly, consider the  constant coefficient  Klein-Gordon equation on $\mathbb R^{1+d}$
\begin{align}
   ( \Box+m^2) u =0, \ u|_{t=0}= u_0, \ \partial_t u |_{t=0}= u_1,
\end{align}
for a mass parameter $m>0$. Under the same Fourier support assumption as above, the solutions $u$ to the Klein-Gordon equations satisfy the following dispersive decay estimate (see e.g.\ \cite{nakaschlag}):
\begin{align}
    \|\nabla_{x} u\|_{L^\infty_x} & \lesssim \langle \lambda\rangle^{\frac{d}{2}+1}t^{-\frac{d}{2}}
    ( \|\nabla_{x} u_0\|_{L^1_x}+\|u_1\|_{L^1_x}).
\end{align}
Notice that the time-decay rate is stronger than for the wave equation, at the expense of derivative loss. In addition, \eqref{dispersive-wave} is true for large frequencies, so that by interpolation one obtains the Strichartz estimates of the form
\begin{align}
\|\langle\nabla_x\rangle^{-\sigma}\nabla_{x}u\|_{L^p_tL^q_x(\mathbb R^{1+d})} & \lesssim \|\nabla_{x}u_0\|_{L^2_x}+\|u_1\|_{L^2_x},
\end{align}
where $(\sigma,p,q)$ satisfy the relation: for $2\le p,q\le\infty$ and $0\le\theta\le1$,
\begin{align}
  \sigma= \frac{d+1+\theta}{4}\left( 1-\frac2q  \right), \quad  \frac2p+\frac{d-1+\theta}{q} = \frac{d-1+\theta}{2},
\end{align}
except for the forbidden endpoint $(p,q)=(2,\infty)$ when $d=3$ and $\theta=0$, see e.g.\ \cite[Lemma 3]{machihara1}.

In this paper, we establish the corresponding Strichartz estimates for the (half) Klein-Gordon equation  
with variable coefficients under a weak asymptotic flatness condition.
This complements the result of \cite{metataru} for the wave equation in this setting.

Given a Lorentzian metric $g=-dt^2+g_{ij}(t,x)dx^idx^j$, we consider the Klein-Gordon equation $(\Box_g-m^2)u=0$. In view of the applications we have in mind it is favorable to consider half-Klein-Gordon equations
\begin{align}
    i\partial_tu = \pm \mathrm{Op}(\sqrt{m^2+g^{ij}(t,x)\xi_i\xi_j})u + (\textrm{lower order terms}).
\end{align}
We denote these lower-order terms by $A'(t,x,D_x)$ and assume a very general bound, which is flexible enough to cover many cases, in particular the application to Dirac equations.\\
\textbf{Assumption on lower order terms:} The lower-order terms $A'$ satisfy the following error-type estimates:
\begin{align}\label{est-lower-error}
    \|A'u\|_{Y^s} \lesssim \epsilon \|u\|_{X^s},
\end{align}
where $s$ refers to the $L^2$ Sobolev regularity of the initial data and the emanating solution.
We postpone precise definitions and further discussion to Section \ref{sec:dirac}.

From now on we put $m=1$ for simplicity and focus on the evolution equation
\begin{align}\label{eq-half-kg}
\begin{aligned}
     D_tu \pm  \mathrm{Op}(\sqrt{1+g^{ij}(t,x)\xi_i\xi_j})u+A'(t,x,D_x)u =f,\\
     u|_{t=0}:= u_0,
     \end{aligned}
\end{align}
where $A'$ is the lower-order term satisfying the error-type estimates \eqref{est-lower-error}.
We shall consider the half-Klein-Gordon operator, which is given by a small perturbation of the usual half-Klein-Gordon operator $D_t\pm \langle D_x\rangle$. 
To be precise, we consider the metric $g=-dt^2+g_{ij}(t,x)dx^idx^j$, where the elliptic part $[g_{ij}]_{i,j=1,\cdots,3}$ satisfies the following:
\[
C_{\alpha,k}:=\sup_{(t,x)\in A_j}\left( |x|^{|\alpha|}|\partial_{t,x}^\alpha ( g^{ij}(t,x)-\delta^{ij})|\right), \quad 0\le |\alpha|\le 2,
\]
where $\alpha=(\alpha_0,\alpha')\in \mathbb{N}_0^{1+3}$, and
\[
C_{\alpha,k}':= \sup_{(t,x)\in A_j}\left( |x|^{\frac{|\alpha|+1}{2}}|\partial_{t,x}^\alpha g^{ij}(t,x)| \right), \quad 3\le|\alpha|\le[\frac d2]+3, \ 0\le \alpha_0\le2,
\]
satisfy
\begin{align}\label{dirac-asymp-flat1}
    \sum_{k\in\mathbb Z} C_{\alpha,k} \le \epsilon, \quad |\alpha|\le2,
\end{align}
and
\begin{align}\label{dirac-asymp-flat2}
   \sum_{k\in\mathbb Z}
   C_{\alpha,k}' \le \epsilon, \quad 3\le|\alpha|\le [\frac d2]+3, \ 0\le\alpha_0\le2.
\end{align}
In addition, we assume that the metric and its derivatives are regular near $0$, i.e., if $B=\{|x|<1\}$ is the unit ball, then we have
\begin{align}\label{dirac-asymp-regular}
    \sup_{(t,x)\in\mathbb R\times B}|x|^{|\alpha|}|\partial_{t,x}^\alpha ( g^{ij}(t,x)-\delta^{ij})| \le C_\alpha, \quad 0\le|\alpha|\le[\frac d2]+3, \ 0\le \alpha_0\le2.
\end{align}

Now we state our first main theorem.
\begin{thm}[Strichartz estimates for the half Klein-Gordon equations]\label{thm-strichartz}
Let $d\ge3$, $|s|<\frac{d+1}{2}$, and \eqref{est-lower-error}, \eqref{dirac-asymp-flat1}, \eqref{dirac-asymp-flat2}, \eqref{dirac-asymp-regular}. Then, the solution $u$ to \eqref{eq-half-kg}
satisfies the Strichartz estimates
	\begin{align}
	\|\langle D_x\rangle^{s-\sigma_1}u\|_{L^{p_1}_t L^{q_1}_x(\mathbb R^{1+d})} & \lesssim \|u_0\|_{H^s_x(\mathbb R^d)}+ \|\langle D_x\rangle^{s+\sigma_2} f\|_{L^{p_2'}_tL^{q_2'}_x(\mathbb R^{1+d})}, 	
	\end{align}
where the tuples $(\sigma_1,p_1,q_1)$ and $(\sigma_2,p_2,q_2)$ satisfy the relations: 
\begin{align}
&\frac{2}{p}+\frac{d-1+\theta}{q} = \frac{d-1+\theta}{2}, \ \sigma=\sigma(q)=\left( \frac{d+1+\theta}{4} \right)\left( 1-\frac2q \right),	\ 0\le\theta\le1,\\
& (p,q)\ne(2,\infty) \text{ if } d=3, \, \theta=0.
\end{align}
\end{thm}
Note that Theorem \ref{thm-strichartz} includes, if $d=3$ and $\theta>0$, the endpoint case $p=2$ and $q=2+\frac4{\theta}$. 
The upper bound on $|s|$ is technical and related to the function spaces for the localised energy estimates. 

The strategy of the proof of Theorem \ref{thm-strichartz} mainly follows \cite{metataru}. 
We construct an outgoing parametrix for the frequency-localised operators via the FBI transform introduced in \cite{metataru,tataru4}. Thanks to the assumptions \eqref{dirac-asymp-flat1}, \eqref{dirac-asymp-flat2}, and \eqref{dirac-asymp-regular}, the principal results from \cite{metataru,tataru4} can be used here. However, there the Klein-Gordon operator does not obey the scaling symmetry. This is the reason why we have to track the dependence on the parameter $\lambda$, which is different from \cite{metataru,tataru4}.
As \cite{metataru,tataru4}, in addition to the parametrix we also provide error estimates via localised energy estimates.


\subsubsection*{Previous results}
The Strichartz estimates for the wave equations with the constant coefficient have been extensively studied. We refer to \cite{ginibrevelo1} for an expository study. 
The forbidden endpoint $(p,q)=(2,\infty)$ with $d=3$ can be compensated at the expense of an extra regularity in the angular variables and one can enjoy the stronger space-time estimates by imposing the angular regularity. We refer to \cite{sterbenz} for the improved Strichartz estimates of the wave equations using angular regularity and \cite{cholee} for more general operators.

Local-in-time Strichartz estimates for the wave operator with variable coefficients have been well-studied. Concerning a parametrix construction, one may consider a suitable localisation of the spatial variables and the frequency variables at the same time. One way of accomplishing this task is to utilise a wave packet. Smith \cite{smith} construct a wave packet parametrix and obtains the Strichartz estimates with $C^{1,1}$-coefficients in $d=2,3$. An alternative way is to use the FBI transform, which is a phase space transform. Tataru \cite{tataru,tataru1,tataru2} proves the Strichartz estimates with rough coefficients. We also refer to the work by Smith - Tataru \cite{smithtataru} for the sharp lower bounds of the solutions to the wave equation with low regularity metric. 

Concerning global-in-time Strichartz estimates, several geometric features of manifolds are also considered. For example, Smith-Sogge \cite{smithsogge} establishes the Strichartz estimates on a spatially compact perturbation of the Euclidean space with odd dimensions. This is later improved to even dimensions by Burq \cite{burq} and Metcalfe \cite{metcalfe} independently. The Strichartz estimates for the Schr\"odinger equations and the wave equations are obtained on an asymptotically Euclidean space in a very weak sense by Tataru \cite{tataru4} and Metcalfe-Tataru \cite{metataru}, using a time-dependent FBI transform. On the non-trapping scattering manifold setup, Hassell-Zhang \cite{hazhang} show the endpoint Schr\"odinger-Strichartz estimates. Then the Strichartz estimates for the wave equations and the Klein-Gordon equations are obtained by Zhang \cite{zhang}. 
The Strichartz estimates for the wave equations on manifold with boundary have been obtained by Blair - Smith - Sogge \cite{blairsmithsogge} and it has been improved to more general setting by Ivanovici - Lebeau - Planchon \cite{ivanovici1} and Ivanovici - Lascar - Lebeau - Planchon \cite{ivanovici2}.

On the study of solutions to the linear Klein-Gordon equations, see \cite{katooz,klai,marshall} and the references therein. Concerning global existence of solutions to nonlinear Klein-Gordon equations, we refer to \cite{delortfang,delort,ginibrevelo,klai1,shatah}. The endpoint Strichartz estimates for the Klein-Gordon equation on the non-trapping scattering manifold  have been obtained by Zhang-Zheng \cite{zhangzheng}. We cannot use this result (even under strong enough assumptions so that we are in the geometric setting of \cite{zhangzheng}) to derive global in time estimates for the cubic Dirac equation due to the presence of lower order perturbations.

We point out that for the frequency localisation to unit scale $|\xi|\approx1$  and $\lambda=1$ the dispersive decay estimate (see \eqref{pointwise-decay}) has been obtained in the thesis by Xue \cite{xue}. Here, we extend this result to arbitrary frequency scales with the optimal derivative loss in \eqref{pointwise-decay} and prove the endpoint Strichartz estimates in Theorem \ref{thm-strichartz}.
\
\subsection{Scattering of the cubic Dirac equations on a curved space-time}
As an application of the endpoint Strichartz estimates Theorem \ref{thm-strichartz} we establish the global well-posedness and scattering of the massive cubic Dirac equation on a $(1+3)$-dimensional asymptotically flat background.

We first recall the linear Dirac equation with massive case $M>0$ on the Minkowski spacetime $(\mathbb R^{1+3},\mathbf m)$, with the metric $\mathbf m=\textrm{diag}(-1,+1,+1,+1)$. The linear Dirac equation \cite{D} is given by
\begin{align*}
	-i\gamma^\mu\partial_\mu\psi+M\psi = 0,
\end{align*}
where $\psi$ is the unknown spinor field $\psi:\mathbb R^{1+3}\rightarrow \mathbb C^4$ and the gamma matrices $\gamma^\mu$ are the constant complex matrices satisfying the following algebraic relations:
\begin{align}
\gamma^\mu\gamma^\nu+\gamma^\nu\gamma^\mu = -2m^{\mu\nu}\mathbf I_{4\times4}.	
\end{align}
A typical argument of studying the Cauchy problems of the Dirac equations is to reformulate the equations as a system of the half-wave for the Klein-Gordon equations. To do this, we first multiply $\gamma^0$ on the equation to get
\begin{align*}
	(-i\partial_t -i\gamma^0\gamma^j\partial_j +M\gamma^0 )\psi  =0.
\end{align*}
Then our desired reformulation can be performed provided that there exists a projection operator $\Pi_\pm^M$ so that it satisfies $\Pi_\pm^M\Pi_\pm^M=\Pi_\pm^M$ and $\Pi_\pm^M\Pi_\mp^M=0$ and the following identity holds:
\begin{align}\label{flat-dirac-prop}
\langle\nabla_x\rangle_M(\Pi_+^M - \Pi_-^M) = -i\gamma^0\gamma^j\partial_j+M\gamma^0. 
\end{align}
Now we introduce the operators, which are defined as the Fourier multipliers
\begin{align}\label{dirac-proj-flat}
	\Pi_\pm^M(\xi) = \frac12\left( \mathbf I_{4\times4} \pm \frac{\xi_j\gamma^0\gamma^j+M\gamma^0}{\langle\xi\rangle_M} \right).
\end{align}
By a straightforward computation, one can deduce that $\Pi_\pm^M\Pi_\pm^M=\Pi_\pm^M$ and $\Pi_\pm^M\Pi_\mp^M=0$, and also the relation \eqref{flat-dirac-prop}. In consequence, the homogeneous Dirac equation is reformulated as the half Klein-Gordon equations as follows:
\begin{align*}
    (-i\partial_t\pm\langle\nabla_x\rangle_M)\Pi^M_\pm \psi = 0,
\end{align*}
where $\widehat{\Pi^M_\pm\psi}(\xi)=\Pi^M_\pm(\xi)\widehat{\psi}(\xi)$ and $\psi=\Pi^M_+\psi+\Pi^M_-\psi$.

A representative toy model for the study of nonlinear Dirac equations is the Dirac equation with a cubic nonlinearity which has the form:
\begin{align}\label{cubic-dirac-flat}
    (-i\gamma^\mu\partial_\mu+M)\psi = (\psi^\dagger\gamma^0 C\psi)C\psi,
\end{align}
where $C\in\mathbb C^4\times\mathbb C^4$ is some complex matrix. For example, we call the equation \eqref{cubic-dirac-flat} the Soler model \cite{soler} if $C=I_{4\times4}$ and the Thirring model \cite{thirring} if $C=\gamma^\mu$.

Solutions to the cubic Dirac equation with $M=0$ can be rescaled and the invariant Sobolev space $\dot{H^1}(\mathbb R^3)$ is called the critical Sobolev space. Also if $M>0$ we call $H^1(\mathbb R^3)$ the critical space.

Our main concern is to investigate the initial value problems for the Dirac equations on a curved background. Now the Minkowski metric $\mathbf m$ is replaced by a generally curved metric $g$. For simplicity we assume that the metric $g$ is decoupled, i.e., the metric $g$ is written by
\begin{align*}
	g = -dt^2+g_{ij}(t,x)dx^idx^j.
\end{align*}
On a curved space-time, the gamma matrices $\gamma^\mu$ are no longer constant matrices. The algebraic relation of the matrices $\gamma^\mu$ is now
\begin{align*}
	\gamma^\mu\gamma^\nu+\gamma^\nu\gamma^\mu = -2g^{\mu\nu}(t,x)\mathbf I_{4\times4}.
\end{align*}
The homogeneous covariant Dirac equation is given by
\begin{align*}
	(-i\gamma^\mu(t,x)\mathbf D_\mu+M)\psi = 0,
\end{align*}
where the covariant derivative $\mathbf D_\mu$ acting on a spinor field is defined by
\begin{align*}
	\mathbf D_\mu = \partial_\mu-\Gamma_\mu.
\end{align*}
The spinorial affine connections $\Gamma_\mu$ are matrices satisfying bounds $|\Gamma_\mu|\lesssim |\partial_{t,x}g_{ij}|$. We refer to Section \ref{sec:dirac} for the details.
We consider the Cauchy problem for the massive Dirac equation with a cubic-type nonlinearity:
\begin{align*}
	(-i\gamma^\mu\mathbf D_\mu+M)\psi ={}& (\psi^\dagger\psi)\gamma^0\psi, \\
	\psi|_{t=0}={}& \psi_0,
\end{align*}
where the metric $g$ satisfies the following asymtotically flatness assumption:
\begin{align}\label{assump-dirac-metric}
    | \partial_t^{\ell} \partial_x^\alpha (g^{ij}(t,x)-\delta^{ij})| \le C_\alpha \epsilon \langle x\rangle^{-1-|\alpha|-\ell}, \quad 0\le |\alpha|\le 4, \ 0\le \ell\le2.
\end{align}
Now we state our second main theorem.
\begin{thm}[Global well-posedness for the cubic Dirac equations]\label{thm-dirac-gwp}
Let $s>1$ and $M>0$ be given. Suppose that the metric $g$ satisfies \eqref{assump-dirac-metric} for $\epsilon>0$  sufficiently small. Then, the Cauchy problem for the cubic Dirac equation on the spacetime $(\mathbb{R}^{1+3}, g)$ is globally well-posed for small initial data $\psi_0\in H^s(\mathbb R^3)$. Furthermore, the solutions scatter to free solutions for $t\rightarrow \pm\infty$.
\end{thm}
\subsubsection*{Previous results}
Dirac equations on curved space-time have been studied by a number of authors. We refer the readers to \cite{parker} for an expository literature of the relativistic quantum field theory. The spectrum of the Dirac operator is studied by B\"ar \cite{bar}. Recently, the dispersive properties of the linear Dirac operators on a curved backgrounds have been studied.  Cacciafesta-Suzzoni \cite{cacciasu1} establish local smoothing estimates for Dirac equations on asymptotically Euclidean space. Then local-in-time and global-in-time Strichartz estimates for the Dirac equations are established by Cacciafesta, Suzzoni, Ben-Artzi, and Zhang \cite{cacciasu,artzcaccia} on a spherically symmetric manifold and in \cite{dancona} on the cosmic string background. In \cite{caccia1} the authors obtain the Strichartz estimates excluding the endpoint $p=2$ under some nontrapping condition using the work of Zhang-Zheng \cite{zhangzheng}. More recently, local-in-time Strichartz estimates for Dirac equations as well half wave and half Klein-Gordon equations are obtained on compact manifolds without boundary \cite{caccia2}.

A key approach of studying the dispersive properties of the Dirac operator in the previous results is to utilise the so-called squaring trick. In other words, by squaring the Dirac operator, one obtains the Klein-Gordon operator with a geometric quantities and study the Klein-Gordon equation which is second order operator. However, we do not follow this strategy since it leads to a derivative nonlinearity. Instead,  we reformulate the Dirac equation on a curved background into a variable coefficient half-Klein-Gordon equation.

We refer the readers to \cite{escobedovega} for the comprehensive study of nonlinear Dirac equations.
Global estimates for the cubic nonlinearity call for $L^2_tL^\infty_x$-estimates, which is the forbidden endpoint. However, in the massive case $M>0$, one can exploit the better time-decay $t^{-1-\frac\theta2}$ for any $0\le\theta\le1$ and establish global well-posedness for small initial data in $H^s(\mathbb R^3)$, $s>1$.
In this way, for any choice of $C$, global well-posedness and scattering results have been obtained using the endpoint Strichartz estimate, which is due to Machihara - Nakamura - Nakanishi - Ozawa \cite{machihara}.
In more specific settings, e.g. when $C$ is the identity matrix $I_{4\times4}$, the nonlinearity exhibits null structure  and scattering results for a small initial data in the critical Sobolev spaces $H^1(\mathbb R^3)$ and $H^{\frac12}(\mathbb{R}^2)$ have been established by Bejenaru - Herr \cite{behe,behe1}. The analogous results for the massless case $M=0$ are obtained by Bournaveas - Candy \cite{boucan}. 

Once we have established the global $L^2_t$-endpoint Strichartz estimate for the half-Klein-Gordon equation the proof of Theorem \ref{thm-dirac-gwp} follows immediately, similarly to \cite{machihara,machihara1}. 

We would like to also mention that in the proof of Theorem \ref{thm-dirac-gwp} we do not use the precise structure of the cubic nonlinearity.
In fact, any cubic half-Klein-Gordon equation can be handled, e.g.
\begin{align}
    (D_t\pm \mathrm{Op}(\sqrt{1+g^{ij}(t,x)\xi_i\xi_j}))u = Q(u)
\end{align}
for any cubic term $Q(u)$. 

\subsection*{Organisation}
The rest of this paper is organised as follows. We end this section by introducing notation.

Section \ref{sec:dirac} is devoted to the introduction of the Dirac operator on curved space-time. The most important part of Section \ref{sec:dirac} is to define projection-type operators in order to reformulate the Dirac equation as half-Klein-Gordon equations up to lower-order terms. We also show that these lower order terms turn out to satisfy the error-type estimates \eqref{est-lower-error} provided that the metric satisfies \eqref{assump-dirac-metric}. Hence the study of the Dirac equation is reduced to a more general class of the half-Klein-Gordon equations.

Section \ref{sec:local-energy-estimate} contains the most important setup of the reduction via the localisation of space and frequency simultaneously.
Beginning with the original half-Klein-Gordon operator $D_t+A^\pm$, we mollify
\begin{align*}
    q(t,x,\xi) = \sqrt{1+g^{ij}\xi_i\xi_j }-\sqrt{1+|\xi|^2},
\end{align*}
by truncation and regularisation of the coefficients,
and we obtain a family of mollified operators $A^\pm_{(k)}$, which obey the frequency-localisation property up to an acceptable error. From this we define a globally mollified operator
\begin{align*}
    \widetilde A^\pm = \sum_{k=-\infty}^{\infty}A^{\pm}_{(k)}S_k,
\end{align*}
where $S_k$ is the projection onto the frequency annulus of size $2^k$. Proposition \ref{prop-mollified-operator} allows us to replace the original $A^\pm$ by $\widetilde{A}^\pm$ at the expense of a small error. We also apply a scaling argument and analyse the scaled operator $\mathcal A^\pm_{(k)}$, which is localised at unit frequency (for each $k$).

In Section \ref{sec:local-energy-estimate}, assuming that we have constructed an outgoing parametrix for the scaled operator $D_t+\mathcal A^\pm_{(k)}$ (which is postponed to Section \ref{sec:para})  satisfying the dispersive estimate, we obtain an outgoing parametrix for the operator $D_t+A^\pm_{(k)}$. Hence, we obtain a parametrix for the globally mollified operator $D_t+ \widetilde A^\pm$ satisfying endpoint Strichartz estimates, which immediately imply the Strichartz estimates for the original half-Klein-Gordon operator due to Proposition \ref{prop-mollified-operator}. Then, we establish small data global well-posedness and scattering for the cubic Dirac equation as a direct application.

In Section \ref{sec:micro}, we recall the phase space transform and  microlocal analysis from \cite{metataru}.

Finally, Section \ref{sec:para} is devoted to the construction of an outgoing parametrix for the scaled operators $D_t+\mathcal A^\pm_{(k)}$. The main novelty of Section \ref{sec:para} is the dispersive inequality with time decay $t^{-\frac d2}$ with the sharp loss of regularity.





\subsection*{Notation}\label{sec:not}

The $L^2-$inner product is given by $\langle\cdot,\cdot\rangle_{L^2}$ so that 
\begin{align*}
\langle f,g\rangle_{L^2_x} = \int_{\mathbb R^d} f(x)\overline{g(x)}\,dx .	
\end{align*}
For complex-valued functions $f,g\in L^2$, we have $\langle f,g\rangle_{L^2} = \overline{\langle g,f\rangle}_{L^2}$.

We write $(D_t,D_x)= (\frac1i\partial_t,\frac1i\partial_x)$. For any multi-index $\alpha\in\mathbb N^d_0$ with $\alpha=(\alpha_1,\cdots,\alpha_d)$, by the notation $\partial_x^\alpha$ we mean $\partial_x^\alpha=\partial_{x_1}^{\alpha_1}\cdots\partial_{x_d}^{\alpha_d}$, where $x=(x_1,\cdots,x_d)\in\mathbb R^d$. We also set $x_0=t$. The partial derivative $\partial_\xi^\beta$ is defined in the similar way.

We adapt the Einstein summation convention. Repeated Greek indices $\mu,\nu,\cdots$ mean summation from $0$ to $d$, and repeated Latin indices $i,j,\cdots$, mean summation from $1$ to $d$. For example, on $\mathbb R^{1+3}$, the Dirac operator $-i\gamma^\mu\mathbf D_\mu=-i\gamma^0\mathbf D_0-i\gamma^j\mathbf D_j = -i\gamma^0\mathbf D_0-i\sum_{j=1}^3\gamma^j\mathbf D_j$.

We usually denote dyadic numbers $\lambda,\mu,\nu\in2^{\mathbb Z}$.

For two positive numbers $A$ and $B$, we write $A\lesssim B$ if $A\le CB$, for some absolute constant $C$ (which only depends on irrelevant parameters). If $C$ can be chosen sufficiently small, we write $A\ll B$.
Also, we write $A\approx B$ if both $A\lesssim B$ and $B\lesssim A$.

The mixed Lebesgue space $L^p_tL^q_x(\mathbb R^{1+d})$ for $1\le p,q\le\infty$ is the set of all equivalence classes of measurable functions $u$ given by $t\mapsto u(t)\in L^q_x(\mathbb R^d)$, and the norm is given by
\begin{align*}
	\|u\|_{L^p_tL^q_x(\mathbb R^{1+d})}^p = \int_{\mathbb R} \left( \int_{\mathbb R^d}|u(t,x)|^q\,dx \right)^{\frac{p}{q}}\,dt,
\end{align*}
with the obvious modification if $p,q=\infty$.
We set $\|u\|_{\langle D_x\rangle^s L^p_tL^q_x}=\|\langle D_x\rangle^{-s}u\|_{L^p_tL^q_x}$ for $s\in\mathbb R$.

Given a small $\epsilon>0$ as in \eqref{dirac-asymp-flat1}, we can find a sequence $\{\epsilon_j\}_{j\in\mathbb Z}\in\ell^1$ such that
\begin{align*}
	\sup_{(t,x)\in A_j}\left( |x|^{|\alpha|}|\partial_{t,x}^\alpha ( g^{ij}(t,x)-\delta^{ij})|\right)\le \epsilon_j
\end{align*}
and $\sum_{j}\epsilon_j\lesssim\epsilon$. Without loss of any generality, we may assume that the sequence $\epsilon_j$ is slowly varying, which says
\begin{align*}
	|\log\epsilon_j-\log\epsilon_{j-1}| \le 2^{-10}.
\end{align*}
We can also choose a function $\epsilon:\mathbb R^+\rightarrow\mathbb R^+$ satisfying 
\begin{align}\label{ep-property}
\begin{aligned}	
	\epsilon_j < \epsilon(\mathfrak s) < 2\epsilon_j, &\quad 2^j < \mathfrak s<2^{j+1}., \\
	|\epsilon'(\mathfrak s)| \le &2^{-5}\mathfrak s^{-1}\epsilon(\mathfrak s),
	\end{aligned}
\end{align}
which implies that
\begin{align}\label{prop-eps1}
	\int_0^\infty \frac{\epsilon(\mathfrak s)}{\mathfrak s}\,d\mathfrak s \approx \epsilon.
\end{align}
We also define the function $\epsilon_k:\mathbb R^+\rightarrow\mathbb R^+$ for $j\in\mathbb Z$ such that
\begin{align*}
	\epsilon_k(\mathfrak s) \approx \epsilon_j, \ \mathfrak s\approx 2^j, \ j\ge -k, \\
	\epsilon_k(\mathfrak s) \approx \epsilon_{-k}, \ s\le 2^{-k}.
\end{align*}
We consider a frequency Littlewood-Paley decomposition
\begin{align*}
    1= \sum_{j=-\infty}^{\infty}S_j(D_x)
\end{align*}
where 
\begin{align*}
    \textrm{supp }s_j\subset \{2^{j-1}<|\xi|<2^{j+1}\}.
\end{align*}
We also consider a smooth spatial Littlewood-Paley decomposition:
\begin{align*}
    1= \sum_{j=-\infty}^\infty \chi_j(x), \quad \mathrm{supp}\,\chi \subset \{ x\in \mathbb R^d :  2^{j-1} < |x| < 2^{j+1} \}.
\end{align*}
We also set $\chi_{<j} = \sum_{k<j} \chi_k $.
\section{Dirac operators}\label{sec:dirac}
This section is devoted to the introduction of covariant Dirac operators on a curved space-times, see \cite{parker}. We consider the Cauchy problems for the Dirac equation with a cubic nonlinearity on a weak asymptotically flat space-time. We reformulate the Dirac equations in terms of the half-Klein-Gordon equations by using a projection-type operator and then restate Theorem \ref{thm-dirac-gwp} in terms of the cubic half-Klein-Gordon equations.

In what follows, we restrict ourselves into a curved metric of the form:
\begin{align*}
	g = -dt^2+ g_{ij}(t,x)dx^idx^j.
\end{align*}
In other words, the metric is decoupled. 
The homogeneous covariant Dirac equation is given by
\begin{align}\label{eq-dirac-cov}
	(-i\gamma^\mu\mathbf D_\mu+M)\psi = 0.
\end{align}
On a generally curved space-time, the gamma matrices $\gamma^\mu$, $\mu=0,\cdots,3$ are no longer constant matrices. Instead, the matrices $\gamma^\mu=\gamma^\mu(t,x)$ also become space-time dependent complex matrices. The algebraic property is then generalised to 
\begin{align}\label{gamma-alge-curv}
\gamma^\mu\gamma^\nu+\gamma^\nu\gamma^\mu = -2g^{\mu\nu}(t,x)\mathbf I_{4\times4}.	
\end{align}
We define the covariant derivative acting on a spinor field $\psi:\mathbb R^{1+3}\rightarrow\mathbb C^4$ to be
\begin{align}
\mathbf D_\mu \psi = (\partial_\mu-\Gamma_\mu)\psi,	
\end{align}
where the spinorial affine connections $\Gamma_\mu=\Gamma_\mu(t,x)$ are matrices defined by the vanishing of the covariant derivative of the gamma matrices:
\begin{align}\label{affine-spin}
	\mathbf D_\mu \gamma_\nu = \partial_\mu\gamma_\nu-\Gamma^\lambda_{\mu\nu}\gamma_\lambda-\Gamma_\mu\gamma_\nu+\gamma_\nu\Gamma_\mu = 0.
\end{align}
One can represent the $\gamma^\mu(t,x)$ matrices in terms of the gamma matrices $\tilde{\gamma}^\mu$ on the flat spacetime by introducing a \textit{vierbein} $b^\alpha_\mu(t,x)$ of vector fields, defined by
\begin{align}\label{metric-vierbein}
	g_{\mu\nu}(t,x) = m_{\alpha\beta}b^\alpha_\mu(t,x)b^\beta_\nu(t,x).
\end{align}
Then the matrices $\gamma^\mu(t,x)$ can be written in the form
\begin{align}
\gamma^\mu(t,x) = b^\mu_\alpha(t,x)\tilde{\gamma}^\alpha.	
\end{align}
Here we used the notation $\gamma^\mu(t,x)$ and $\tilde\gamma^\mu$ to distinguish the gamma matrices dependent on space-time from the constant matrices.
Now a set of $\Gamma_\mu(t,x)$ satisfying \eqref{affine-spin} are given by
\begin{align}
\Gamma_\mu(t,x) = -\frac14 \tilde{\gamma}_\alpha\tilde{\gamma}_\beta b^{\alpha\lambda}(t,x)\mathbf D_\mu b^\beta_\lambda(t,x),	
\end{align}
where
\begin{align*}
	\mathbf D_\mu b^\beta_\lambda = \partial_\mu b^\beta_\lambda - \Gamma^\sigma_{\mu\lambda}b^\beta_\sigma.
\end{align*}
We refer the readers to \cite[Section 3.9]{parker} for the details. We have an obvious inequality $|\Gamma_\mu(t,x)|\lesssim |\partial_{t,x}g_{ij}|$. See also Proposition 2.1 of \cite{caccia1}.

Motivated by the study of Dirac equations on the flat space-time,
we would like to rewrite the Dirac equation in terms of the half Klein-Gordon equation. 
To do this, as one has done in the flat case, we multiply $\gamma^0$ on the equation \eqref{eq-dirac-cov} and rewrite the equation as
\begin{align*}
	(-i\partial_t-i\gamma^0\gamma^j\mathbf D_j-i\Gamma_0+M\gamma^0)\psi = 0.
\end{align*}

To reformulate the above equation in terms of the half-Klein-Gordon equations, we introduce the symbols $\pi^M_\pm$ given by
\begin{align}
    \begin{aligned}
        \pi_\pm^M (t,x,\xi) := \frac12 \left( I_{4\times4}\pm \frac{\gamma^0\gamma^j\xi_j+M\gamma^0}{\sqrt{M^2+g^{ij}(t,x)\xi_i\xi_j }} \right).
    \end{aligned}
\end{align}
A straightforward computation gives 
\begin{align}\label{eq-dirac-proj-sym}
    (\pi_\pm^M)^2 = \pi_\pm^M, \quad \pi_\pm^M \pi_\mp^M = 0, \quad \pi_+^M+ \pi_-^M = I_{4\times4}.
\end{align}
Before the reformulation, we give a brief review of several facts of pseudodifferential operators. We refer the readers to the literature \cite[Section 18]{hoermander}.
\begin{prop}[Theorem 18.1.8 of \cite{hoermander}]
	If $a$ and $b$ are symbols, then formally 
	\begin{align*}
		a(x,D_x)b(x,D_x) = (a\circ b)(x,D_x), 
	\end{align*}
	and has the asymptotic expansion 
	\begin{align*}
		(a\circ b)(x,\xi) = \sum_j \frac{(i\langle\partial_y,\partial_\eta\rangle)^j}{j!} a(x,\eta)b(y,\xi)\bigg|_{\eta=\xi,y=x}.
	\end{align*}
\end{prop}
In particular, we have the first-order and second-order symbol expansions (see (B-1) of \cite{gajaotataru}.)
\begin{align*}
	a\circ b = ab +\frac1i\int_0^1 r_{1,s}\,ds = ab+\frac1i\frac{\partial a}{\partial\xi}\frac{\partial b}{\partial x}-\frac12 \int_0^1 r_{2,s}\,ds, 
\end{align*}
where the remainder is given by
\begin{align}\label{remainder-comp-sym}
	r_{j,s} (x,\xi) = e^{is\langle\partial_y,\partial_\eta\rangle}\langle\partial_y,\partial_\eta\rangle^j [a(x,\eta)b(y,\xi)]_{y=x,\eta=\xi}.
\end{align}
Using the above results, we deduce that the commutator of pseudodifferential operators can be written as
\begin{align*}
	[a, b](x,D_x) = \{a,b\}(x,D_x) + r(x,D_x),
\end{align*}
where $\{a,b\}$ is the usual Poisson bracket given by
\begin{align*}
	\{a,b\} = \frac{\partial a}{\partial\xi}\frac{\partial b}{\partial x} - \frac{\partial a}{\partial x}\frac{\partial b}{\partial\xi},
\end{align*}
and the remainder term $r(x,D_x)$ is given by \eqref{remainder-comp-sym}.
\begin{prop}[Theorem 18.1.13 of \cite{hoermander}]\label{prop-bdd-pdo}
Let $m\in\mathbb R$ and $a$ be a symbol of the class $S^m$. Then the operator $a(x,D_x)$ is a continuous operator from $H^s_x$ to $H^{s-m}_x$ for every $s\in\mathbb R$. Furthermore, $a(x,D_x)$ is a continuous operator from $H^{s,p}_x$ to $H^{s-m,p}_x$, for $1\le p \le\infty$.	
\end{prop}
It is remarkable that the norm of the operator $a(x,D_x):H^{s}_x\rightarrow H^{s-m}_x$ can be estimated by semi-norms of the symbol $a\in S^m$. See also the proof of Theorem 18.1.11 of \cite{hoermander}. We also recall the theorem by Calder\'on-Valliancourt:
\begin{prop}
If a symbol $a(x,\xi)$ has bounded derivatives $\partial_x^\alpha\partial_\xi^\beta a(x,\xi)$ for $|\alpha|+|\beta|\le N_d$ with some fixed number $N_d$ dependent on the dimension $d$, then the associated pseudodifferential operator $a(x,D_x)$ is $L^2$-bounded.	
\end{prop}
Furthermore, the number $N_d$ can be improved to $[\frac d2]+1$, see Cordes \cite{cordes}. Thus, for boundedness of the pseudodifferential operators, we need only a restricted number of derivatives of the symbols. In particular, if one considers the $L^2$-norm of the commutator of the pseudodifferential operators of the form $[A(x),B(D_x)]$, we only need to compute the derivatives $D_x^\alpha a(x) D_\xi^\alpha b(\xi)$ for $|\alpha|\le [\frac d2]+3$.

From now on, we set
\begin{align}
    \mathcal H(t,x,\xi):= \gamma^0\gamma^j\xi_j+M\gamma^0, \quad \omega_g(t,x,\xi) := \sqrt{M^2+g^{ij}(t,x)\xi_i\xi_j }.
\end{align}

\begin{lem}
    We have
    \begin{align*}
        \mathcal H^2 = \omega_g^2I_{4\times4}, \quad \mathcal H\pi_\pm^M = \pm\omega_g \pi_\pm^M.
    \end{align*}
\end{lem}
\begin{proof}
    The proof follows via a straightforward computation. Indeed, we have
    \begin{align*}
        \mathcal H^2 & = (\gamma^0\gamma^j\xi_j+M\gamma^0)^2  \\
        & = \gamma^0\gamma^j\gamma^0\gamma^k\xi_j\xi_k + M \gamma^0\gamma^j\gamma^0\xi_j+ M\gamma^0\gamma^0\gamma^j\xi_j + M \gamma^0\gamma^0.
    \end{align*}
    From the Clifford relation, we have
    \begin{align*}
        \gamma^j\gamma^k+\gamma^k\gamma^j = -2g^{jk}I_{4\times4}, \quad (\gamma^0)^2 = I_{4\times4}, \quad \gamma^0\gamma^j+\gamma^j\gamma^0 = 0,
    \end{align*}
    and hence we obtain
     \begin{align*}
        \mathcal H^2 
        & = (-\gamma^j\gamma^k\xi_j\xi_k + M^2) I_{4\times4} \\
        & = (g^{jk}\xi_k\xi_k + M^2)I_{4\times4} \\
        & = \omega_g^2 I_{4\times4}.
    \end{align*}
    Now we write $\pi_\pm^M= \frac12(I\pm\dfrac{\mathcal H}{\omega_g})  $. Then we see that
    \begin{align*}
        \mathcal H \pi_\pm^M & = \frac12 \left( \mathcal H\pm \frac{\mathcal H^2}{\omega_g} \right) = \frac12 \left( \mathcal H\pm \omega_g\right) = \pm\omega_g\frac12\left(I_{4\times4}\pm \frac{\mathcal H}{\omega_g}   \right),
    \end{align*}
    which completes the proof.
\end{proof}

Now we define the pseudodifferential operator from the symbols via the standard calculus:
\begin{align}
    \Pi_\pm^M(t,x,D_x) := \mathrm{Op}( \pi_\pm^M), \quad \Omega_g(t,x,D_x) := \mathrm{Op}(\omega_g).
\end{align}
The following is an immediate consequence of the identity \eqref{eq-dirac-proj-sym}:
\begin{equation}\label{eq-dirac-proj}
  \Pi_+^M+\Pi_-^M = I_{4\times4}, \quad   (\Pi_\pm^M)^2- \Pi_\pm^M \in OPS^{-1}, \quad \Pi_\pm^M\Pi_\mp^M \in OPS^{-1}.
\end{equation}
Now we reformulate the Dirac equation in terms of the half-Klein-Gordon equations.
We recall the Dirac equation:
\begin{align*}
    (-i\gamma^\mu\mathbf D_\mu +M)\psi = 0.
\end{align*}
After multiplying $\gamma^0$ on both sides, it can be also rewritten as follows:
\begin{align*}
    (-i\partial_t + i\gamma^0\gamma^j\partial_j +i\Gamma_0+i\gamma^0\gamma^j\Gamma_j+M\gamma^0)\psi = 0.
\end{align*}
Now we introduce the operator given by
\begin{align}
    H(t,x,D_x) := \gamma^0\gamma^jD_j +M\gamma^0 + i\Gamma_0+i\gamma^0\gamma^j\Gamma_j = \mathrm{Op}(\mathcal H) + B_0.
\end{align}
Then we have 
\begin{align}
    (D_t+H)\psi = 0.
\end{align}
On the other hand, we put $\psi_\pm := \Pi_\pm^M\psi $. Then
\begin{align*}
    (D_t\pm\Omega_g)\psi_\pm & = (D_t\pm \Omega_g)\Pi_\pm^M\psi \\
    & = D_t\Pi_\pm^M\psi \pm\Omega_g \Pi_\pm^M\psi \\
    & = \Pi_\pm^MD_t\psi +[ D_t, \Pi_\pm^M ] \psi \pm\Omega_g\Pi_\pm^M\psi \\
    & = \Pi_\pm^M(D_t +H - H )\psi +[ D_t, \Pi_\pm^M ] \psi \pm\Omega_g\Pi_\pm^M\psi \\
    & = \Pi_\pm^M (D_t+H)\psi + R_\pm\psi,
\end{align*}
where 
\begin{align}\label{dirac-remainder}
    R_\pm\psi := [D_t,\Pi_\pm^M] \psi \pm \Omega_g \Pi_\pm^M\psi - \Pi_\pm^M H\psi.
\end{align}
We conclude that the Dirac equation is transferred into the half-Klein-Gordon equation after the application of the projection operator $\Pi_\pm^M$, up to the remainder $R_\pm$. Now we need to show that the remainder term $R_\pm\psi$ is at most of class $OPS^{0}$. However, this is in fact obvious, since we already have $\omega_g\pi_\pm^M - \pi_\pm^M\mathcal H = 0 $ at the symbol level.

The following proposition ensures that the remainder term $R_\pm\psi$ is the lower-order term satisfying the error-type estimates \eqref{est-lower-error}, if the metric satisfies the following decay and regularity assumptions:

We define 
\begin{align*}
    \|u\|_{X_k} = \langle 2^k\rangle^{-\frac12} 2^k\|u\|_{L^2(A_{<-k})}+\langle 2^k\rangle^{-\frac12}2^{\frac k2}\sup_{j\ge -k}\| |x|^{-\frac12}u\|_{L^2(A_j)},
\end{align*}
where $A_j = \mathbb R\times\{ 2^{j-1}<|x|<2^{j+1}\}$, and $A_{<j}=\mathbb R\times\{ |x|<2^j \}$. We define the function space
\begin{align*}
    \|u\|_{X^s}^2 = \sum_{k=-\infty}^\infty \langle2^{k}\rangle^{2s}\|S_k u\|_{X_k}^2
\end{align*}
and $Y^s=(X^{-s})'$, i.e.\ the dual of $X^{-s}$. We have
\begin{align*}
    \|u\|_{Y^s}^2 = \sum_{k=-\infty}^\infty \langle2^{k}\rangle^{2s}\|S_k u\|_{X'_k}^2,
\end{align*}
with \begin{align*}
    \|u\|_{X'_k} = \langle 2^k\rangle^{\frac12} 2^{-k}\|u\|_{L^2(A_{<-k})}+\langle 2^k\rangle^{\frac12} 2^{-\frac k2}\sum_{j\ge -k}\| |x|^{\frac12}u\|_{L^2(A_j)}.
\end{align*}

\begin{prop}\label{prop-error-dirac}
Suppose that the metric $g$ satisfies \eqref{assump-dirac-metric}. Then the remainder terms $R_\pm\psi$
satisfy, for any $s\in\mathbb R$,
\[
\|R_\pm u\|_{ Y^s}
       \leq C_{s,M}\epsilon\|u\|_{ X^s}.
\]
The constant depends on finitely many fixed metric seminorms, not on the
time interval or the dyadic frequency.
\end{prop}
\begin{proof}
    We denote exact symbol composition by \(\#\), and set
\(C(a,b)=a\#b-ab\). Pointwise cancellation gives the exact symbol identity
\begin{equation}\label{eq:rsymbol}
 r_\pm=-i\partial_t\pi_\pm
       \pm C(\omega,\pi_\pm)-C(\pi_\pm,h)-\pi_\pm\# B_0.
\end{equation}
Its leading order-zero terms are
\[
 -i\partial_t\pi_\pm
 +\frac1i\sum_j\left(
 \pm\partial_{\xi_j}\omega\,\partial_{x_j}\pi_\pm
 -\partial_{\xi_j}\pi_\pm\,\partial_{x_j}\mathcal H\right)
 -\pi_\pm B_0.
\]
For example,
\[
 \partial_t\pi_\pm=O(\epsilon\langle x\rangle^{-2}),\quad
 \partial_x\pi_\pm=O(\epsilon\langle x\rangle^{-2}),\quad
 \partial_\xi\pi_\pm=O_M(\langle\xi\rangle^{-1}),
\]
\[
 \partial_\xi\omega=O_M(1),\qquad
 \partial_x\mathcal H=O(\epsilon\langle x\rangle^{-2}\langle\xi\rangle).
\]
In particular, we have $\partial_x^\alpha B_0 = O(\epsilon \langle x\rangle^{-2-|\alpha|}) $.
The integral composition remainder has the same weighted bounds:
each composition error has at least one spatial derivative of a metric
coefficient and further derivatives are controlled by \eqref{assump-dirac-metric}.
Integration by parts with the
weights \(\langle x\rangle\) and \(\langle\xi\rangle\), gives
\begin{equation}\label{eq:rbound}
 |\partial_x^\alpha\partial_\xi^\nu r_\pm|
 \lesssim_{\alpha,\nu,s,M}
 \varepsilon\langle x\rangle^{-2-|\alpha|}\langle\xi\rangle^{-|\nu|}.
\end{equation}
Since \(M>0\), there is no singularity at zero frequency.

Consequently weighted composition and the Calderon-Vaillancourt theorem give
\begin{equation}\label{eq:B}
 B_{\pm,s}(t):=\langle x\rangle \langle D_x\rangle^s R_\pm(t)\langle D_x\rangle^{-s}\langle x\rangle,
 \qquad \sup_t\|B_{\pm,s}(t)\|_{L^2_x\to L^2_x}
 \lesssim_{s,M}\epsilon.
\end{equation}
Both spatial factors exhaust the two powers of decay in
\eqref{eq:rbound}. The Sobolev conjugation preserves frequency order zero.

We claim the following:
For any \(s\in\mathbb R\), we have
\begin{equation}\label{eq:claim}
 \|\langle x\rangle^{-1}\langle D_x\rangle^su\|_{L^2_{t,x}}
 \lesssim_s\|u\|_{ X^s}.
\end{equation}
To prove this claim, let \(u_k\) be supported at spatial frequencies \(|\xi|\approx\lambda\) and put
 \(N_k=\|u_k\|_{X_k}\).

If \(\lambda\geq1\), then
\[
 N_k\approx\lambda^{1/2}\|u_k\|_{L^2(|x|<\lambda^{-1})}
 +\sup_{r\geq\lambda^{-1}}r^{-1/2}
                   \|u_k\|_{L^2(|x|\approx r)},
\]
where radii $r$ are dyadic. Thus
\begin{align*}
 \|\langle x\rangle^{-1}u_k\|_2^2
 &\lesssim\left(\lambda^{-1}
       +\sum_{r\geq\lambda^{-1}}\frac{r}{1+r^2}\right)N_k^2
 \lesssim N_k^2.
\end{align*}
Indeed the summands are bounded by \(r\) for \(r\leq1\) and
by \(r^{-1}\) for \(r\geq1\), and both dyadic sums converge.

If \(0<\lambda\leq1\), put \(r_0=\lambda^{-1}\). On the exterior region $|x|\ge r_0$,
\[
 \|\langle x\rangle^{-1}u_k\|_{L^2(|x|\geq r_0)}^2
 \lesssim\sum_{r\geq r_0}\frac1{\lambda r}N_k^2
 \lesssim N_k^2.
\]
For the inner ball $|x|\lesssim r_0$, one has to use frequency localization. Choose a smooth
Fourier cutoff equal to one on the support of \(\widehat u_k\), so that
\(u_k=K_\lambda*u_k\), where
\(K_\lambda(x)=\lambda^3K(\lambda x)\) and \(K\) is of Schwartz class.
For \(|x|\leq r_0\), splitting the integration variable into
\(B_{4r_0}\) and shells of radius \(2^jr_0\), \(j\geq2\), gives
\begin{align*}
 \sup_{|x|\leq r_0}|u_k(t,x)|
 \lesssim_L{}&\lambda^{3/2}\|u_k(t)\|_{L^2(B_{4r_0})}+\lambda^{3/2}\sum_{j\geq2}2^{-j(L-3/2)}
       \|u_k(t)\|_{L^2(|x|\approx2^jr_0)}.
\end{align*}
We take \(L^2_t\) after the spatial supremum and use the Minkowski inequality.
By the norm definition the first term is bounded by
\(C\lambda^{-1}N_k\) and the second by
\(C\lambda^{-1}2^{j/2}N_k\). For \(L>2\) this yields
\[
 \|u_k\|_{L^2_tL^\infty_x(B_{r_0})}
 \lesssim\lambda^{1/2}N_k.
\]
Since the dimension is three, we observe
\[
 \|\langle x\rangle^{-1}u_k\|_{L^2_{t,x}(B_{r_0})}^2
 \lesssim r_0\lambda N_k^2=N_k^2.
\]
We thus have proved a uniform dyadic weighted embedding at all frequencies.

Weighted Littlewood-Paley theory implies
\begin{align*}
 \|\langle x\rangle^{-1}\langle D_x\rangle^su\|_2^2
 &\lesssim_s\sum_k\langle2^k\rangle^{2s}
            \|\langle x\rangle^{-1}S_ku\|_2^2\lesssim_s\sum_k\langle2^k\rangle^{2s}
            \|S_ku\|_{ X_k}^2.
\end{align*}

Now, the claim \eqref{eq:claim} is proved and we complete the proof of the error estimates, i.e.\ Proposition \ref{prop-error-dirac}:
We introduce smooth test spinors \(u,v\). Then the definition of \(B_{\pm,s}\)
and the self-adjointness of \(\langle D_x\rangle^s\) give
\[
 \langle R_\pm u,v\rangle_{t,x}
 =
 \left\langle
 B_{\pm,s}(\langle x\rangle^{-1}\langle D_x\rangle^s u),\langle x\rangle^{-1}\langle D_x\rangle^{-s}v
 \right\rangle_{t,x}.
\]
Using \eqref{eq:B} and  \eqref{eq:claim}, with Sobolev
indices \(s\) and \(-s\), respectively, we obtain
\begin{align*}
 |\langle R_\pm u,v\rangle_{t,x}|
 &\leq
 \sup_t\|B_{\pm,s}(t)\|_{L_x^2\to L_x^2}
 \|\langle x\rangle^{-1}\langle D_x\rangle^su\|_{L^2_{t,x}}
 \|\langle x\rangle^{-1}\langle D_x\rangle^{-s}v\|_{L^2_{t,x}}\\
 &\lesssim_{s,M}
 \varepsilon
 \|u\|_{ X^s}
 \|v\|_{ X^{-s}}.
\end{align*}
The same formula defines a bounded sesquilinear form on
\( X^s\times X^{-s}\), agreeing with the
distributional pairing on smooth test spinors.
Since \( Y^s=( X^{-s})'\), it follows that
\[
 \|R_\pm u\|_{ Y^s}
 =
 \sup_{\|v\|_{ X^{-s}}\leq1}
 |\langle R_\pm u,v\rangle_{t,x}|
 \lesssim_{s,M}
 \epsilon\|u\|_{ X^s},
\]
which completes the proof.
\end{proof}

In what follows, we adapt the argument of \cite{metataru,tataru4} to construct an outgoing parametrix for the half-Klein-Gordon operator, using the Weyl quantization. To relate this convention to the standard quantization used above, we write
\[
\Omega_g = \mathrm{Op}^w(\omega_g) + E_g, \quad E_g := \mathrm{Op}(\omega_g)-\mathrm{Op}^w(\omega_g).
\]
The change of quantization formula (see \eqref{eq-change-quantization}) shows that $E_g$ has order $0$. Furthermore, its symbol $e_g$ satisfies
\[
|\partial_x^\alpha \partial_\xi^\beta e_g| \lesssim_{\alpha,\beta,M} \epsilon \langle x\rangle^{-2-|\alpha|} \langle \xi \rangle^{-|\beta|}.
\]
The above inequality allows us to use the argument of Proposition \ref{prop-error-dirac} and we obtain $\|E_g u\|_{Y^s} \lesssim \epsilon \|u\|_{X^s} $. Therefore the Dirac equation becomes
\[
(D_t\pm\mathrm{Op}^w(\omega_g))\psi_\pm = R_\pm \psi \mp E_g\psi,
\]
and the additional remainder terms can be treated within the same error estimates.
\subsection{Nonlinear Dirac equations on a curved space-time}
We are concerned with initial value problems for the Dirac equations with a cubic-type nonlinearity, given by
\begin{equation}\label{eq-dirac-cubic}
\begin{split}
	(-i\gamma^\mu\mathbf D_\mu+M)\psi & = (\psi^\dagger\psi)\gamma^0\psi \\
	\psi|_{t=0} & = \psi_0,
\end{split}
\end{equation}
where $\psi^\dagger$ is the complex conjugate transpose of the spinor $\psi$. We apply the operators $\Pi_\pm^M$ introduced in the previous section and reformulate the systems \eqref{eq-dirac-cubic}:
\begin{align}\label{eq-dirac-cubic-proj}
\left\{
\begin{array}{l}
 (D_t+\Omega_g)\psi_+ = \Pi_+^M [( \psi^\dagger\psi)\psi ]+R_+\psi	, \\
  (D_t-\Omega_g)\psi_- = \Pi_-^M [( \psi^\dagger\psi)\psi ]+R_-\psi, \\
	(\Pi^M_+\psi,\Pi^M_-\psi)|_{t=0} = (\Pi_+^M\psi_{0},\Pi_-^M\psi_{0}).
\end{array}\right.
\end{align}

Now Theorem \ref{thm-dirac-gwp} can be restated in terms of $(\Pi^M_+\psi,\Pi^M_-\psi)$ as follows:
\begin{thm}\label{thm-dirac-reform}
Let $s>1$ and $M>0$. Suppose that the metric $ g$ satisfies \eqref{assump-dirac-metric} and $\epsilon>0$ is sufficiently small. Then, the Cauchy problem for the equations \eqref{eq-dirac-cubic-proj} on the curved spacetime $(\mathbb R^{1+3}, g)$ is globally well-posed for small initial data: $\|\Pi_+^M\psi_{0}\|_{H^s_x(\mathbb R^3)}+\|\Pi_-^M\psi_{0}\|_{H^s_x(\mathbb R^3)}\ll1$. Furthermore, these solutions scatter to free solutions for $t\rightarrow \pm\infty$.
\end{thm}
Due to the cubic nonlinear feature of the Cauchy problems \eqref{eq-dirac-cubic-proj}, it suffices to investigate the $L^2_tL^\infty_x$-estimates. Therefore in the rest of this paper, we focus on the endpoint Strichartz estimates for the half-Klein-Gordon operators $D_t+\mathrm{Op}(\sqrt{1+g^{ij}\xi^i\xi^j})$.
\section{Half-Klein-Gordon operators with a small perturbation}\label{sec:local-energy-estimate}
 We study the Cauchy problems for the half-Klein-Gordon equations with variable coefficients:
\begin{align}\label{main-eq-half-kg}
\begin{aligned}
\left(D_t\pm \mathrm{Op}\sqrt{1+g^{ij}(t,x)\xi_i\xi_j}\right)u&=f, \quad (t,x)\in\mathbb R^{1+d}, \ d\ge3, \\
u|_{t=0} & = u_0	,
\end{aligned}
\end{align}
where the metic satisfies the regularity conditions \eqref{dirac-asymp-flat1} and \eqref{dirac-asymp-flat2}. In what follows, we put $A^\pm(t,x,D_x) = \mathrm{Op}\sqrt{1+g^{ij}\xi_i\xi_j} $ for simplicity.

\subsection{Mollification of the half-Klein-Gordon operator}
We mollify the perturbed symbol:
\begin{align}
 q(t,x,\xi) :=   \sqrt{1+ g^{ij}(t,x)\xi_i\xi_j } - \sqrt{1 +|\xi|^2 }.
\end{align}
To do this, we define
\begin{align}
    q_{(k)}(t,x,\xi) := \sum_{l<k-4}S_{<l}\chi_{<k-2l} S_l q(t,x,\xi).
\end{align}
Then we define the mollified half-Klein-Gordon operator:
\begin{align}
    A^\pm_{(k)}(t,x,D_x) := \sqrt{1+|D_x|^2}+ \mathrm{Op}( q_{(k)}(t,x,\xi)).
\end{align}


This mollification gives the following: see also (32) of \cite{metataru}
\begin{align}\label{bdd-mollification}
    \begin{aligned}
         |\partial_x^\alpha \partial_\xi^\beta q_{(k)}(t,x,\xi)| \le c_\alpha \epsilon_k(|x|) 2^{|\alpha|k}(1+2^k|x|)^{-|\alpha|}\langle \xi\rangle^{1-|\beta|}, \quad |\alpha|\le2, \\
          |\partial_x^\alpha \partial_\xi^\beta q_{(k)}(t,x,\xi)| \le c_\alpha\epsilon_k(|x|)2^{|\alpha|k}(1+2^k|x|)^{-1-\frac{|\alpha|}{2}}\langle\xi\rangle^{1-|\beta|} , \quad |\alpha|\ge2.
    \end{aligned}
\end{align}

We also define a global mollified operator
\begin{align*}
    \widetilde{A}^\pm = \sum_{k=-\infty}^\infty A^\pm_{(k)}S_k.
\end{align*}
\begin{prop}\label{prop-mollified-operator}
The following estimates hold:
\begin{align*}
    \|(A^\pm-\widetilde{A}^\pm)u\|_{Y^s} \lesssim \epsilon \|u\|_{X^s}, \\
    \|(\widetilde{A}^\pm-A^\pm_{(k)})S_lu\|_{X'_k} \lesssim \epsilon \|S_lu\|_{X_k}, \quad |l-k|\le2, \\
    \|[A^\pm_{(k)},S_k]u\|_{X'_k} \lesssim \epsilon \|u\|_{X_k}.
\end{align*}
\end{prop}
\begin{proof}
We put $\pm=+$ and $A^+ = A$ for simplicity and write
\begin{align}
    A- \widetilde{A}  = A_{\rm low}+ A_{\rm mid}+A_{\rm high},
\end{align}
where
\begin{align}
    \begin{aligned}
        A_{\rm low} := \sum_{k=-\infty}^{\infty} \left( \sum_{l<k-4} (S_{<l}\chi_{\ge k-2l})(S_l q) \right) S_k, \\
        A_{\rm mid} := \sum_{k=-\infty}^{\infty} \left( \sum_{l=k-4}^{l=k+4} (S_l q) \right) S_k, \\
         A_{\rm high} := \sum_{k=-\infty}^{\infty} \left( \sum_{l>k+4} (S_l q) \right) S_k.
    \end{aligned}
\end{align}
Then the remaining task readily follows by the argument of \cite{metataru}. We refer the readers to the proof of Proposition 1 of \cite{metataru}.
\end{proof}
\subsection{Scaling argument}\label{sec:scaling}
By the definition of the mollified operator $\widetilde{A}$, we are able to concentrate on a dyadic piece of $\widetilde A$:
\begin{align}
    (D_t+ A^\pm_{(k)})S_k u = 0.
\end{align}
Now we rescale the frequency-localised equation. To do this, we set $s=\lambda t$, $y=\lambda x$, and $u_\lambda = S_k u(\lambda^{-1}s,\lambda^{-1}y)$. Then the above frequency-localised equation is transformed into 
\begin{align}
    \left( D_s +\mathcal A^\pm_{(k)}(s,y,D_y)  \right) u_\lambda =0,
\end{align}
where 
\begin{align}
    \mathcal A^\pm_{(k)}(s,y,D_y) := \pm  \mathrm{Op}_y( a_{(k)}(s,y,\eta)), 
\end{align}
and
\begin{align}
    a_{(k)}(s,y,\eta) := \sqrt{\lambda^{-2}+|\eta|^2}+\lambda^{-1} q_{(k)}(\lambda^{-1}s,\lambda^{-1}y,\lambda\eta) .
\end{align}
Now we need to show that the functions $a_{(k)}(s,y,\eta)$ belong to a symbol class $S^1$ uniform in $k$ on $|\eta|\approx1$. Indeed, we note that
\begin{align*}
    \partial_y^\alpha q_{(k)}(\lambda^{-1}y) = \lambda^{-|\alpha|}(\partial_x^\alpha q_{(k)})(\lambda^{-1}y).
\end{align*}
Then combining this with
the bound \eqref{bdd-mollification} yields the following inequality:
\begin{align}
    | \partial_y^\alpha \partial_\eta^\beta  q_{(k)}(\lambda^{-1}y)| \lesssim \epsilon_k(\lambda^{-1}|y|) 2^{k|\alpha|} \lambda^{-|\alpha|} (1+2^k \lambda^{-1}|y|)^{-|\alpha|}.
\end{align}
Furthermore, if we put $\lambda=2^k$, then we obtain
\begin{align}
     | \partial_y^\alpha \partial_\eta^\beta  q_{(k)}(\lambda^{-1}y)| \lesssim  \epsilon_k(\lambda^{-1}|y|) (1+|y|)^{-|\alpha|}.
\end{align}
Therefore, we conclude that the rescaled operators belong to a fixed perturbative symbol class at the unit frequency. All frequency-localised estimates can be proved after rescaling at frequency $\approx1$, uniformly in $k$, and then summed over dyadic $k$. 

\subsection{Localised energy estimates}\label{subsec:loc_en}
In this section, we shall prove the localised energy estimates for the (half) Klein-Gordon operator.
\begin{thm}[Localised energy estimates for the half Klein-Gordon operators]\label{thm-loc-en-est}
Consider the Cauchy problem for the inhomogeneous half-Klein-Gordon equation \eqref{main-eq-half-kg} and suppose that $\epsilon>0$ is sufficiently small. Then the half Klein-Gordon operator $D_t+A^\pm(t,x,D_x)$ satisfies the localised energy estimates globally in time in the following sense: For each initial data $u_0\in H^s_x$ and each inhomogeneous term $f\in L^1_tH^s_x+Y^s$, there exists an unique solution $u\in L^\infty_tH^s_x\cap X^s $ to \eqref{main-eq-half-kg} satisfying the estimates
	\begin{align}
	\|u\|_{L^\infty_tH^s_x\cap X^s} & \lesssim \|u_0\|_{H^s_x}+\|f\|_{L^1_tH^s_x+Y^s}.	
	\end{align}
\end{thm}
The proof is very similar to \cite{tataru4}. Let $\{\alpha_m\}_{m\in\mathbb Z}$ be a positive slowly varying sequene with $\sum\alpha_k=1$. We define the space $X_{k,\alpha}$ with the norm
\begin{align*}
	\|u\|_{X_{k,\alpha}}^2 = \langle 2^{k}\rangle^{-1} 2^{2k}\|u\|^2_{L^2(A_{<-k})}+\langle 2^{k}\rangle^{-1}2^k\sum_{j\ge-k}\alpha_j \|(|x|+2^{-k})^{-\frac12}u\|^2_{L^2(A_j)}
\end{align*}
and the dual space
\begin{align*}
	\|u\|_{X'_{k,\alpha}}^2 = \langle 2^{k}\rangle2^{-2k}\|u\|^2_{L^2(A_{<-k})}+\langle 2^{k}\rangle2^{-k}\sum_{j\ge-k}\alpha_j^{-1} \|(|x|+2^{-k})^{\frac12}u\|^2_{L^2(A_j)}
\end{align*}
To prove Theorem \ref{thm-loc-en-est} we show the following localised version:
\begin{prop}\label{prop-local-local}
Assume that $\epsilon>0$ is sufficiently small. Then the bound
\begin{align}\label{local-energy-localised}
\|u\|_{L^\infty_tL^2_x\cap X_{k,\alpha}} \lesssim \|u(0)\|_{L^2_x}+\|(D_t+A^\pm_{(k)}u\|_{L^1_tL^2_x+X_{k,\alpha}'}	
\end{align}
	holds for all functions $u\in L^\infty_tL^2_x\cap X_{k,\alpha}$ localised at frequency $2^k$ uniformly with respect to all slowly varying squences $\{\alpha_m\}$ with $\sum_{m=-\infty}^\infty\alpha_m=1$.
\end{prop}
\begin{proof}
In what follows, we are only concerned with $\pm=+$ and put $A^+_{(k)}=A_{(k)}$ for simplicity.
	Let $Q$ be an $L^2$-bounded self-adjoint operator on $\mathbb R^d$. Then the operator $C_k=i[A_{(k)},Q]$ is also self-adjoint and we have
	\begin{align*}
		2\Im \langle A_{(k)}u,Qu\rangle_{L^2_x} = \langle C_k u,u\rangle_{L^2_x}.
	\end{align*}
	Then we obtain
	\begin{align*}
		\frac{d}{dt}\langle u, Qu\rangle_{L^2_x} = -2\Im \langle (D_t+A_{(k)})u, Q u\rangle_{L^2_x} + \langle C_k u,u\rangle_{L^2_x}.
	\end{align*}
	If we put $Q=I$ this gives the energy estimate
	\begin{align*}
		\frac{d}{dt}\|u(t)\|^2_{L^2_x} = -2\Im \langle (D_t+A_{(k)})u,u\rangle_{L^2_x}.
	\end{align*}
	If $\delta>0$ is a small parameter then the modified energy given by
	\begin{align*}
		E(u) = \|u\|_{L^2_x}^2-\delta\langle u,Qu\rangle_{L^2_x}
	\end{align*}
	is positive-definite and satisfies
	\begin{align*}
		\frac{d}{dt}E(u) = 2\Im\langle (D_t+A_{(k)})u,(1-\delta Q)u\rangle_{L^2_x} - \delta\langle C_ku,u\rangle_{L^2_x}.
	\end{align*}
	Integrating in time we obtain
	\begin{align}
	\|u\|_{L^\infty_tL^2_x}^2+\delta \langle C_ku,u\rangle_{L^2_x} \lesssim \|u(0)\|_{L^2_x}^2+\|(D_t+A_{(k)})S_k u\|_{L^1_tL^2_x+ X_{k,\alpha}'}\|(1-\delta Q)u\|_{L^\infty_tL^2_x\cap X_{k,\alpha}}	
	\end{align}
	We shall use the following lemma:
	\begin{lem}[Lemma 10 of \cite{tataru4}]
		For each $k\in\mathbb Z$ and each slowly varying sequence $\{\alpha_m\}$ there exists an $L^2$-bounded self-adjoint operator $Q$ so that
				\begin{align}
		\langle C_ku,u\rangle \gtrsim \|u\|_{X_{k,\alpha}}^2,\label{est-lem-c} \\
		\|Qu\|_{X_{k,\alpha}} \lesssim \|u\|_{X_{k,\alpha}},	 \label{est-lem-q}
		\end{align}
for all functions $u$ localised at frequency $2^k$.
	\end{lem}
Then the estimate \eqref{local-energy-localised} is obtained by using the Cauchy-Schwartz inequality and the above lemma. Now we prove the lemma. By scaling $(t,x)\to(\lambda^{-1}t,\lambda^{-1}x)$ we restrict ourselves into the unit scale, i.e., $|\xi|\approx1$ and focus on the scaled operator $\mathcal A_{(k)}$. Let $\delta>0$ be a small parameter to be chosen later. We increase the sequence $\{\alpha_m\}$ so that it remains slowly varying and the following properties hold:
\begin{align*}
	\sum_{m>0}\alpha_m \approx 1, \ \epsilon_m \le \delta\alpha_{m+\log_2\delta},
\end{align*}
which implies that $\epsilon<\delta$. To the sequence $\{\alpha_m\}$ we associate a slowly varying function $\alpha$ with the property that $\alpha(s)\approx \alpha_m$, for $s\approx 2^m$ and regularity $|\partial^k\alpha(s)|\approx (1+s)^{-k}\alpha(s)$, for $k\le 4$. We consider a decreasing function $\phi$ on $\mathbb R^+$ with the following properties:
\begin{enumerate}
	\item $\phi(s)\approx(1+s)^{-1}$ for $s>0$ and $|\partial^k\phi(s)|\lesssim (1+s)^{-k-1}$ for $k\le4$.
	\item $\phi(s)+s\phi'(s)\approx(1+s)^{-1}\alpha(s)$ for $s>0$.
	\item $\phi(|x|)$ is localised at frequency at most $O(1)$.
\end{enumerate}
For simplicity, we consider the function defined by
\begin{align*}
	s\phi(s) = \int_0^s \frac{\alpha(r)}{(1+r^2)^\frac12}\,dr.
\end{align*}
A straightforward computation gives
\begin{align*}
    \phi(s)+s\phi'(s) = \frac{\alpha(s)}{(1+s^2)^\frac12}, \\
    2\phi'(s)+s\phi''(s) = \frac{\alpha'(s)}{(1+s^2)^\frac12}-\frac{3s\alpha(s)}{(1+s^2)^\frac32}.
\end{align*}
We define the operator $Q(x,D_x)$ to be
\begin{align*}
	Q(x,D_x) = \delta \left(D_xx\phi(\delta|x|)+\phi(\delta|x|)xD_x\right).
\end{align*}
Then $Q$ is the $L^2_x$-bounded self-adjoint operator on $\mathbb R^d$ and satisfies the estimates \eqref{est-lem-q}. (See Lemma 10 of \cite{tataru4}.) We are left to show the estimates \eqref{est-lem-c}. It suffices to prove that
\begin{align}\label{bdd-lower-C}
|\langle C_0u,u\rangle_{L^2_x}| \gtrsim \left\langle  \frac{\delta\alpha(\delta|x|)}{1+\delta|x|}u,u \right\rangle_{L^2_x}.	
\end{align}
Note that
\begin{align*}
    |\langle C_0u,u\rangle| & = |\langle[ \mathcal A_{(k)}, Q(x,D_x)]u\rangle| \\
    & \ge |\langle [\sqrt{\lambda^{-2}+|D_x|^2},Q(x,D_x)]u,u\rangle| - |\langle \mathrm{Op}(q_{(k)})  ,Q(x,D_x)]u,u \rangle|,
\end{align*}
where we recall that $\mathcal A_{(k)}(s,y,D_y) = \mathrm{Op}_y( a_{(k)}(s,y,\eta)) $ and $a_{(k)}-\sqrt{\lambda^{-2}+|\eta|^2}= q_{(k)}(\lambda^{-1}s,\lambda^{-1}y,\lambda \eta) $.
We compute the commutator $[\mathcal A_{(k)},Q]$. Then
\begin{align*}
	[\sqrt{\lambda^{-2}+|D_x|^2},Q(x,D_x)] & = \delta [\sqrt{\lambda^{-2}+|D_x|^2}, \phi(\delta|x|)+x\phi'(\delta|x|)\delta\frac{x}{|x|}]+2\delta[\sqrt{\lambda^{-2}+|D_x|^2}, \phi(\delta|x|)x]D_x,
\end{align*}
and
\begin{align*}
	[ \mathrm{Op}(q_{(k)}) ,Q(x,D_x)] & = \delta [ \mathrm{Op}(q_{(k)}) ,\phi(\delta|x|)+x\phi'(\delta|x|\delta\frac{x}{|x|})] \\
	& \qquad+2\delta[ \mathrm{Op}(q_{(k)}),\phi(\delta|x|)x]D_x
\end{align*}
To estimate the $L^2$-bound of the commutators, we compute the Poisson bracket:
\begin{align*}
    [\sqrt{\lambda^{-2}+|D_x|^2}, \phi(\delta|\cdot|)\cdot](x,D_x) = \left(\frac{\partial}{\partial\xi}\sqrt{\lambda^{-2}+|\xi|^2}\frac{\partial}{\partial x}\phi(\delta|x|)x\right)(x,D_x) +r(x,D_x),
\end{align*}
where $r(x,D_x)$ is a remainder.
Then
\begin{align*}
    \frac{\partial}{\partial\xi}\sqrt{\lambda^{-2}+|\xi|^2}\frac{\partial}{\partial x}\phi(\delta|x|)x = \frac{\xi}{\sqrt{\lambda^{-2}+|\xi|^2}}(\phi(\delta|x|)+\delta x\phi'(\delta|x|)\frac{x}{|x|})
\end{align*}
is the principal symbol of the commutator. From the above computation, we see that
\begin{align*}
    \left|\frac{\xi}{\sqrt{\lambda^{-2}+|\xi|^2}}(\phi(\delta|x|)+\delta x\phi'(\delta|x|)\frac{x}{|x|})\right| \gtrsim \frac{\alpha(\delta|x|)}{1+\delta|x|}.
\end{align*}
On the other hand,
\begin{align*}
    2\delta\phi'(\delta|x|)+\delta^2 x\phi''(\delta |x|) \lesssim \frac{\delta \alpha(\delta|x|)}{(1+\delta|x|)^2}
\end{align*}
To determine the $L^2$-norm of the Poisson bracket, we need only at most fourth order derivatives of the symbol. 
By choosing $\delta>0$ small enough, we see that the zeroth-order derivative of the principal symbol dominates the bound.
Now we write
\begin{align*}
    |\langle [\sqrt{\lambda^{-2}+|D_x|^2},Q(x,D_x)]u,u\rangle | & \ge 2\delta|\langle [\sqrt{\lambda^{-2}+|D_x|^2}, \phi(\delta|x|)x]D_x u,u\rangle | \\
    &\qquad-\delta |\langle [\sqrt{\lambda^{-2}+|D_x|^2}, \phi(\delta|x|)+x\phi'(\delta|x|)\delta\frac{x}{|x|}]u,u\rangle| \\
    & \gtrsim \langle \frac{\delta\alpha(\delta |x|)}{1+\delta|x|}D_xu,u\rangle \approx \langle \frac{\delta\alpha(\delta |x|)}{1+\delta|x|}u,u\rangle,
\end{align*}
where we used that the function $u$ is localised at frequency $1$. The computation of the perturbative terms containing $q_{(k)}$ follows in the similar way and we get the desired lower bound \eqref{bdd-lower-C} provided that $\epsilon<\delta$.

Low frequency case also follows immediately. Indeed,
if $2^k\le1$, i.e., $k\le0$, we have
\begin{align*}
    \langle [\sqrt{1+|D_x|^2},\phi(\delta|x|)x]D_xu,u\rangle &\gtrsim \langle\frac{ 2^k \alpha(\delta|x|) }{1+\delta|x|}u,u\rangle.
\end{align*}
We conclude that the estimates \eqref{est-lem-c} hold for all $k\in\mathbb Z$.
\end{proof}

\begin{proof}[Proof of Theorem \ref{thm-loc-en-est}]
The argument to obtain the localised energy estimates from the estimates \eqref{local-energy-localised} is identical as the proof of Theorem 3 of \cite{tataru4}. From Proposition \ref{prop-local-local} one can deduce that for any function $u\in L^\infty_tL^2_x\cap X^0$ localised at frequency $|\xi|\approx2^k$, the following localised estimates hold:
\begin{align*}
	\|S_ku\|_{L^\infty_tL^2_x}+\|S_ku\|_{X_k} & \lesssim \|S_ku_0\|_{L^2_x}+ \|(D_t+A^\pm_{(k)})S_ku\|_{L^1_tL^2_x+X_{k}'}.
	\end{align*}
From this,  we infer that
\begin{align*}
	\|u\|_{L^\infty_tH^s_x\cap X^s} ^2
	\lesssim{} & \sum_{k\in\mathbb Z}\langle 2^k\rangle^{2s}\|S_ku\|_{L^\infty_tL^2_x\cap X_k}^2 \\
	 \lesssim{} &\sum_{k\in\mathbb Z}\langle 2^k\rangle^{2s}\|S_ku(0)\|_{L^2_x}^2+\sum_{k\in\mathbb Z}\langle 2^k\rangle^{2s}\|(D_t+A^\pm_{(k)})S_ku\|_{L^1_tL^2_x+X_k'}^2
     \end{align*}
     Using Proposition \ref{prop-mollified-operator} this is bounded by
     \begin{align*}
   & \|u(0)\|_{H^s}^2+\sum_{k\in\mathbb Z}\langle 2^k\rangle^{2s}\left(\|S_k(D_t+A^\pm_{(k)})u\|_{L^1_tL^2_x+X_k'}^2+ \|[A^\pm_{(k)},S_k]u\|_{X_k'}^2
   \right)\\
    \lesssim{} & \|u(0)\|^2_{H^s}+\|(D_t+\widetilde A^\pm)u\|_{L^1_tH^s_x+Y^s}^2+\epsilon \|u\|_{X^s}^2 \\
    \lesssim{} & \|u(0)\|^2_{H^s}+\|(D_t+ A^\pm)u\|_{L^1_tH^s_x+Y^s}^2+\|(\widetilde A^\pm-A^\pm)u\|_{Y^s}+\epsilon \|u\|_{X^s}^2 \\
     \lesssim{}  &\|u(0)\|^2_{H^s}+\|(D_t+ A^\pm)u\|_{L^1_tH^s_x+Y^s}^2+\epsilon \|u\|_{X^s}^2,
\end{align*}
and because $\epsilon$ is small the proof is complete.
 \end{proof}
\subsection{Parametrices and the endpoint Strichartz estimates}
Now we construct an outgoing parametrix for the evolution operator $D_t+A_{(k)}(t,x,D_x)$.
We first use the scaling argument $(t,x)\to (\lambda^{-1}t,\lambda^{-1}x)$ for $\lambda\ge1$ so that the operators $A^\pm_{(k)}$ are transformed to $\mathcal A^\pm_{(k)}$, see Subsection \ref{sec:scaling}. We recall that the rescaled operator $\mathcal A^\pm_{(k)}$ is given by
\begin{align*}
    \mathcal A^\pm_{(k)}(t,x,D_x) = \pm \mathrm{Op}( a_{(k)}(t,x,\xi)),
\end{align*}
where $a_{(k)}(t,x,\xi) = \sqrt{\lambda^{-2}+|\xi|^2}+q_{(k)}(\lambda^{-1}t,\lambda^{-1}x,\lambda \xi) $.
The rescaled operator is localised at unit frequency $|\xi|\approx 1$ provided that $\lambda=2^k$ for $k>0$.

Since the operator $D_t+A^\pm_{(k)}$ for $k\le0$, i.e., the operator localised in the unit disc $\{|\xi|\le1\}$ can be treated as similar manner we consider only $k>0$ in the following.


Now we construct an outgoing parametrix for the rescaled operators $D_t+\mathcal A^\pm_{(k)}$ with $\lambda=2^k$.
\begin{prop}\label{prop-parametrix}
    Assume that $\epsilon>0$ is sufficiently small. There is an outgoing parametrix $\mathcal K^\pm_{(k)}$ for $D_t+\mathcal A^\pm_{(k)}$ satisfying the following: 
    \begin{enumerate}
        \item $L^2$-bound: \begin{align}
        \|\mathcal K^\pm_{(k)}(t,s)\|_{L^2_x\rightarrow L^2_x} \lesssim1, 
        \end{align}
        \item Error estimate:
        \begin{align}
            \begin{aligned}
                \|(1+|x|)^N (D_t+\mathcal A^\pm_{(k)})\mathcal K^\pm_{(k)}(t,s)\|_{L^2_x\rightarrow L^2_x} \lesssim (1+|t-s|)^{-N}, \quad t\neq s, \\
                \|(1+|x|)^N D_t(D_t+\mathcal A^\pm_{(k)})\mathcal K^\pm_{(k)}(t,s)\|_{L^2_x\rightarrow L^2_x} \lesssim (1+|t-s|)^{-N}, \quad t\neq s,
            \end{aligned}
        \end{align}
        \item Jump condition: $\mathcal K^\pm_{(k)}(s+0,s)$ and $\mathcal K_{(k)}^\pm(s-0,s)$ are $S^0_{1,0}$-type pseudodifferential operators satisfying
        \begin{align}
            (\mathcal K^\pm_{(k)}(s+0,s)-\mathcal K^\pm_{(k)}(s-0,s))S_1 = S_1,
        \end{align}
        \item Outgoing parametrix:
        \begin{align}
            \|\mathbf1_{\{|x|<2^{-10}|t-s|\}}\mathcal K^\pm_{(k)}(t,s)\|_{L^2_x\rightarrow L^2_x} \lesssim(1+|t-s|)^{-N},
        \end{align}
        \item Pointwise decay: for $0\le\theta\le1$,
        \begin{align}
            \|\mathcal K^\pm_{(k)}(t,s)\|_{L^1_x\rightarrow L^\infty_x} \lesssim \langle 2^k\rangle^\theta (1+|t-s|)^{-\frac{d-1+\theta}{2}}.
        \end{align}
    \end{enumerate}
\end{prop}
The factor $\langle 2^k\rangle$ of the bounds in Proposition is related to the scaling argument $(t,x)\to(2^{-k}t,2^{-k}x)$.
The proof of Proposition \ref{prop-parametrix} is postponed to Section 5. In the rest of this section we focus on the proof of the endpoint Strichartz estimates Theorem \ref{thm-strichartz} and Theorem \ref{thm-dirac-reform} using Proposition \ref{prop-parametrix}.
\begin{prop}
	An outgoing parametrix $\mathcal K^\pm_{(k)}$ as in Proposition \ref{prop-parametrix} satisfies the following:
	\begin{enumerate}
		\item (regularity) For any Strichartz pairs $(p_1,q_1)$ respectively $(p_2,q_2)$ with $q_1\le q_2$ we have
		\begin{align}\label{est-stri-local}
		\langle2^k\rangle^{-(1-\frac2{q_1})\frac\theta2}\|\mathcal K^\pm_{(k)} f\|_{L^{p_1}_tL^{q_1}_x}+\|\mathcal K^\pm_{(k)} f\|_{X_1} & \lesssim \langle2^k\rangle^{(1-\frac2{q_2})\frac\theta2} \|f\|_{L^{p_2'}_tL^{q_2'}_x}, 
		\end{align}
\item (error estimate) For any Strichartz pair $(p,q)$ we have
\begin{align}\label{error-local}
	\| [(D_t+\mathcal A^\pm_{(k)})\mathcal K^\pm_{(k)}-1]f\|_{X'_0} & \lesssim \langle2^k\rangle^{(1-\frac2{q_2})\frac\theta2} \|f\|_{L^{p'}_tL^{q'}_x}.
\end{align} 
	\end{enumerate}

\end{prop}
\begin{proof}
The error estimates follow directly from the error-type estimate of Proposition \ref{prop-parametrix}. We leave the proof of \eqref{est-stri-local} to the Appendix, which mainly follows the work of \cite{tataru4}.
\end{proof}
We recall that a tuple $(\sigma,p,q)$ is called the Strichartz admissible pair if it satisfies
\begin{align}\label{strichartz-admissible-pair}
\frac{2}{p}+\frac{d-1+\theta}{q} = \frac{d-1+\theta}{2}, \ \sigma = \left( \frac{d+1+\theta}{4} \right)\left( 1-\frac2q \right)
\end{align}
for $0\le\theta\le1$. Here the $(\sigma,p,q)=(1,2,\infty)$ with $d=3$ is the forbidden endpoint. So far we have considered the scaled operator $\mathcal A^\pm_{(k)}$. By rescaling, we recover the operator $A^\pm_{(k)}$, and hence we are able to obtain the localised Strichartz estimates and the error-type estimates for each $k\in\mathbb Z$. Let $\widetilde K^\pm_{(k)}$ denote the rescaled outgoing parametrices. Then the dyadic decomposition implies the Strichartz estimates for the mollified operator $\widetilde{A}^\pm$. The full outgoing parametrices for $\tilde{A}^\pm$ are given as
\begin{align*}
	\widetilde K^\pm(t,s) =\sum_{k \in \mathbb Z}\widetilde K^\pm_{(k)}(t,s).
\end{align*}
In consequence, we obtain the following:
\begin{prop}\label{prop-stri-error}
	There is an outgoing parametrix $\widetilde{K}^\pm$ for the operator $D_t+\widetilde A^{\pm}(t,x,D_x)$ which satisfy the following:
	\begin{enumerate}
	\item (Strichartz estimates) For any $s\in\mathbb R$ and Strichartz admissible pairs $(\sigma_1,p_1,q_1)$ and $(\sigma_2,p_2,q_2)$ satisfying \eqref{strichartz-admissible-pair} with $q_1\le q_2$ we have
	\begin{align}
	\|\widetilde K^\pm f\|_{\langle D_x\rangle^{\sigma_1-s}L^{p_1}_tL^{q_1}_x\cap X^s} \lesssim \|f\|_{\langle D_x\rangle^{-\sigma_2-s}L^{p_2'}_tL^{q_2'}_x},	
	\end{align}
\item (Error-type estimates) For any $s\in\mathbb R$ and Strichartz pair $(\sigma,p,q)$ we have
\begin{align}
\| [(D_t+\widetilde{A}^\pm) \widetilde K^\pm-1]f\|_{Y^s} & \lesssim \|f\|_{\langle D_x\rangle^{-\sigma-s}L^{p'}_tL^{q'}_x}.	
\end{align}
\end{enumerate}
\end{prop}
The next step is to establish an $L^2\to L^p_tL^q_x$ bound using duality.
\begin{prop}\label{prop-2-pq-dual}
	With $\widetilde K^\pm$ as in Proposition \ref{prop-stri-error} we have
	\begin{align}
	\| u\|_{\langle D_x\rangle^{-s+\sigma}L^p_tL^q_x} & \lesssim \|u\|_{L^\infty_t H^s_x\cap X^s} + \|(D_t+\widetilde A^\pm) u\|_{Y^s}.	
	\end{align}
\end{prop}
\begin{proof}
	Consider arbitrary $-\infty<T^-< T^+<+\infty$. We proceed by the usual duality argument. Indeed, we set $g\in \langle D_x\rangle^{s-\sigma}L^{p'}_tL^{q'}_x$. Then we have
	\begin{align*}
		\int_{T^-}^{T^+} \langle u,g\rangle_{L^2_x}\,dt & = \int_{T^-}^{T^+} \langle u, (D_t+\widetilde A^\pm) \widetilde K^\pm g\rangle_{L^2_x}\,dt   - \int_{T^-}^{T^+} \langle u,[(D_t+\widetilde A^\pm)\widetilde  K^\pm-1]g\rangle_{L^2_x}\,dt .
	\end{align*}
 Since the symbol of the operator $\widetilde A^\pm(t,x,D_x)$ is real, the operator $\widetilde A^\pm$ is self-adjoint and hence we see that
	\begin{align*}
		\int_{T_-}^{T^+}\langle u, \widetilde A^\pm(t,x,D_x)\widetilde K^\pm g\rangle_{L^2_x}\,dt & = \int_{T_-}^{T^+} \langle \widetilde A^\pm(t,x,D_x) u,{\widetilde K^\pm g}\rangle_{L^2_x}\,dxdt,
	\end{align*}
	and the use of integration by parts gives
	\begin{align*}
		\int_{T_-}^{T^+} \langle u, D_t\widetilde K^\pm g\rangle_{L^2_x}\,dt & = \int_{T_-}^{T^+} \langle D_t u,\widetilde K^\pm g\rangle_{L^2_x}\,dt + \langle u,\widetilde K^\pm g\rangle_{L^2_x}\bigg|^{T^+}_{T_-}.
	\end{align*}
	Thus we have
	\begin{align*}
		\int_{T^-}^{T^+}\langle u, g\rangle_{L^2_x}\,dt = \int_{T^-}^{T^+} \langle (D_t+\widetilde A^\pm)u, \widetilde K^\pm g\rangle_{L^2_x}\,dt - \int_{T^-}^{T^+}\langle u, [(D_t+\widetilde A^\pm)\widetilde K^\pm-1]g\rangle_{L^2_x}\,dt+\langle u,\widetilde K^\pm g\rangle_{L^2_x}\bigg|^{T^+}_{T_-},
	\end{align*}
	and hence we obtain
	\begin{align*}
		\left| \int_{T^-}^{T^+}\langle u, g\rangle_{L^2_x}\,dt\right| & \lesssim \|(D_t+\widetilde A^\pm)u\|_{Y^s} \|\widetilde K^\pm g\|_{X^{-s}}+\|u\|_{X^{s}}\|[(D_t+\widetilde A^\pm)\widetilde K^\pm-1]g\|_{Y^{-s}}+\|u\|_{L^\infty_t H^s_x}\|\widetilde K^\pm g\|_{L^\infty_tH^{-s}_x} \\
		& \lesssim \left(\|u\|_{L^\infty_tH^s_x\cap X^s}+\|(D_t+\widetilde A^\pm)u\|_{Y^s}\right) \|g\|_{\langle D_x\rangle^{s-\sigma}L^{p'}_tL^{q'}_x}.
	\end{align*}
    Since the constants are independent of $T^\pm$, we finally obtain
	\begin{align*}
		\|u\|_{\langle D_x\rangle^{-s+\sigma}L^p_tL^q_x} & = \sup\left\{ \left| \int_{\mathbb R} \langle u,g\rangle_{L^2_x}\,dt \right| : \|g\|_{\langle D_x\rangle^{s-\sigma}L^{p'}_tL^{q'}_x}\le1 \right\} \\
		& \lesssim \|u\|_{L^\infty_tH^s_x\cap X^s}+\|(D_t+\widetilde A^\pm)u\|_{Y^s},
	\end{align*} 
	which completes the proof.
\end{proof}
Now we cover the full range of $q$ for the Strichartz estimates.
\begin{prop} We have
\begin{align}
\|\widetilde K^\pm f\|_{\langle D_x\rangle^{\sigma_1-s}L^{p_1}_tL^{q_1}_x\cap X^s} & \lesssim \|f\|_{\langle D_x\rangle^{-s-\sigma_2}L^{p_2'}_tL^{q_2'}_x},	
\end{align}
for any Strichartz admissible $(\sigma_j,p_j,q_j)$, $j=1,2$.
\end{prop}
\begin{proof}
We repeat the duality argument with $u=\widetilde K^\pm f$ and $g\in \langle D_x\rangle^{s-\sigma_1}L^{p_1'}_tL^{q_1'}_x$. Then 
	\begin{align*}
		\int_{T^-}^{T^+}\langle u, g\rangle_{L^2_x}\,dt & = \int_{T^-}^{T^+} \langle (D_t+\widetilde A^\pm)u, \widetilde K^\pm g\rangle_{L^2_x}\,dt - \int_{T^-}^{T^+}\langle u, [(D_t+\widetilde A^\pm)\widetilde K^\pm-1]g\rangle_{L^2_x}\,dt+\langle u,\widetilde K^\pm g\rangle_{L^2_x}\bigg|^{T^+}_{T_-} \\
		& =  \int_{T^-}^{T^+} \langle (D_t+\widetilde A^\pm)\widetilde K^\pm f, \widetilde K^\pm g\rangle_{L^2_x}\,dt - \int_{T^-}^{T^+}\langle \widetilde K^\pm f, [(D_t+\widetilde A^\pm)\widetilde K^\pm-1]g\rangle_{L^2_x}\,dt \\
        &\qquad\qquad\qquad+\langle \widetilde K^\pm f,\widetilde K^\pm g\rangle_{L^2_x}\bigg|^{T^+}_{T_-}  \\
		& = \int_{T^-}^{T^+} \langle [(D_t+\widetilde A^\pm)\widetilde K^\pm -1] f, \widetilde K^\pm g\rangle_{L^2_x}\,dt  + \int_{T^-}^{T^+}\langle f, \widetilde K^\pm g\rangle_{L^2_x}\,dt \\
		& \qquad - \int_{T^-}^{T^+}\langle \mathcal K^\pm f, [(D_t+\widetilde A^\pm)\widetilde K^\pm-1]g\rangle_{L^2_x}\,dt+\langle \widetilde K^\pm f,\widetilde K^\pm g\rangle_{L^2_x}\bigg|^{T^+}_{T_-}.
	\end{align*}
	Now we use the estimates for $q_1>q_2$:
	\begin{align*}
		\|\widetilde K^\pm g\|_{\langle D_x\rangle^{s+\sigma_2}L^{p_2}_tL^{q_2}_x\cap X^{-s}} & \lesssim \|g\|_{\langle D_x\rangle^{-s-\sigma_1}L^{p_1'}_tL^{q_1'}_x}. 
	\end{align*}
		Then we obtain
	\begin{align*}
		\left| \int_{T^-}^{T^+}\langle \widetilde K^\pm f, g\rangle_{L^2_x}\,dt\right| & \lesssim \|[(D_t+\widetilde A^\pm)\widetilde K^\pm -1] f\|_{Y^s} \|\widetilde K^\pm g\|_{X^{-s}}+\|f\|_{\langle D_x\rangle^{-s-\sigma_2}L^{p_2'}_tL^{q_2'}_x}\|\widetilde K^\pm g\|_{\langle D_x\rangle^{s+\sigma_2}L^{p_2}_tL^{q_2}_x} \\
		& \qquad + \|\widetilde K^\pm f\|_{X^s} \|[(D_t+\widetilde A^\pm)\widetilde K^\pm-1]g\|_{Y^{-s}} + \|\widetilde K^\pm f\|_{L^\infty_tH^s_x} \|\widetilde K^\pm g\|_{L^\infty_tH^{-s}_x} \\
		& \lesssim \|f\|_{\langle D_x\rangle^{-s-\sigma_2}L^{p_2'}_tL^{q_2'}_x}\|g\|_{\langle D_x\rangle^{-s-\sigma_1}L^{p_1'}_tL^{q_1'}_x},
	\end{align*}
    and the claim follows by duality.
\end{proof}
\subsection{Proof of Theorem \ref{thm-strichartz}}
Now we arrive at the proof of Theorem \ref{thm-strichartz}, the Strichartz estimates for the inhomogeneous half-Klein-Gordon equation:
\begin{align*}
(D_t\pm\textrm{Op}(\sqrt{1+g^{ij}(t,x)\xi_i\xi_j})+A'(t,x,D_x))u = f+g, \quad u|_{t=0}= u_0\in H^s_x(\mathbb R^d),    
\end{align*}
where $A'$ is the lower order terms satisfying the error-type estimates and $f\in \langle D_x\rangle^{-s-\sigma_2}L^{p_2'}_tL^{q_2'}_x$ and $g\in Y^s$.
We want to control the $\langle D_x\rangle^{-s+\sigma_1}L^{p_1}_tL^{q_1}_x\cap X^s$-norm of the solution $u$. To do this, we first rewrite the above equation as follows:
\begin{align*}
    (D_t+\widetilde A^\pm) u = f+g+ (\widetilde A^\pm - A^\pm)u.
\end{align*}
Then we write
\begin{align*}
	u = \widetilde K^\pm f + \phi,
\end{align*}
where $\phi$ solves the equation
\begin{align*}
	(D_t+\widetilde A^\pm)\phi = [1-(D_t+\widetilde A^\pm)\widetilde K^\pm]f+g+(\widetilde A^\pm - A^\pm)u  .
\end{align*}
Then we see that
\begin{align*}
	\|u\|_{\langle D_x\rangle^{-s+\sigma_1}L^{p_1}_tL^{q_1}_x\cap X^s} & \lesssim \|\widetilde K^\pm f\|_{\langle D_x\rangle^{-s+\sigma_1}L^{p_1}_tL^{q_1}_x\cap X^s}+ \|\phi\|_{\langle D_x\rangle^{-s+\sigma_1}L^{p_1}_tL^{q_1}_x\cap X^s} .
\end{align*}
Then we apply in order Proposition \ref{prop-2-pq-dual}, the localised energy estimates Theorem \ref{thm-loc-en-est}, and the error-type estimates for the parametrix $\widetilde{K}^\pm$ to get 
\begin{align*}
	& \quad \|\phi\|_{\langle D_x\rangle^{-s+\sigma_1}L^{p_1}_tL^{q_1}_x\cap X^s}  \\
    & \lesssim \|\phi\|_{L^\infty_tH^s_x\cap X^s} + \|(D_t+\widetilde A^\pm)\phi\|_{Y^s} \\
	& \lesssim \|\phi(0)\|_{H^s}+\|(D_t+\widetilde A^\pm)\phi\|_{L^1_tH^s_x+Y^s} + \|[1-(D_t+\widetilde A^\pm)\widetilde K^\pm]f\|_{Y^s} +\|g\|_{Y^s}+ \|(\widetilde A^\pm - A^\pm)u\|_{Y^s} \\
	& \lesssim \|u_0\|_{H^s_x}+\|(D_t+\widetilde A^\pm)\phi\|_{Y^s}+ \|f\|_{\langle D_x\rangle^{-s-\sigma_2}L^{p_2'}_tL^{q_2'}_x}+\|g\|_{Y^s} +\epsilon\|u\|_{X^s} \\
	& \lesssim \|u_0\|_{H^s_x}+ \|f\|_{\langle D_x\rangle^{-s-\sigma_2}L^{p_2'}_tL^{q_2'}_x}+\|g\|_{Y^s}+\epsilon\|u\|_{X^s} ,
\end{align*}
where we used Proposition \ref{prop-mollified-operator} to control the error term $\widetilde{A}^\pm u-A^\pm u$.
Hence we obtain
\begin{align}\label{stri-prin}
	\|u\|_{\langle D_x\rangle^{-s+\sigma_1}L^{p_1}_tL^{q_1}_x\cap X^s} & \lesssim \|u_0\|_{H^s_x}+ \|(D_t+\widetilde A^\pm)u\|_{\langle D_x\rangle^{-s-\sigma_2}L^{p_2'}_tL^{q_2'}_x+Y^s} .
\end{align}
Now we apply the above Strichartz estimates to the original half-Klein-Gordon equation. We see that
\begin{align*}
    \|u\|_{\langle D_x\rangle^{-s+\sigma_1}L^{p_1}_tL^{q_1}_x\cap X^s} & \lesssim \|u_0\|_{H^s_x}+ \|(D_t+\widetilde A^\pm)u\|_{\langle D_x\rangle^{-s-\sigma_2}L^{p_2'}_tL^{q_2'}_x+Y^s} \\
    & \lesssim \|u_0\|_{H^s_x}+\|A'u\|_{\langle D_x\rangle^{-s-\sigma_2}L^{p_2'}_tL^{q_2'}_x+Y^s}+\|f+g\|_{\langle D_x\rangle^{-s-\sigma_2}L^{p_2'}_tL^{q_2'}_x+Y^s} \\
    & \lesssim \|u_0\|_{H^s_x}+\|A'u\|_{Y^s}+\|f\|_{\langle D_x\rangle^{-s-\sigma_2}L^{p_2'}_tL^{q_2'}_x}+\|g\|_{Y^s} \\
    & \lesssim \|u_0\|_{H^s_x}+\|f\|_{\langle D_x\rangle^{-s-\sigma_2}L^{p_2'}_tL^{q_2'}_x}+\|g\|_{Y^s} +\epsilon \|u\|_{X^s},
\end{align*}
where the term $\epsilon\|u\|_{X^s}$ can be absorbed into the left hand-side provided that $\epsilon$ is sufficiently small.

\subsection{Proof of Theorem \ref{thm-dirac-reform}} As a direct application of Theorem \ref{thm-strichartz}, we establish the global well-posedness for the cubic Dirac equation. Suppose that $\|\Pi_\pm^M\psi_0\|_{H^s_x}\le\eta$ and $\eta$ is small. We set
\begin{align*}
    Z^s := C ( \mathbb R; H^s_x(\mathbb R^3))\cap L^2_tL^\infty_x \cap X^s
\end{align*}
and 
\[
\mathcal X^s := Z^s\times Z^s, \quad \|\psi\|_{\mathcal X^s} := \|\psi_+\|_{Z^s}+\|\psi_-\|_{Z^s}, \quad B_\delta := \{ \psi : \|\psi\|_{\mathcal X^s}\le\delta \}.
\]
We shall construct a solution $\psi\in \mathcal X^s$ to the half-Klein-Gordon equation:
\[
(D_t\pm\Omega_g)\psi_\pm = \Pi_\pm^M[(\psi^\dagger\psi)\psi  ] +R_\pm \psi.
\]
Given an input $\psi=\psi_++\psi_-$, $(\psi_+,\psi_-)\in \mathcal X^s$, we define $\Psi_\pm(\psi)$ to be a solution to the linear inhomogeneous equation:
\[
(D_t\pm\Omega_g)\Psi_\pm(\psi) = \Pi_\pm^M[(\psi^\dagger\psi)\psi  ] +R_\pm \psi, \quad \Psi_\pm(\psi)|_{t=0} := \Pi_\pm^M\psi_0.
\]
Then Proposition \ref{prop-error-dirac} and a simple product estimate yield $\Pi_\pm^M[(\psi^\dagger\psi)\psi  ] \in L^1_tH^s_x $ and $R_\pm\psi\in Y^s$. Hence Theorem \ref{thm-loc-en-est} ensures that $\Psi_\pm(\psi)\in L^\infty_tH^s_x\cap X^s$ and it is unique. 

To control $L^2_tL^\infty_x$-norm of $\Psi_\pm(\psi)$, we apply Theorem \ref{thm-strichartz} with $\theta=\varepsilon\ll1$ after the use of the Sobolev embedding to obtain
\begin{align*}
    \|\Psi_\pm(\psi)\|_{L^2_tL^\infty_x} & \lesssim \|\Psi_\pm(\psi)\|_{\langle D_x\rangle^{-\frac{3}{q}}L^2_tL^q_x} \lesssim \|\Pi_\pm^M\psi_0\|_{H^s_x}+\|\Pi_\pm^M [(\psi^\dagger\psi)\psi]\|_{L^1_tH^s_x},
\end{align*}
where $(2,q)$ is the Strichartz admissible pair satisfying \eqref{strichartz-admissible-pair} with $\theta=\varepsilon$. Here $s>1$ is close to $1$. 
 Since the operators $\Pi_\pm^M$ are of the class $OPS^0$, we see that
 \begin{align*}
     \|\Pi_\pm^M [(\psi^\dagger\psi)\psi]\|_{L^1_tH^s_x} & \lesssim \| (\psi^\dagger\psi)\psi\|_{L^1_tH^s_x} \\
     & \lesssim \sum_{\pm_1,\pm_2,\pm_3\in \{+,-\}} \| [(\Pi^M_{\pm_1}\psi)^\dagger\Pi^M_{\pm_2}\psi]\Pi_{\pm_3}^M\psi \|_{L^1_tH^s_x} \\
     & \lesssim \sum_{\tau} \|\Pi_{\pm_{\tau(1)}}^M\psi\|_{L^\infty_tH^s_x}\|\Pi_{\pm_{\tau(2)}}^M\psi\|_{L^2_tL^\infty_x}\|\Pi_{\pm_{\tau(3)}}^M\psi\|_{L^2_tL^\infty_x} \\
     & \lesssim \|\psi\|_{\mathcal X^s}^3,
 \end{align*}
 where $\tau:\{1,2,3\}\rightarrow \{1,2,3\}$ are permutations. Thus we have $\Psi_\pm(\psi)\in \mathcal X^s$.

 Now we let $\psi\in B_\delta$. Then the above estimate gives
 \[
\|\Psi_\pm(\psi)\|_{\mathcal X^s} \le C \|\Psi_\pm \psi_0\|_{H^s_x}+ C\|\psi\|_{\mathcal X^s}^3 + C \epsilon\|\psi\|_{\mathcal X^s} \le C\eta + C\delta^3 + C\epsilon\delta.
 \]
For a sufficiently small $\eta>0$, we choose $\delta<\frac1{C}$ and we have $\Psi_\pm(\psi)\in B_\delta$. Then an application of the argument to the difference $\psi-\psi'$ allows us to use the contraction principle, and hence we can obtain the fixed point so that $\Psi_\pm(\psi)=\psi$, which is the desired solution.

Now we are concerned with the scattering. Let $\psi$ be the solution so that $\Pi_\pm^M\psi\in\mathcal X^s$. It suffices to prove that $\Pi_\pm^M\psi(t)$ is a Cauchy-sequence in $H^s$ as $t\to \infty$, i.e.\ the quantity
 \begin{align*}
     \|\Pi_\pm^M\psi(t)-\Pi_\pm^M\psi(t')\|_{H^s_x}
 \end{align*}
 is arbitrarily small provided that $t>t'>T$ for some large $T>0$. We first write
 \begin{align*}
     \Pi_\pm^M\psi = \widetilde K^\pm[\Pi_\pm^M((\psi^\dagger\psi)\psi)]+\phi^\pm,
 \end{align*}
 where $\phi^\pm$ is an error of the parametrix. Then we use Theorem \ref{thm-strichartz} to get
 \begin{align*}
      \|\Pi_\pm^M\psi(t)-\Pi_\pm^M\psi(t')\|_{H^s_x} & \lesssim \|\Pi_\pm^M[(\psi^\dagger\psi)\psi]\|_{L^1_{[t',\infty)}H^s_x} \\
      & \lesssim \|(\psi^\dagger\psi)\psi]\|_{L^1_{[t',\infty)}H^s_x} \\
      & \lesssim \|\psi\|_{L^2_{[t',\infty)}L^\infty_x}^2\|\psi\|_{L^\infty_tH^s_x}.
 \end{align*}
 Since $\psi$ is the solution, we have $\|\psi\|_{L^2_tL^\infty_x}\le \|\Pi_+^M\psi\|_{L^2_tL^\infty_x}+\|\Pi_-^M\psi\|_{L^2_tL^\infty_x}<\infty$, and hence we can choose a large $T>0$ so that the quantity $\|\psi\|_{L^2_{[t',\infty)}L^\infty_x}$ is arbitrarily small for any $t'>T$. This completes the proof of Theorem \ref{thm-dirac-reform}. 
\section{Microlocal analysis and phase space transforms}\label{sec:micro}
This section is devoted to a brief introduction of the phase space transforms and preliminaries for microlocal analysis, which will be used for the outgoing parametrix construction later. We refer the readers to \cite{tataru3} for an expository introduction and also Section 6,7 of \cite{tataru4} and Section 9,10 of \cite{metataru}.

Throughout this section, we often use functions in the Schwartz class $\mathscr S(\mathbb R^d)$, which is defined by the set of all smooth functions $f$ satisfying the inequality $\sup_{x}|x^\alpha\partial^\beta_xf(x)|<\infty  $.
We also use Schwartz functions $u\in \mathscr S(T^*\mathbb R^d)$ defined on the phase space. Due to a suitable diffeomorphism between $T^*\mathbb R^d$ and $\mathbb R^d\times\mathbb R^d$, we abuse notation and simply write $T^*\mathbb R^d=\mathbb R^d\times\mathbb R^d$.
\subsection{Bargman transforms} We begin with the introduction of the Bargman transform, which is a phase space transform.
\begin{defn}
	The Bargman transform $T$ of a tempered distribution $f\in\mathscr S'(\mathbb R^d)$ is defined by 
	\begin{align}
	Tf(x,\xi) = c_d \int_{\mathbb R^d}e^{-\frac12|x-y|^2}e^{i\xi\cdot(x-y)}f(y)\,dy,	
	\end{align}
where $c_d=2^{-\frac{d}{2}}\pi^{-\frac34d}$, as in \cite{tataru3}.
\end{defn}
We refer the readers to \cite{tataru3} for several properties of the Bargman transform $T$ in details. Here we only remark that the transform $T$ is unitary, i.e., $T^*T=I$. 
Given a tempered distribution $a\in\mathscr S'(\mathbb R^d\times\mathbb R^d)$, we define the corresponding Weyl operator $a^w$ as an operator mapping $a^w:\mathscr S(\mathbb R^d)\rightarrow\mathscr S'(\mathbb R^d)$ given by
\begin{align*}
	a^w(x,D_x)f(x) = \int a\left(\frac{x+y}{2},\xi\right)e^{i\xi\cdot(x-y)}\,d\xi.
\end{align*} 
We also introduce the distance function $d:T^*\mathbb R^d\times T^*\mathbb R^d\rightarrow \mathbb R^+$ on the phase space defined by
\begin{align*}
	d((x_1,\xi_1),(x_2,\xi_2))^2 = |x_1-x_2|^2+|\xi_1-\xi_2|^2. 
\end{align*}
We define the symbol classes $S^m_{\rho,\delta}$ adapted to the literature \cite{hoermander}.
\begin{defn}
If $m$ is a real number then $S^m_{\rho,\delta}(\mathbb R^d\times\mathbb R^d)$ is the set of all $a\in C^\infty(\mathbb R^d\times\mathbb R^d)$ such that for all multi-indices $\alpha,\beta\in\mathbb N^d$ the function $a$ satisfies the bound
\begin{align*}
	|\partial^\alpha_x\partial^\beta_\xi a(x,\xi)| \le c_{\alpha\beta}(1+|\xi|)^{m-\rho|\beta|+\delta|\alpha|}, \quad x,\xi\in\mathbb R^d,
\end{align*}	
where $0\le\rho,\delta\le1$. 
We also denote by $OPS^{m}_{\rho,\delta}$ the corresponding class of symbols in $S^m_{\rho,\delta}$.
\end{defn}
Given a pseudodifferential operator in the Weyl calculus $A^w\in OPS^0_{00}$, we define its phase space image to be the conjugated operator $TAT^*$. The kernel of the phase space image is called the phase space kernel of the operator $A^w$. 
In order to go beyond a fixed-time scale, we introduce the rescaled version $T_{\frac1s}$ of the Bargman transform $T$ so called the FBI transform:
\begin{align*}
	T_{\frac1s}u(t,x,\xi) = c_ds^{-\frac d4}\int_{\mathbb R^d}e^{-\frac{|x-y|^2}{2s}} e^{i\xi\cdot(x-y)}u(t,y)\,dy.
\end{align*}
As the Bargman transforms hold $T^*T=I$, we still have $T^*_{\frac1s}T_{\frac1s}=I$. 
Corresponding to this scaling, the distance function $d$ is now rescaled in $t^{\frac12}\times t^{-\frac12}$-scale as 
\begin{align*}
	d_t((x,\xi),(y,\eta))^2 = t^{-1}|x-y|^2+t|\xi-\eta|^2.
\end{align*}
\subsection{A long time phase space parametrix}
As \cite{metataru,tataru4} we are concerned with the following global-in-time evolution equation:
\begin{align}\label{global-evol-eq}
(D_t+a^w(t,x,D_x)-ib^w(t,x,D_x)+c^w(t,x,D_x))u = 0, \ t>0	,
\end{align}
with the symbols $a,b,c$ on $\mathbb R^+\times T^*\mathbb R^d$. In other words, we allow the symbols $a,b,c$ equipped with time-dependent scales. 
\begin{defn}\label{def-symbol-k}
We define the symbol class $S^{(k)}_{}$ to be the set of all symbols on $T^*\mathbb R^d$ satisfying the bound
\begin{align*}
	|\partial^\alpha_x\partial^\beta_\xi a(x,\xi)| \le c_{\alpha\beta}, \quad |\alpha|+|\beta|\ge k.
\end{align*}
Furthermore, we define time-dependent pseudodifferential operators in $\mathbb R\times T^*\mathbb R^d$ with the corresponding symbol classes $\ell^1S^{(k)}_{}$ whose symbols satisfy the bound
\begin{align*}
	\sum_{j \in \mathbb Z }2^{j(\frac{|\alpha|-|\beta|}{2}+1)}\|\partial^\alpha_x\partial^\beta_\xi a(t,x,\xi)\|_{L^\infty(\{t\approx 2^j\})} \le c_{\alpha\beta}, \quad |\alpha|+|\beta|\ge k.
\end{align*}
We denote by $\ell^1 S^{(2)}_{\epsilon}$ the subclass of $\ell^1 S^{(k)}_{}$ whose seminorms are $O(\epsilon)$ when $|\alpha|+|\beta|=2$.
\end{defn}
We consider the evolution equations of the form \eqref{global-evol-eq} with real symbols $a\in \ell^1S^{(2)}_{\epsilon}$ and $b\in \ell^1S^{(1)}_{}$ with $b\ge0$ and a possibly complex symbol $c\in\ell^1 S^{(0)}_{}$. In what follows, the symbol $a$ will derive the evolution, while the symbol $b$ play a role as a damping term and $c$ can be viewed as a bounded error. We would like to mention further that the symbols $a,b,c$ here are not frequency-localised. 
We consider the Hamilton flow of the evolution $D_t+a^w$ and we denote its evolutions by $\chi(t,s)$. We use the notation $t\mapsto (x_t,\xi_t)$ to denote the trajectories of the flow. 
We define the function $\psi$ by the integration of the symbol $b$ along the Hamilton flow:
\begin{align*}
	\psi(t,x_t,\xi_t) = \int_1^t b(\sigma,x_\sigma,\xi_\sigma)\,d\sigma.
\end{align*}
We set the lower limit of the integration above to be $1$. In our practical analysis, we only consider the differences $\psi(t,x_t,\xi_t)-\psi(s,x_s,\xi_s)$. 
We end this section with the pointwise bound for the phase space kernels of the evolution operator $S(t,s)$. We refer the readers to \cite{tataru4} for the proof.
\begin{thm}[Theorem 31 of \cite{tataru4}]\label{kernel-bdd}
	Let $a\in\ell^1S^{(2)}_{\epsilon},\, b\in \ell^1S^{(1)}_{}$ be real symbols with $b\ge0$ and $c\in\ell^1 S^{(0)}_{}$. Then for $0<s<t$ the kernels $K(t,s)$ of the conjugated operators $T_{\frac1t}S(t,s)T^*_{\frac1s}$ satisfy the bound
	\begin{align}
	\begin{aligned}
		|K(t,x,\xi_t,s,x_s,\xi)| \lesssim \left( \frac ts \right)^{\frac d4} \left( 1+(\psi(x_t,\xi_t)-\psi(x_s,\xi_s))^2+t^{-1}|x-x_t|^2 + s|\xi-\xi_s|^2 \right)^{-N},
	\end{aligned}
	\end{align}
	for any $N>0$.
\end{thm}

\section{The parametrix construction}\label{sec:para}
We finally arrive at the proof of the properties of the outgoing parametrix $\mathcal K^\pm_{(k)}$ for the operators $D_t+\mathcal A^\pm_{(k)}$, Proposition \ref{prop-parametrix}.
The principal result of this section is the dispersive inequality \eqref{pointwise-decay}.

Recall that Proposition \ref{prop-mollified-operator} allows us to study the globally mollified operator $\widetilde{A}^\pm$ instead of $A^\pm$. Each dyadic piece $A^\pm_{(k)}$ of the operator $\widetilde{A}^\pm$ satisfies frequency localisation up to an acceptable error. After the scaling arugment $(t,x)\to (\lambda^{-1}t,\lambda^{-1}x) $ done in Section \ref{sec:scaling}, the analysis of the localised operator $A^\pm_{(k)}$ has been reduced to $\mathcal A^\pm_{(k)}$, which is localised at the unit frequency scale $|\xi|\approx1$, where we choose $\lambda=2^k$. However, we keep a parameter dependence resulting from the lack of a precise scale-invariance.

\subsection{Reduction to Weyl quantization}
We will construct the phase-space parametrix of the scaled operator $D_t+\mathcal A^\pm_{(k)}$ in the Weyl calculus as \cite{metataru,tataru4}. Thus we need to replace the standard quantization by the Weyl quantization. More precisely, for each $\lambda=2^k$, we put $\mathcal A^+_{(k)}=\mathcal A_{(k)}=\mathrm{Op}_y(a_{(k)}(s,y,D_y))$. Then we have
\begin{align}
    \mathrm{Op}_y(a_{(k)}(s,y,\eta)) = a^w_{(k)}(s,y,D_y) +R_k(s,y,D_y),
\end{align}
where $a^w_{(k)}(s,y,D_y)$ is the Weyl-Operator with the same symbol and $R_k$ is a lower-order operator. In particular, by the change of quantization formula, 
\begin{align}
    R_k = \mathrm{Op}_y(r_k(s,y,\eta)), 
\end{align}
where
\begin{align}\label{eq-change-quantization}
    r_k = \frac{i}{2} \partial_y\cdot\partial_\eta a_{(k)}+ \sum_{|\alpha|\ge2} \frac1{\alpha!} \left( \frac i2\right)^{|\alpha|} \partial_y^\alpha\partial_\eta^\alpha a_{(k)}.
\end{align}
We refer to Theorem 18.5.10. of \cite{hoermander}.
Then the bound \eqref{bdd-mollification} yields 
\begin{align}
    |r_k| \lesssim \epsilon_k (1+|y|)^{-1},
\end{align}
which ensures $R_k = \mathrm{Op}^w(\frac i2 \partial_y\cdot\partial_\eta a_k) +\mathrm{Op}^{w}S^{-1} $. Thus $R_k$ is an acceptable lower-order error terms. We also remark that the principal symbols of $\mathrm{Op}(a_{(k)}) $ and $a_{(k)}^w$ are identical.




\subsection{A perturbation of the half Klein-Gordon operators}

Now we shall construct a parametrix for the scaled operator $D_t+\mathcal A^+_{(k)}$. We let $\mathcal S^\pm_{(k)}(t,s)$ denote the $L^2$-evolution operator generated by $D_t+\mathcal A^+_{(k)}$. We only treat the case $k\geq 0$ explicitly, the case $k<0$ is in fact easier. 

Due to the frequency-localisation of the operator $\mathcal A^+_{(k)}$,
we observe that
\begin{align}\label{freq-local-evol-op}
	\mathcal S^\pm_{(k)}(t,s)S_{2^{-10}\le\cdot\le2^{10}} = S_{2^{-10}\le\cdot\le2^{10}}\mathcal S^\pm_{(k)}(t,s) = S_{2^{-10}\le\cdot\le2^{10}}\mathcal S^\pm_{(k)}(t,s)S_{2^{-10}\le\cdot\le2^{10}}.
\end{align}
We refer the readers to Section 10 of \cite{metataru}.
\subsection{The parametrix construction}
We partition the annulus $|\xi|\approx1$ in the phase space
$$
\rho_{-1}(\xi)+ \rho_0(\xi)+\rho_{1}(\xi) =  \sum_\pm \sum_{j\ge0}p_{j}^\pm(x,\xi),
$$
with
\begin{align*}
\textrm{supp }p_{j}^\pm \subset \left\{ 2^{j-1}<|x|<2^{j+1},  \pm\frac{x}{|x|}\cdot\frac{\xi}{|\xi|} \ge -2^{-5} \right\} 	, \quad j\ge1 \\
\textrm{supp }p_{0}^\pm \subset \left\{ |x|<2,  \pm\frac{x}{|x|}\cdot\frac{\xi}{|\xi|} \ge -2^{-5} \right\}.
\end{align*}
We denote the corresponding Weyl operator by $\mathcal P^\pm_{j}(x,D_x)$. \\

Now we construct evolution operators $\mathcal S^\pm_{(k)}(t,s)$ associated to a the damped half Klein-Gordon operator $\mathcal{A}^+_{(k)}$, localised on the annulus $\frac12<|\xi|<2$, for $\lambda=2^k$. Then,  we form a parametrix $\mathcal K^+_{(k)}$ for the rescaled operator $D_t+\mathcal A^+_{(k)}$ by setting
\begin{align*}
\mathcal K^+_{(k)}(t,s) = \begin{cases}
 	\sum_{j=0}^\infty \mathcal S^-_{(k)}(t,s)\mathcal P^-_{j}(x,D_x), \quad t<s, \\
 	\sum_{j=0}^\infty \mathcal S^+_{(k)}(t,s)\mathcal P^+_{j}(x,D_x), \quad t>s.
 \end{cases}	
\end{align*} 
Now Proposition \ref{prop-parametrix} follows directly from the following theorem.
\begin{thm}\label{main-thm-parametrix}
For each $s\in\mathbb R$, there is an outgoing evolution operator $\mathcal S^+_{(k)}$ for the half-Klein-Gordon operator $D_t+\mathcal A^+_{(k)}$ in $\{t>s\}$, localised on the annulus $|\xi|\approx1$ in the frequency space and satisfies the following: for $0\le\theta\le1$,
\begin{enumerate}
	\item $L^2$-bound:
	$$
	\|\mathcal S^+_{(k)}(t,s)\|_{L^2\rightarrow L^2}\lesssim1.
	$$
	\item Error estimates:
 \begin{align}\label{error-estimates}
 \begin{aligned}
 \|x^\alpha (D_t+\mathcal A^+_{(k)})\mathcal S^+_{(k)}(t,s)\mathcal P^+_{j}\|_{L^2\rightarrow L^2} \lesssim (2^j+|t-s|)^{-N}, 	\\
 \|x^\alpha D_t(D_t+\mathcal A^+_{(k)})\mathcal S^+_{(k)}(t,s)\mathcal P^+_{j}\|_{L^2\rightarrow L^2} \lesssim (2^j+|t-s|)^{-N}.
 \end{aligned}
 \end{align}
\item Initial data:
$$
\mathcal S^+_{(k)}(s+0,s) = \mathbf 1_{L^2},
$$
where $\mathbf1_{L^2}$ is the identity operator on $L^2$.
\item Outgoing parametrix:\begin{align}\label{outgoing-parametrix}
\| \mathbf 1_{\{|x|<2^{-10}(|t-s|+2^j)\}}\mathcal S^+_{(k)}(t,s)\mathcal P^+_{j}\|_{L^2\rightarrow L^2} \lesssim (2^j+|t-s|)^{-N}.
\end{align}
\item Finite speed:\begin{align}\label{finite-speed}
\| \mathbf 1_{\{|x|>2^{10}(|t-s|+2^j)\}}\mathcal S^+_{(k)}(t,s)\mathcal P_{j}\|_{L^2\rightarrow L^2} \lesssim  (2^j+|t-s|)^{-N}.
\end{align}
\item Frequency localisation:
\begin{align}\label{frequency-localiseation}
\|(\mathbf 1_{L^2}-S_{\frac{1}{16}\le\cdot\le16})\mathcal S^+_{(k)}(t,s)\mathcal P^+_{j}\|_{L^2\rightarrow L^2} \lesssim (2^j+|t-s|)^{-N}.
\end{align}
\item Pointwise decay:
\begin{align}\label{pointwise-decay}
\|\mathcal S^+_{(k)}(t,s)\mathcal P^+_{j}\|_{L^1\rightarrow L^\infty} \lesssim  \lambda^{\theta}(1+|t-s|)^{-\frac{d-1+\theta}{2}}.
\end{align}
\end{enumerate}	
\end{thm}
Theorem \ref{main-thm-parametrix} implies Proposition \ref{prop-parametrix}.
Its proof relies on the work by \cite{metataru}. We obtain the kernel bound using Theorem \ref{thm-bdd-kernel} and exploit the full rank of the Jacobian matrix of the flow map $\xi_t\to x_t$  to establish the time decay $\approx t^{-\frac d2}$.


In the remainder of this section, we concentrate on the proof of Theorem \ref{main-thm-parametrix}. 
We follow the argument by \cite{metataru,tataru4} to prove Theorem \ref{main-thm-parametrix}. In this subsection, 
we suppress $\mathcal S^+_{(k)}=\mathcal S_{(k)}$ and so on.
At first we shall establish bounds on the kernel of $\mathfrak S_{(k)}(t,s)$, given by
\begin{align*}
	\mathfrak S_{(k)}(t,s) = T_{\frac1t}\mathcal S_{(k)}(t,s)P_{\frac{1}{4}\le\cdot\le4}T^*_{\frac1s},
\end{align*}
using Theorem \ref{kernel-bdd}.
To do this, we note that  $\mathcal A_{(k)}-\sqrt{\lambda^{-2}+|\xi|^2}\in \ell^1 S^{(2)}_\epsilon  $, where $\lambda=2^k$. Thus in view of the work by \cite{metataru,tataru4}, we are now concerned with the evolution equation:
\begin{align}\label{evol-eq}
   \left( D_t + \langle D_x\rangle_\lambda + a^w_\lambda(t,x,D_x) -ib^w_\lambda(t,x,D_x) \right) u = 0,
\end{align}
where $\langle D_x\rangle=\sqrt{\lambda^{-2}+|D_x|^2 }$ and $a_\lambda \in \ell^1S^{(2)}_\epsilon$. Here $b_\lambda \in \ell^1 S^{(1)}$ is an additional damping term which will be determined in Section \ref{sec:construction-damping}. Since $\mathcal A_{(k)}$ is localised in the unit frequency up to an acceptable error, we also assume that the operators $a^w_\lambda$ and $b^w_\lambda$ are localised in the unit frequency.

We let $\chi_\lambda(t,s)$ denote the Hamilton flow associated to the evolution $D_t+\langle D_x\rangle_\lambda+ a^w_\lambda$. At each time $t$ we consider the symplectic map $\mu$ defined by
\begin{align*}
	\mu_t(x,\xi) = (x+t\xi\langle\xi\rangle_\lambda^{-1},\xi),
\end{align*} 
which corresponds to the Hamilton flow for the scaled half Klein-Gordon operator $D_t+\langle D_x\rangle_\lambda$ on the flat space. This extends to a space-time symplectic map 
\begin{align*}
	\mu(t,\tau,x,\xi) = (t,\tau-\langle\xi\rangle_\lambda, x+t\xi\langle\xi\rangle_\lambda^{-1},\xi).
\end{align*}
We let $\mathfrak p_\lambda$ be the symbol defined by 
\begin{align*}
\mathfrak p_\lambda(t,\tau,x,\xi) = \tau+\langle\xi\rangle_\lambda+ a_\lambda(t,x,\xi),
\end{align*}
then 
\begin{align*}
 \mathfrak p_\lambda(\mu(t,\tau,x,\xi)) = \tau+ a_\lambda(t,x+t\xi\langle\xi\rangle_\lambda^{-1},\xi).
\end{align*}
Hence the conjugate of the Hamilton flow $\chi_\lambda(t,s)$ for $\tau+\langle\xi\rangle_\lambda+ a_\lambda$ with respect to $\mu_t$ is the Hamilton flow $\mathscr X(t,s)$ for $\tau+ a_\lambda(t,x+t\xi\langle\xi\rangle_\lambda^{-1},\xi)$,
\begin{align*}
	\chi_\lambda(t,s) = \mu_t \circ\mathscr X(t,s)\circ\mu_s^{-1} .
\end{align*} 
 Now we let $\lambda=2^k$, $\mu=\lambda^{-2}$. For a moment, we put
\begin{align}
    \begin{aligned}
        a_{(k)}(t,x,\xi) := \sqrt{\mu+|\xi|^2}+ \lambda^{-1}q_{(k)}(\lambda^{-1}t,\lambda^{-1}x,\lambda\xi), \\
        c_{(k)}(t,x,\xi) := \sqrt{\mu+|\xi|^2}^{-1}+ \lambda r_{(k)}(\lambda^{-1}t,\lambda^{-1}x,\lambda\xi),
    \end{aligned}
\end{align}
where $r_{(k)}$ is the mollification of $r(t,x,\xi):=\sqrt{1+g^{ij}\xi_i\xi_j}^{-1}-\sqrt{1+|\xi|^2}^{-1}$.
Since the differential operator $\xi\cdot\partial_\xi$ commutes with the mollification, we have
\begin{align}
    (\xi\cdot\partial_\xi-1)a_{(k)} = -\mu c_{(k)}.
\end{align}
Using the above identity, we have the following identities for the symbols:
\begin{align}\label{eq:jac-euler}
    \begin{aligned}
       ( \partial_\xi^2  a_{(k)})\xi = -\mu \partial_\xi c_{(k)}, \quad (\partial^2_{\xi x}a_{(k)})\xi-\partial_x a_{(k)} = -\mu \partial_x c_{(k)}.
    \end{aligned}
\end{align}
In what follows, we put $a:= a_{(k)} $ and $c:=c_{(k)}$. By the mollification \eqref{bdd-mollification}, on the outgoing region $|\xi_\tau|\approx1$, $1+|x_\tau|\approx 1+\tau$, we have
\begin{align}\label{eq:jac-coefficients}
    \begin{aligned}
        |\partial_\xi^\beta(a-\sqrt{\mu+|\xi|^2})| \lesssim \epsilon_k(\tau), \quad |\beta|\le2, \quad |\partial_xa|+|\partial_{x\xi}^2a|+|\partial_xc| \lesssim \frac{\epsilon_k(\tau)}{1+\tau}, \\
        |\partial_x^2a| \lesssim \frac{\epsilon_k(\tau)}{(1+\tau)^2}, \quad |\partial_\xi( c- \sqrt{\mu+|\xi|^2}^{-1})| \lesssim \epsilon_k(\tau),  
    \end{aligned}
\end{align}
where
\begin{align*}
    \sup_k \epsilon_k(\tau)+\int_0^\infty \frac{\epsilon_k(\tau)}{1+\tau} \,d\tau \lesssim \epsilon.
\end{align*}
We provide a version of \cite[Proposition 3.2.7]{xue}.
\begin{prop}\label{lipschitz-con}
Let the mollified symbols $a$ and $c$ satisfy \eqref{eq:jac-euler} and
\eqref{eq:jac-coefficients} on the outgoing region described above,
and let $t>s\ge1$. If $\epsilon$ is sufficiently small, the Hamilton flow
$\chi_\lambda(t,s)$ generated by $a$ admits local mixed coordinates
$(x_s,\eta)$, where $\eta=\xi_t$.
We put
\[
 \Phi^d_{KG}(\eta)
 :=\partial_\eta^2\sqrt{\mu+|\eta|^2}
 =\sqrt{\mu+|\eta|^2}^{-3}\bigl((\mu+|\eta|^2)I_d-\eta\otimes\eta\bigr),
\]
we have
\begin{equation}\label{eq:jac-mixed-flow}
 \frac{\partial(x_t,\xi_s)}{\partial(x_s,\xi_t)}
 =\begin{bmatrix}
 I_d+O(\epsilon)& (t-s)\Phi^d_{KG}(\eta)+\mathcal R\\
 O\!\left(\dfrac{\epsilon}{1+s}\right)&I_d+O(\epsilon)
 \end{bmatrix},
\end{equation}
where
\begin{equation}\label{eq:jac-refined-remainder}
 \mathcal R=\mathcal R^T,\qquad
 \|\mathcal R\|\lesssim\epsilon(t-s),\qquad
 |\mathcal R\eta|\lesssim\epsilon\mu(t-s).
\end{equation}
\end{prop}
\begin{proof}
  For simplicity, we put
\begin{align*}
h_\mu(\tau,x,\xi) = \sqrt{ \mu + \xi^T G(\tau,x)\xi }, \quad G = [g^{ij}].
\end{align*}
Then we have
\begin{align*}
    \partial_\xi^2 h_{\mu} = \frac{G}{h_\mu}-\frac{(G\xi)(G\xi)^T}{h_\mu^3}, \quad  \partial_\xi^2 h_{\mu} \cdot \xi = \mu \frac{G\xi}{h_\mu^3}.
\end{align*}
We fix the initial position $x_s$ and let $\xi_t=r\eta$. We let $    Z_\tau = \partial_r x_{\tau}|_{r=1}, \ W_{\tau} = \partial_r \xi_{\tau}|_{r=1} - \xi_\tau$.
Then  $  Z_s = 0, \ W_t = 0, \ Z_t = \frac{\partial x_t}{\partial \xi_t}\eta$.
By differentiating the Hamilton equation, we obtain
\begin{align*}
        \dot Z = \partial_\xi\partial_x h_\mu Z+ \partial_\xi^2h_\mu W + \mu \frac{G\xi}{h_\mu^3}, \quad 
        \dot W = -\partial_x^2h_{\mu}Z - \partial_x\partial_\xi h_\mu W -\mu \frac{\partial_x h_\mu}{h_\mu^2}.
\end{align*}

\begin{lem}\label{lem-jacobian-dfrac}
    We have
    \begin{align}
        \begin{aligned}
            \frac{\partial x_t}{\partial\xi_t}\eta = (t-s) \frac{\mu\eta}{(\mu+|\eta|^2)^\frac32}+ O(\epsilon\mu(t-s)). 
        \end{aligned}
    \end{align}
\end{lem}
We note that the matrix $J:=\dfrac{\partial x_t}{\partial\xi_t}$ is symmetric. Then the matrix $J$ can be represented as follows:
\begin{align*}
   \frac{J}{t-s}  \begin{bmatrix}
       \sqrt{\mu+|\eta|^2}^{-1} I_{d-1}+O(\epsilon) & O(\epsilon\mu) \\ 
       O(\epsilon \mu) & \mu( \sqrt{\mu+|\eta|^2}^{-3}+ O(\epsilon) ). 
   \end{bmatrix}
\end{align*}
Therefore the determinant of $J$ is $\lesssim (t-s)^d \lambda^{-2}$.

From now on we give the proof of Lemma \ref{lem-jacobian-dfrac} using \eqref{eq:jac-coefficients}.
From the Hamilton equation, we have
\begin{align*}
        Z_\tau = \int_s^{\tau} ( \partial_\xi\partial_x h_\mu Z+ \partial_\xi^2h_\mu W + \mu \frac{G\xi}{h_\mu^3} )(\sigma)\,d\sigma , \quad        W_\tau = \int_{\tau}^t  (\partial_x^2h_{\mu}Z + \partial_x\partial_\xi h_\mu W +\mu \frac{\partial_x h_\mu}{h_\mu^2})(\sigma)\,d\sigma
\end{align*}
Now we set
\begin{align*}
    C_Z := \sup_{s < \tau \le t} \frac{|Z_\tau|}{\mu(\tau-s)}, \quad C_W := \sup_{s<\tau\le t}\frac{|W_\tau|}{\mu}. 
\end{align*}
By \eqref{eq:jac-coefficients}, we see that
\begin{align*}
    \frac{|Z_\tau|}{\mu(\tau-s)} & \lesssim \frac{C_Z}{\tau-s}\int \frac{e(\sigma)}{1+\sigma}(\sigma-s)\,d\sigma + C_W+1  \lesssim \epsilon C_Z+ C_W+1.
\end{align*}
On the other hand,
\begin{align*}
    \frac{|W_\tau|}{\mu} & \lesssim C_Z \int_\tau^t \frac{e(\sigma)(\sigma-s)}{(1+\sigma)^2}\,d\sigma + (C_W+1) \int^t_\tau \frac{e(\sigma)}{1+\sigma}\,d\sigma  \lesssim \epsilon( C_Z+C_W+1).
\end{align*}
For a sufficiently small $\epsilon>0$, we have $C_W \lesssim \epsilon (C_Z+1)$, and hence we have    $C_Z \lesssim 1$, and $C_W \lesssim \epsilon$. Therefore we have $|Z_\tau| \lesssim \mu( \tau-s)$ and $|W_\tau| \lesssim \epsilon\mu$.

We also need to compare the flow $\xi_t$ with the flow associated with the flat space. Indeed, \eqref{eq:jac-coefficients} yields
\begin{align*}
    |\xi_\tau-\eta| = \left| \int_\tau^t \partial_x h_{\mu}(\sigma)\,d\sigma \right| \lesssim \epsilon.
\end{align*}
For simplicity we put $f:= \frac{G \xi}{h_\mu^3} $. By $G=I+O(\epsilon)$ and $|\eta|\approx1$, we obviously have
\begin{align*}
    f(\tau) = \frac{G\xi_\tau}{h_\mu(\tau)^3} = \frac{\eta}{(\mu+|\eta|^2)^\frac32 } + O(\epsilon).
\end{align*}
Then we write
\begin{align*}
    Z_t - \mu(t-s) \frac{\eta}{(\mu+|\eta|^2)^\frac32} = \int_s^t \partial_{\xi x}^2h_{\mu} Z \,d\tau + \int_s^t \partial_\xi^2 h_\mu W\,d\tau + \mu \int_s^t f - \frac{\eta}{(\mu+|\eta|^2)^\frac32}\,d\tau.
\end{align*}
Using \eqref{eq:jac-coefficients}, we see that
\begin{align*}
  \left| \int_s^t \partial_{\xi x}^2h_{\mu} Z \,d\tau \right|  \lesssim \mu \int_s^t \frac{e(\tau)}{1+\tau}(\tau-s)\,d\tau \lesssim \epsilon \mu(t-s), \\
  \left| \int_s^t \partial_\xi^2 h_\mu W\,d\tau \right| + \left|  \mu \int_s^t f - \frac{\eta}{(\mu+|\eta|^2)^\frac32}\,d\tau \right| \lesssim \epsilon\mu(t-s).
\end{align*}
\end{proof}

We pay a special attention to the Jacobian matrix $\left[\dfrac{\partial x_t}{\partial\xi_t}\right]$. After a straightforward computation, we obtain the following:
\begin{lem}\label{pro-jacobian-xi}
Let $d\ge3$ and $\xi\in\mathbb R^d$. Suppose that the $d\times d$ matrix $\Phi^d_{KG}(\xi)$ is given by
\begin{align*}
    \Phi^d_{KG}(\xi) = \langle\xi\rangle_\lambda^{-3}(\langle\xi\rangle_\lambda ^2\mathbf I_d-\xi\otimes\xi),
\end{align*}
where the matrix components of the tensor product $\xi\otimes\xi$ are given by $[\xi\otimes\xi]_{ij}=\xi_i\xi_j$. Then for every non-zero $\xi\in\mathbb R^d$ the matrix $\Phi^d_{KG}(\xi)$ has full rank and the eigenvalues $\mu_1,\cdots,\mu_d$ of the matrix are
\begin{align*}
    \mu_1=\cdots=\mu_{d-1} = \langle\xi\rangle_\lambda^{-1}, \ \mu_d = \lambda^{-2}\langle\xi\rangle_\lambda^{-3}.
\end{align*}
\end{lem}

Recall that  $|\xi|\approx1$. Then
\begin{align*}
    \mu_1 \approx \cdots \approx \mu_{d-1} \approx \frac{1}{2^{-2k}+1} \approx1,
\end{align*}
and
\begin{align*}
    \mu_d \approx \frac{2^{-2k}}{(2^{-2k}+1)^{\frac32}} \approx \frac{2^{-2k}}{2^{-3k}(1+2^{2k})^{\frac32}} \approx 2^{-2k}.
\end{align*}
\begin{rem}\label{rmk:smallfreq}
    On the other hand, in the case $k<0$, which we do not discuss in detail, we simply have $\mu_j\approx1$, $j=1,\cdots,d$, which is the only essential change in this case.
\end{rem}
\begin{rem}\label{rem-jacobian-wave}
It is remarkable to compare the results above with the similar results for the wave operator. We define the $d\times d$ matrix $\Phi^d_{\rm w}(\xi)$ by
\begin{align*}
	\Phi^d_{\rm w}(\xi) = |\xi|^2\mathbf I_d-\xi\otimes\xi. 
\end{align*}
Then the matrix $\Phi^d_{\rm w}(\xi)$ has the maximal rank $d-1$ for non-zero $\xi$.
\end{rem}

As in the previous section we define the integral of the damping term
\begin{align*}
	\Psi_\lambda(t,x_t,\xi_t) = \int_1^t  b_\lambda(\sigma,x_\sigma,\xi_\sigma)\,d\sigma,
\end{align*}
along the $\chi_\lambda$-flow. Then the integral $\Psi_\lambda$ is the $\mu$-conjugate of the integral $\psi$ of $b(t,x,\xi)=b_\lambda(t,x+t\xi\langle\xi\rangle^{-1}_\lambda,\xi)$ along the $\mathscr X$-flow. 
\begin{prop}\label{prop-jacobi-b}
If $ a_\lambda\in\ell^1S^{(2)}_{\epsilon}$ with a sufficiently small $\epsilon>0$ and $ b_\lambda\in\ell^1S^{(1)}_{}$ and $0<s<t$, then we have
\begin{align}
\frac{\partial(\Psi_\lambda(x_t,\xi_t)-\Psi_\lambda(x_s,\xi_s))}{\partial(x_s,\xi_t)} = \begin{bmatrix}
	O(s^{-\frac12}) & O(t^\frac12)
\end{bmatrix}	.
\end{align}
\end{prop}
For $x=(x_1,\cdots,x_d)\in\mathbb R^d$ and $\delta>0$ we let $\|\cdot\|_{\delta}$ denote the norm defined by $\|x\|_\delta = \sqrt{x_1^2+\cdots+x_{d-1}^2+\delta^2 x_d^2}$.
Then we can establish the pointwise bounds for the kernels of the operators $\mathfrak S_{(k)}(t,s)$. 
\begin{thm}\label{thm-bdd-kernel}
	Consider the evolution equation:
	\begin{align*}
		(D_t +\langle D_x\rangle_\lambda+a^w_\lambda(t,x,D_x)-ib^w_\lambda(t,x,D_x))u = 0,
	\end{align*}
	where the symbols $a_\lambda\in\ell^1S^{(2)}_{\epsilon}$ and $b_\lambda\in\ell^1S^{(1)}_{}$ are real with a sufficiently small $\epsilon>0$. We assume that the symbols $a_\lambda$ and $b_\lambda$ are frequency-localised at $|\xi|\approx1$. Let $t>s$ and define the frequency-localised $L^2$-evolution operators of the equation \eqref{evol-eq} by $\mathcal S_{(k)}(t,s)$.
    Then, the kernels $\mathfrak K_{(k)}$ of the phase space image $\mathfrak S_{(k)}$ of the operator $\mathcal S_{(k)}(t,s)$ satisfies the following pointwise bound:
	\begin{align}
	\begin{aligned}
		|\mathfrak K_{(k)}(t,x,\xi_t,s,x_s,\xi)| & \lesssim \left(\frac{t}{s}\right)^{-\frac d4}\left( 1+(\Psi_\lambda(x_t,\xi_t)-\Psi_\lambda(x_s,\xi_s))^2+t^{-1}\|x-x_t\|_\lambda^2+s\|\xi-\xi_s\|_{\lambda^{-1}}^2 \right)^{-N} \\
		&\qquad\qquad\qquad \times \left( 1+ t d(|\xi_s|,[1/{64},64]) \right)^{-N},
	\end{aligned}
	\end{align}
    with $\lambda=2^k$.
\end{thm}

\begin{proof}
To obtain the pointwise bound of the kernels, we follow the proof of Theorem 5 of \cite{metataru}. In other words, we exploit Theorem \ref{kernel-bdd} together with a $\chi_\lambda$-conjugation with respect to the flow for the rescaled half Klein-Gordon operator $D_t+\langle D_x\rangle_\lambda$ on the flat space. We denote the conjugated operator of the evolution operator $\mathcal S_{(k)}(t,s)$ by
\begin{align*}
	\mathscr T_{(k)}(t,s) = e^{it\langle D_x\rangle_\lambda}\mathcal S_{(k)}(t,s)e^{-is\langle D_x\rangle_\lambda}.
\end{align*}
Then a computation yields
\begin{align*}
	\frac{d}{dt}\mathscr T_{(k)}(t,s) = -i e^{it\langle D_x\rangle_\lambda}(a^w_\lambda(t,x,D_x)-ib^w_\lambda(t,x,D_x))e^{-it\langle D_x\rangle_\lambda }\mathscr T_{(k)}(t,s),
\end{align*}
and hence we deduce that the conjugated operator $\mathscr T_{(k)}(t,s)$ is the evolution operator associated to the pseudodifferential operator
\begin{align*}
	e^{it\langle D_x\rangle_\lambda}(a^w_\lambda(t,x,D_x)-ib^w_\lambda(t,x,D_x))e^{-it\langle D_x\rangle_\lambda}. 
\end{align*}
Furthermore the operator $\mathscr T_\lambda$ can be viewed as the pseudodifferential operator of the form:
\begin{align*}
	\tilde a^w(t,x,D_x) -i\tilde b^w(t,x,D_x)+\tilde c^w(t,x,D_x),
\end{align*}
where the operator $\tilde c^w$ turns out to be a remainder term such that the symbol $\tilde c$ is in $\ell^1S^{(0)}_{}$. This implies that the phase space kernel of the operator $\mathscr T_{(k)}(t,s)$ satisfies the pointwise bound as Theorem \ref{kernel-bdd}. 
Now we pay attention to our evolution equation \eqref{evol-eq}. We rewrite the phase space image $\mathfrak S_{(k)}(t,s)$ of the evolution $\mathcal S_{(k)}(t,s)$ as
\begin{align*}
	\mathfrak S_{(k)}(t,s) & = T_{\frac1t}\mathcal S_{(k)}(t,s)T_{\frac1s}^* = T_{\frac1t}e^{-it\langle D_x\rangle_\lambda}e^{it\langle D_x\rangle_\lambda}\mathcal S_{(k)}(t,s)e^{-is\langle D_x\rangle_\lambda}e^{is\langle D_x\rangle_\lambda}T_{\frac1s}^* \\
	& = T_{\frac1t}e^{-it\langle D_x\rangle_\lambda}\mathscr T_{(k)}(t,s)e^{is\langle D_x\rangle_\lambda}T_{\frac1s}^* \\
	& = T_{\frac1t}e^{-it\langle D_x\rangle_\lambda}S_{\frac{1}{64}\le\cdot\le64}\mathscr T_{(k)}(t,s)S_{\frac{1}{64}\le\cdot\le64}e^{is\langle D_x\rangle_\lambda} T_{\frac1s}^*.
\end{align*}
Here we used the frequency-localisation property \eqref{freq-local-evol-op} of the evolution $\mathcal S_{(k)}(t,s)$ to obtain the last equality. After an application of the identity $T^*T=I$ of the FBI transform we see that the phase space kernels $\mathfrak S_{(k)}(t,s)$ turns out to be the three fold-composition of the conjugated operators of the form:
\begin{align*}
	\mathfrak S_{(k)}(t,s) = \left( T_{\frac1t}e^{-it\langle D_x\rangle_\lambda}P_{\frac{1}{64}\le\cdot\le64} T^*_{\frac1t} \right)\left(T_{\frac1t} \mathscr T_{(k)}(t,s) T^*_{\frac1s} \right) \left( T_{\frac1s}S_{\frac{1}{64}\le\cdot\le64} e^{is\langle D_x\rangle_\lambda}T^*_{\frac1s} \right) .
\end{align*}
Now we observe the rapid decay of each operators using Proposition 8 of \cite{tataru4}.
In fact, the first operator $T_{\frac1t}e^{-it\langle D_x\rangle_\lambda}S_{\frac{1}{64}\le\cdot\le64} T^*_{\frac1t}$ is the conjugation of the usual forward propagator of the rescaled half-Klein-Gordon operators in the flat space-time. After a suitable change of variables, we see that its kernel satisfies a rapid decay on the $t^\frac12\times t^{-\frac12}$-scale away from the graph of the map $\mu_t$ with all the directions tangent to the characteristic hypersurface $$\Sigma_\lambda= \{(\tau,\xi)\colon \tau=\sqrt{\lambda^{-2}+|\xi|^2} \},$$ and the kernel decays rapidly on the $(t\lambda^{-2})^\frac12\times(t\lambda^{-2})^{-\frac12}$ scale away from the same graph along the normal to the surface $\Sigma_\lambda$. The kernel also decays away from the support of the projection operator $S_{\frac1{64}\le\cdot\le64}$ at the same time.\\
\noindent Meanwhile, we have already seen that the kernel of the second operator $T_{\frac1t} \mathscr T_{(k)}(t,s) T^*_{\frac1s}$ satisfies the pointwise bound as Theorem \ref{kernel-bdd}. The third operator $T_{\frac1s}S_{\frac{1}{64}\le\cdot\le64} e^{is\langle D_x\rangle_\lambda}T^*_{\frac1s}$ also turns out to be the conjugation of the usual backward Klein-Gordon propagator and hence its kernel decay rapidly in the similar way as the first operator.
In consequence, we see that the composition of the three operators replaces the Hamilton flow associated to $a$ by the Hamilton flow associated to $a_\lambda$ and the function $\psi$ with $\Psi_\lambda$ in the kernel bounds. This completes the proof of Theorem \ref{thm-bdd-kernel}.
\end{proof}

In advance to the proof of Theorem \ref{main-thm-parametrix}, we would like to remark some important facts. By the translation-invariance in the time variables we may assume that the initial time $s$ is set to be $2^j$. We observe that along the forward Hamilton flow which starts in the support of the symbols of $\mathcal P_{j}$ we have
\begin{align*}
	\frac14<|\xi|<4, \ |x| \approx t .
\end{align*}
In this regime we have $a_{(k)}(t,x,\xi)-\langle\xi\rangle_{\lambda}\in \ell^1 S^{(2)}_\epsilon$, where $a_{(k)}$ is the symbol of the operator $\mathcal A_{(k)}$.

Due to the estimates of the outgoing-parametrix \eqref{outgoing-parametrix}, the finite speed property \eqref{finite-speed}, and the frequency-localisation \eqref{frequency-localiseation}, without loss we may study the symbols $a_\lambda$ in the region $\{|x|\ll t\}$ and outside the annulus $\{ \frac{1}{32}< |\xi| <32 \}$, which gives only a negligible error as in \eqref{error-estimates}. It turns out that we only need to focus on the evolution governed by a symbol $\langle\xi\rangle_\lambda+a_\lambda(t,x,\xi)$, with $a_\lambda\in\ell^1S^{(2)}_{\epsilon}$ so that the symbol $a_\lambda$ vanishes outside the regime $\{ \frac{1}{64} <|\xi|<64 \}$ and the operator $a^w_\lambda$ is localised in the unit scale. 

We essentially have $a_\lambda(t,x,\xi)=a_{(k)}-\langle \xi\rangle_\lambda$ and by the estimates \eqref{bdd-mollification} we may assume the better regularity
\begin{align}\label{bdd-perturb-a}
\begin{aligned}
|\partial^\alpha_x\partial^\beta_\xi a_\lambda(t,x,\xi ) | & \lesssim \epsilon(t)(1+t)^{-|\alpha|}, \quad |\alpha|\le 2, \\
|\partial^\alpha_x\partial^\beta_\xi a_\lambda(t,x,\xi ) | & \lesssim \epsilon(t)(1+t)^{-1-\frac{|\alpha|}{2}}, \quad 2\le |\alpha|.
\end{aligned}	
\end{align}
The second inequality of \eqref{bdd-perturb-a} shall only be used if the time scales are too small to ensure $s^\frac12\times s^{-\frac12}$ packets at time $s$ to separate in time $t$. See also (99) of \cite{metataru}. 

However, due to the uncertainty principle, the bounds \eqref{outgoing-parametrix}, \eqref{finite-speed}, \eqref{frequency-localiseation} are inconsistent with the evolution operator $\mathcal S_{(k)}$ associated to the operator $\langle D_x\rangle_\lambda+a_\lambda$.

To remedy this problem and to establish sufficient time-decay as in \eqref{outgoing-parametrix}, \eqref{finite-speed}, \eqref{frequency-localiseation}, we make use of \textit{an artificial damping term} $b_\lambda\in\ell^1S^{(1)}_{}$, which is real and positive. In Section \ref{sec:construction-damping}, we give the explicit construction of the symbol $b_\lambda$ so that the operator $b^w_\lambda$ plays a role as a damping term, which truly resolves the aforementioned leakage, and it vanishes on the main propagation regime $\{|x|\approx t\}\cap\{|\xi|\approx1\}$ simultaneously. See also \cite{metataru,tataru4}.

In summary, we shall define the forward evolution operator $S_{(k)}(t,s)$ of the evolution equation
\begin{align*}
	(D_t+\langle D_x\rangle_\lambda+a^w_\lambda(t,x,D_x))u = ib_\lambda^w(t,x,D_x)u.
\end{align*}
We may replace the operator $\mathcal S_{(k)}(t,s)$ by the truncated operator $S_{1/64\le\cdot\le64}\mathcal S_{(k)}$ to ensure the frequency-localisation property of our parametrix. We shall prove
\begin{align}\label{frequency-localisation-re}
\begin{aligned}
	\|x^\alpha S_{\cdot<\frac{1}{16}}\mathcal S_{{(k)}}(t,s)\mathcal P_{j}\|_{L^2\rightarrow L^2} & \lesssim (2^j+|t-s|)^{-N}, \\
	\|x^\alpha \partial_x^\beta S_{\cdot>16}\mathcal S_{(k)}(t,s)\mathcal P_{j}\|_{L^2\rightarrow L^2} & \lesssim (2^j+|t-s|)^{-N},
\end{aligned}	
\end{align}
which implies that the error from the truncation can be absorbed into other bounds. To prove \eqref{error-estimates} we show the bounds on the damping term $b^w_\lambda$:
\begin{align}\label{error-estimate-damping}
\begin{aligned}
	\|x^\alpha b^w_\lambda(t,x,D_x)\mathcal S_{(k)}(t,s)\mathcal P_{j}\|_{L^2\rightarrow L^2} & \lesssim  (2^j+|t-s|)^{-N}, \\
	\|x^\alpha D_tb^w_\lambda(t,x,D_x)\mathcal S_{(k)}(t,s)\mathcal P_{j}\|_{L^2\rightarrow L^2} & \lesssim  (2^j+|t-s|)^{-N},
\end{aligned}	
\end{align}
which obviously implies that the error estimates \eqref{error-estimates} hold.
Once $\mathcal S_{(k)}(t,s)$ is fixed, $L^2$-bound follows by Proposition 11 of \cite{metataru}. The property $S_{(k)}(s+0,s)=\mathbf 1_{L^2}$ is also obvious. Now we focus on the pointwise decay \eqref{pointwise-decay}.
\subsection{Proof of dispersive inequality \eqref{pointwise-decay}}\label{subsec:proof-disp} Without loss of generality we may assume that the initial time $s$ is greater than $1$. In order to establish the required estimates, it suffices to prove the bounds for $\theta=0$ and $\theta=1$. The required bound follows after an interpolation.
\subsubsection*{Case 1: $|t-s|\lesssim1$}
 If $|t-s|\lesssim 1$, a simple use of the Sobolev embedding together with the $L^2$-bounds for the evolution operator gives 
\begin{align*}
\|\mathcal S_{(k)}(t,s)\mathcal P_{j} u_0\|_{L^\infty_x} & \lesssim  \|\mathcal S_{(k)} (t,s)\mathcal P_{j} u_0\|_{L^2_x} \lesssim  \|\mathcal P_{j} u_0\|_{L^2_x}  \lesssim \|u_0\|_{L^1_x},
\end{align*}
where we used the fact that the evolution operators $\mathcal S_{(k)}(t,s)$ are localised in the unit annular $\{\frac12<|\xi|<2\}$ uniformly in $\lambda$.
Since $|t-s|\lesssim1$, we obviously have 
\begin{align*}
	\|\mathcal S_{(k)}(t,s)\mathcal P_{j} u_0\|_{L^\infty_x} & \lesssim (1+|t-s|)^{-\frac{d}{2}}\|u_0\|_{L^1_x},
\end{align*}
and hence we see a better estimate on the small time scale.
Now we consider the long time regime, i.e., $1\ll |t-s|$. We divide this long-time case into the two cases, $s\le|t-s|$ and $|t-s|\le s$.
\subsubsection*{Case 2: $s\le|t-s|$}
 We consider the case $s\le|t-s|$, which becomes our main propagation regime, in which case one can neglect the damping term by the construction of the symbol $b_\lambda(t,x,\xi)$ in Section \ref{sec:construction-damping}. Here we choose the initial data $u(s,y)$ to be the Dirac delta distribution so that $\|u(s)\|_{L^1_x}=1$. Then applying the FBI transform we have
\begin{align*}
T_{\frac1s}u(s,x_s,\xi) = c_ds^{-\frac d4}e^{-\frac{x_s^2}{2s}}e^{i\xi\cdot x_s},	
\end{align*}
with an absolute constant $c_d$, which depends only on the spatial dimension $d\ge3$.
Now we apply the phase-space image of the $L^2$-evolution operators $\mathfrak S_{(k)}(t,s)$ and we write in the form
\begin{align*}
	\mathfrak S_{(k)}(t,s)[T_{\frac1s}u](t,x,\xi_t) = \int \mathfrak K_{(k)}(t,x,\xi_t,s,x_s,\xi)T_{\frac1s}u(s,x_s,\xi ) \,dx_s d\xi.
\end{align*}
By the use of the definition of the phase-space image of the evolution operators and the identity $T^*_{\frac1t}T_{\frac1t} = I$, we also have
\begin{align*}
\mathfrak S_{(k)}(t,s)T_{\frac1s}u &= (T_{\frac1t}\mathcal S_{(k)}(t,s)T_{\frac1s}^*T_{\frac1s}u(s)) 	
 = T_{\frac1t}\mathcal S_{(k)} (t,s)u(s) 
 = T_{\frac1t}u(t),
\end{align*}
and hence we deduce that
\begin{align*}
	T_{\frac1t}u(t,x,\xi_t) = \int \mathfrak K_{(k)}(t,x,\xi_t,s,x_s,\xi)T_{\frac1s}u(s,x_s,\xi ) \,dx_s d\xi.
\end{align*}
We use the kernel bound Theorem \ref{thm-bdd-kernel} to obtain
\begin{align*}
	& \left|\int \mathfrak K_{(k)}(t,x,\xi_t,s,x_s,\xi)T_{\frac1s}u(s,x_s,\xi ) \,dx_s d\xi \right| \\
	& \lesssim t^{-\frac d4} \int (1+t^{-1}\|x-x_t(x_s,\xi_t)\|_\lambda^2)^{-N/2}(1+s\|\xi-\xi_s(x_t,\xi_s)\|_{\lambda^{-1}}^2)^{-N/2}e^{-\frac{x_s^2}{2s}}\,d\xi dx_s,
\end{align*}
for any $N>0$. Recall that the distorted Euclidean norm $\|\cdot\|_\delta$ is defined by $\|x\|_\delta^2 = |x_1|^2+|x_2|^2+\cdots+\delta^2|x_d|^2$. The bounds for the integral follow in a straightforward way. Indeed, we see that
\begin{align*}
	\int (1+s\|\xi-\xi_s(x_t,\xi_s)\|_{\lambda^{-1}}^2)^{-N/2}\, d\xi & = \int_{\|\xi-\xi_s(x_t,\xi_s)\|_{\lambda^{-1}}\le 2s^{-\frac12}} (1+s\|\xi-\xi_s(x_t,\xi_s)\|_{\lambda^{-1}}^2)^{-N/2} \,d\xi \\
	& \qquad  + \int_{\|\xi-\xi_s(x_t,\xi_s)\|_{\lambda^{-1}}\ge 2s^{-\frac12}} (1+s\|\xi-\xi_s(x_t,\xi_s)\|_{\lambda^{-1}}^2)^{-N/2} \,d\xi.
\end{align*}
Then we have
\begin{align*}
	\int_{\|\xi-\xi_s(x_t,\xi_s)\|_{\lambda^{-1}}\le 2s^{-\frac12}} (1+s\|\xi-\xi_s(x_t,\xi_s)\|_{\lambda^{-1}}^2)^{-N/2} \,d\xi & \lesssim \int_{\|\xi-\xi_s(x_t,\xi_s)\|_{\lambda^{-1}}\le 2s^{-\frac12}} 1\,d\xi \lesssim \lambda s^{-\frac d2},
\end{align*}
and
\begin{align*}
	\int_{\|\xi-\xi_s(x_t,\xi_s)\|_{\lambda^{-1}}\ge 2s^{-\frac12}} (1+s\|\xi-\xi_s(x_t,\xi_s)\|_{\lambda^{-1}}^2)^{-N/2} \,d\xi & \lesssim \lambda s^{-\frac d2} \int_{2}^\infty (1+\rho^2)^{-N/2}\rho^{d-1}\,d\rho \lesssim \lambda s^{-\frac d2},
\end{align*}
when $N>d$. For the integration in $x_s$, we note that the map $x_s\mapsto x_t(x_s,\xi_t)$ is Lipschitz-continuous and the Lipschitz constant is bounded by $O(1)$. Now we integrate with respect to $x_s$ in the similar way as above to get the bound
\begin{align*}
&	\left|\int \mathfrak K_{(k)}(t,x,\xi_t,s,x_s,\xi)T_{\frac1s}u(s,x_s,\xi ) \,dx_s d\xi \right|   \lesssim \lambda t^{-\frac d4} \left( 1+t^{-1}\|x-x_t(\xi_t,0)\|_\lambda^2 \right)^{-N/2}.
\end{align*}
Finally, we apply $T^*_{\frac1t}$ to the $T_{\frac1t}u(t,x,\xi_t)$ and write in the form
\begin{align*}
	T^*_{\frac1t}\mathfrak S_{(k)}(t,s)[T_{\frac1s}u] = [T^*_{\frac1t}T_{\frac 1t}u(t,x,\xi_t)](t,y) = u(t,y)
\end{align*}
and hence we have the bound as
\begin{align*}
	|u(t,y)| & \lesssim \lambda t^{-\frac d2}\int ( 1+t^{-1}\|x-x_t(\xi_t,0)\|_\lambda^2 )^{-N/2} e^{-\frac{|y-x|^2}{2t}}\,dxd\xi_t \\
	& \lesssim \lambda t^{-\frac d2}\int ( 1+t^{-1}\|x-y\|_\lambda^2 )^{-N/4} \,dx \int  ( 1+t^{-1}\|y-x_t(\xi_t,0)\|_\lambda^2 )^{-N/4}d\xi_t.
\end{align*}
By the straightforward way as the previous integration we get
\begin{align*}
	\int ( 1+t^{-1}\|x-y\|_\lambda^2 )^{-N/4} \,dx \lesssim \lambda^{-1} t^{\frac d2},
\end{align*}
and hence we finally obtain the integral
\begin{align*}
	|u(t,y)| & \lesssim\int  ( 1+t^{-1}\|y-x_t(\xi_t,0)\|_\lambda^2 )^{-N/4}d\xi_t.
\end{align*}
Here the exponent $-\frac N4$ plays no crucial role and we may put $\frac N4\to N$ for simplicity.
{
It suffices to show that the above integral satisfies the estimate
\begin{align}
  \int  ( 1+t^{-1}|y-x_t(\xi_t,0)|^2 )^{-N}d\xi_t & \lesssim  t^{-\frac{d-1}{2}}\min\{1,\lambda t^{-\frac12}\}.
\end{align}
Thus it is not harmful to assume $\lambda^2\ll t$.
Now we consider the two cases: $t^{-1/2}\|y-x_t(\xi_t,0))\|_\lambda\lesssim1$ and $t^{-1/2}\|y-x_t(\xi_t,0))\|_\lambda\gg1$, i.e., 
\begin{align*}
	|u(t,y)| & \lesssim \int_{\|y-x_t(\xi_t,0)\|_\lambda\lesssim  t^\frac12}  ( 1+t^{-1}\|y-x_t(\xi_t,0)\|_\lambda^2 )^{-N}d\xi_t \\
	&\qquad + \int_{\|y-x_t(\xi_t,0)\|_\lambda\gg t^\frac12}  ( 1+t^{-1}\|y-x_t(\xi_t,0)\|_\lambda^2 )^{-N}d\xi_t \\
	& := I_1+I_2.
\end{align*}
We first consider $I_1$.
We let $\eta\in\mathbb R^d$ be any fixed point and set
\begin{align*}
    G_1^< := \{ \|y-x_t(\xi_t,0)\|_\lambda \lesssim t^\frac12 \}, \quad G_1^> := \{ \|y-x_t(\xi_t,0)\|_\lambda \gg t^\frac12 \}, \\
    G_2^< := \{ \|x_t(\eta,0)-x_t(\xi_t,0)\|_\lambda \lesssim t^\frac12 \},\quad G_2^> := \{ \|x_t(\eta,0)-x_t(\xi_t,0)\|_\lambda \gg t^\frac12 \}
\end{align*}
Then we rewrite the above integral $I_1$ as
\begin{align*}
	I_1 & \lesssim  \int_{G_1^<} ( 1+t^{-1}\|y-x_t(\eta,0)\|_\lambda^2 )^{-N/2} ( 1+t^{-1}\|x_t(\eta,0)-x_t(\xi_t,0)\|_\lambda^2 )^{-N/2}d\xi_t \\
	& \lesssim \int_{G_1^<\cap G_2^<} ( 1+t^{-1}\|y-x_t(\eta,0)\|_\lambda^2 )^{-N/2} ( 1+t^{-1}\|x_t(\eta,0)-x_t(\xi_t,0)\|_\lambda^2 )^{-N/2}d\xi_t \\
	& \quad + \int_{G_1^<\cap G_2^>}  ( 1+t^{-1}\|y-x_t(\eta,0)\|_\lambda^2 )^{-N/2} ( 1+t^{-1}\|x_t(\eta,0)-x_t(\xi_t,0)\|_\lambda^2 )^{-N/2}d\xi_t \\
	&:= I_{11}+I_{12}.
\end{align*}
Then
\begin{align*}
	I_{11} & \lesssim \int_{G_1^<\cap G_2^<} ( 1+t^{-1}\|x_t(\eta,0)-x_t(\xi_t,0)\|_\lambda^2 )^{-N/2}d\xi_t.
\end{align*}
In order to control the above integral, we consider the inverse Lipschitz constant of the map $\xi_t\to x_t(\xi_t,0)$. After an appropriate use of change of variables, (or the diagonalisation of the Jacobian matrix of the map $\xi_t\mapsto x_t(\xi_t,0)$ we have the inequality:
\begin{align}\label{inverse-lipschitz}
	\begin{aligned}
	    |\xi_j-\eta_j| &\lesssim  t^{-1}|x_j(\xi,0)-x_j(\eta,0)|, \quad j=1,2,\cdots,d-1, \\
        |\xi_d-\eta_d| & \lesssim t^{-1}\lambda^2 |x_d(\xi,0)-x_d(\eta,0)|.
	\end{aligned}
\end{align}
By using the above inequality \eqref{inverse-lipschitz}, we have
\begin{align*}
t\|\xi_t-\eta\|_{\lambda^{-1}}  \lesssim \|x_t(\xi_t,0)-x_t(\eta,0)\|_\lambda \lesssim t^{\frac12},
\end{align*}
and hence
\begin{align*}
    |\xi_j-\eta_j| \lesssim t^{-\frac12}, \ j=1,2,\cdots,d-1,\quad |\xi_d-\eta_d| \lesssim t^{-\frac12}\lambda.
\end{align*}
Thus we have
\begin{align*}
	I_{11} & \lesssim \int_{ G_1^<\cap G_2^<} ( 1+t^{-1}\|x_t(\eta,0)-x_t(\xi_t,0)\|_\lambda^2 )^{-N/2}d\xi_t \\
    & \lesssim \int_{\|\xi-\eta\|_{\lambda^{-1}}\lesssim t^{-\frac12}} 1\,d\xi \\
    &\lesssim  \lambda t^{-\frac d2},
\end{align*}
where we simply get the measure of the support.
Now we control the integral $I_{12}$:
\begin{align*}
    I_{12} & \lesssim \int_{G_1^<\cap G_2^>} (1+t^{-1}\|x_t(\xi_t,0)-x_t(\eta,0)\|_\lambda^2)^{-N/2}\,d\xi_t.
\end{align*}
Using the inverse Lipschitz constant of the map $\xi_t\mapsto x_t(\xi_t,0)$ we have $\|x_t(\xi_t,0)-x_t(\eta,0)\|_\lambda \gtrsim t\|\xi_t-\eta\|_{ \lambda^{-1}}$. If $\|\xi_t-\eta\|_{\lambda^{-1}} \lesssim t^{-\frac12}$, then we obtain the bound $\lambda t^{-\frac d2}$. On the other hand, if $\|\xi_t-\eta\|_{\lambda^{-1}}\gg t^{-\frac12}$, then
\begin{align*}
  &  \int_{G_1^<\cap G_2^>} (1+t^{-1}\|x_t(\xi_t,0)-x_t(\eta,0)\|_\lambda^2)^{-N/2}\,d\xi_t \\
  & \lesssim \int_{\|\xi-\eta\|_{\lambda^{-1}}\gg t^{-\frac12}} (1+t\|\xi-\eta\|_{\lambda^{-1}}^2)^{-N/2}\,d\xi \\
  & \lesssim \lambda t^{-\frac d2}.
\end{align*}
We consider the integral $I_2$. We decompose it into the two integrals in the similar way and get 
\begin{align*}
	I_2 & \lesssim  \int_{G_1^>} ( 1+t^{-1}\|y-x_t(\eta,0)\|_\lambda^2 )^{-N/2} ( 1+t^{-1}\|x_t(\eta,0)-x_t(\xi_t,0)\|_\lambda^2 )^{-N/2}d\xi_t \\
	& \lesssim \int_{G_1^>\cap G_2^<} ( 1+t^{-1}\|y-x_t(\eta,0)\|_\lambda^2 )^{-N/2} ( 1+t^{-1}\|x_t(\eta,0)-x_t(\xi_t,0)\|_\lambda^2 )^{-N/2}d\xi_t \\
	& \quad + \int_{G_1^>\cap G_2^>}  ( 1+t^{-1}\|y-x_t(\eta,0)\|_\lambda^2 )^{-N/2} ( 1+t^{-1}\|x_t(\eta,0)-x_t(\xi_t,0)\|_\lambda^2 )^{-N/2}d\xi_t \\
	&:= I_{21}+I_{22}.
	\end{align*}
    Then we see that the control of the integrals $I_{21}$ and $I_{22}$ follow in the identical manner as the estimates of the integral $I_1$. We omit the details.
\subsubsection*{Case 3: $|t-s|\le s$} If $|t-s|\le 2^j$, we reinitialise the time scale so that the initial time is set to be $|t-s|$. We need the bounds on the symbol of the damping term $b^w_\lambda$:
\begin{align}
\begin{aligned}
	|\partial_x^\alpha \partial^\beta_\xi b_\lambda(t,x,\xi)| & \lesssim t^{-\frac12-|\alpha|}, \quad |\alpha|\le1, \\
	|\partial_x^\alpha \partial^\beta_\xi b_\lambda(t,x,\xi)| & \lesssim t^{-1-\frac{|\alpha|}{2}}, \quad |\alpha|\ge2,
\end{aligned}	
\end{align}
which can be obtained by the straightforward computation from the construction of the damping term $b_\lambda(t,x,\xi)$ in Section \ref{sec:construction-damping}. The above estimate together with the bound \eqref{bdd-perturb-a} implies that the symbol $a_\lambda$ and $b_\lambda$ again satisfy $a_\lambda\in\ell^1S^{(2)}_{\epsilon}$ and $b_\lambda\in \ell^1S^{(1)}_{}$, respectively. This ensures that we can apply Theorem \ref{thm-bdd-kernel}, and hence the required time decay is obtained by the identical manner as the previous case. See also \cite{metataru,tataru4}. This completes the proof of \eqref{pointwise-decay}.
\subsection{Construction of the damping term $b^w_\lambda$}\label{sec:construction-damping}
In what follows we focus on the proof of several error-type estimates of Theorem \ref{main-thm-parametrix}. To do this, we need to construct exactly the damping term $\mathfrak B^w_\lambda$.
\begin{lem}\label{property-damping}
Let $\lambda\ge1$. There exists a symbol $b_\lambda\in\ell^1 S^{(1)}_{}$ which satisfies
\begin{align}
\begin{aligned}
	|\partial^\alpha_x\partial^\beta_\xi b_\lambda(t,x,\xi)| & \lesssim t^{-\frac12-|\alpha|}, \ |\alpha|\le1, \\
	|\partial^\alpha_x\partial^\beta_\xi b_\lambda(t,x,\xi)| & \lesssim t^{-1-\frac{|\alpha|}{2}}, \ |\alpha|\ge1,
	\end{aligned}
\end{align} 	
for $|t-s|\le 2^j$, and
\begin{enumerate}
	\item $t^\frac34b_\lambda$ is non-increasing along the Hamilton flow for the evolution $D_t+\langle D_x\rangle_\lambda+a^w_\lambda$ and
	\begin{align*}
		0< t^\frac34b_\lambda(t,x_t,\xi_t)<1 \Rightarrow b_\lambda(2t,x_{2t},\xi_{2t}) = 0.
	\end{align*}
	\item At the initial time, we have \begin{align*}
		b_\lambda(2^j,x,\xi) = 0, \textrm{ in the region } \{\frac18 <|\xi|<8, \ 2^{j-2}<|x|<2^{j+2}, \ x\cdot\xi \ge -2^{-4}|x||\xi| \}.
	\end{align*}
	\item At any time $t\ge 2^j$, we have \begin{align*}
		b_\lambda(t,x,\xi) = t^{-\frac34}, \textrm{ outside the region } \{ \frac{1}{16} < |\xi| <16 \}\cap \{ 2^{-6}t< |x|<2^{6}t \}.
	\end{align*}
\end{enumerate}
\end{lem}
We may replace the power $\frac34$ in the above Lemma by any number between $\frac12$ and $1$. See also \cite{tataru4}.
\begin{lem}\label{concl-damping}
Let the symbol $b_\lambda\in\ell^1 S^{(1)}_{}$ satisfy the above properties of Lemma \ref{property-damping}. Then the bounds \eqref{outgoing-parametrix}, \eqref{finite-speed}, \eqref{frequency-localiseation},	 \eqref{frequency-localisation-re}, and \eqref{error-estimate-damping} hold.
\end{lem}
In what follows, we give the explicit construction of the symbols $b_\lambda(t,x,\xi)$ of the damping term $b^w_\lambda(t,x,D_x)$. Then we examine the properties of the damping term which truly plays a crucial role to cover the region outside the main propagation regime and consequently to complete the proof of Theorem \ref{main-thm-parametrix}.
\begin{proof}[Construction of the function $b_\lambda=b_\lambda(t,x,\xi)$]
	We define the increasing bounded function $e(s)$ to be 
	\begin{align*}
		e(s) = \frac1\epsilon\int_0^s \frac{\epsilon(\mathfrak s)}{\mathfrak s}\,d\mathfrak s.
	\end{align*}
	We choose a smooth, non-decreasing cut-off function $\phi\in C^\infty(\mathbb R)$ such that $\phi(s) = 0$ when $s\in(-\infty,0)$ and $\phi=1$ on $(1,\infty)$. Now we define the function $b_\lambda$ by
	\begin{align*}
		b_\lambda(t,x,\xi) = t^{-\frac34}\left(1 - \phi(b_{1})\phi(b_{2})\phi(b_{3})\phi(b_{4})\phi(b_{5}) \right),
	\end{align*}
	where 
	\begin{enumerate}
		\item $b_{1}$ is settled to cut-off frequencies which are relatively too larger than $1$, i.e., $|\xi|\gg1$: \begin{align*}
			b_{1}(t,\xi) = \frac{2^{\frac72}+e(t)-|\xi|}{\epsilon(t)},
		\end{align*}
		\item $b_{2}$ is settled to cut-off frequencies which are relatively too smaller than $1$, i.e., $|\xi|\ll1$: \begin{align*}
			b_{2}(t,\xi) = \frac{|\xi|-2^{-\frac72}+ce(t)}{\epsilon(t)},
		\end{align*}
		where $c$ is a fixed small constant,
		\item $b_{3}$ is settled to detect the outgoing waves:\begin{align*}
			b_{3}(t,x,\xi) = \frac{2^{-\frac12}|x|||\xi|+x\cdot\xi}{2^{-12}|x|},
		\end{align*}
		\item $b_{4}$ is settled to cut-off the values of $|x|$ which are too large: \begin{align*}
			b_{4}(t,x) = \frac{2^6t-|x|}{t},
		\end{align*}
		\item $b_{5}$ is settled to cut-off the values of $|x|$ which are too small: \begin{align*}
			b_{5}(t,x,\xi) = \frac{|x||\xi|-2^{-5}t|\xi|+x\cdot\xi}{2^{-10}t}.
		\end{align*}
	\end{enumerate}
	It is easy to observe that $b_\lambda=0$ if $b_{j}=1$, $j=1,\cdots,5$. In other words, the set $\{b_\lambda=0\}$ contains the region 
	\begin{align*}
		\{ \frac18 < |\xi|<8\}\cap \{ t/4 < |x| <4t \} \cap \{ x\cdot \xi \ge -2^{-4}|x| \},
	\end{align*}
	provided that $\epsilon>0$ is sufficiently small. On the other hand, we also observe that $b_\lambda<t^{-\frac34}$ if $0<b_{j}<1$, $j=1,\cdots,5$. More precisely, we see that the set $\{t^\frac34 b_\lambda<1\}$ is contained in the region
	\begin{align*}
		\{ \frac1{16} <|\xi|<16\}\cap \{2^{-6}t<|x|<2^6t\}\cap \{x\cdot\xi > -2^{-\frac12}|x||\}:=D_t,
	\end{align*}
	which shows that the conditions (2) and (3) are satisfied. Now we prove the condition (1). On the set $\{b_\lambda=0\}$, it is trivial, and hence we focus on the set $D_t$. It suffices to show that for each $b_{j}$, we have
	\begin{align*}
		\frac{d}{dt}b_{j}(t,x_t,\xi_t) \ge \frac2t, 
	\end{align*}
	on the set $D_t\cap\{0\le b_{j}\le1\}$,
	where $t\mapsto(x_t,\xi_t)$ is a trajectory of the Hamilton flow for the evolution $D_t+\langle D_x\rangle_\lambda+a^w_\lambda$. We first notice that for each $(x_t,\xi_t)\in D_t$, 
	\begin{align*}
		\frac{d}{dt}\xi_t = O(\frac{\epsilon(t)}{t}), \ \frac{d}{dt}x_t = \frac{\xi_t}{\langle \xi_t\rangle_\lambda}+O(\epsilon(t)),
	\end{align*}
	which easily gives 
	\begin{align*}
		\frac{d}{dt}|\xi_t| = \frac{\xi_t}{|\xi_t|}\frac{O(\epsilon(t))}{t}, 
	\end{align*}
	and hence the lower bounds for $b_{1}$ and $b_{2}$ easily follows. Indeed, we see that
	\begin{align*}
		\frac{d}{dt}b_{1}(t,\xi_t) = \frac{e'(t)}{\epsilon(t)} - \frac{1}{\epsilon(t)}\frac{\xi_t}{|\xi_t|}\frac{O(\epsilon(t))}{t}-\frac{2^{\frac72}-e(t)-|\xi_t|}{\epsilon(t)^2}\epsilon'(t) \ge \frac2t.
	\end{align*}
	Here we used the properties of $\epsilon(t)$ \eqref{ep-property} and the definition of $e(t)$, which gives the desired lower bound on the set $\{0\le b_{1}\le1\}$. For $b_{2}$, in the identical manner we see that
	\begin{align*}
		\frac{d}{dt}b_{2}(t,\xi_t) = \frac{1}{\epsilon(t)}\frac{\xi_t}{|\xi_t|}\frac{O(\epsilon(t))}{t}+c\frac{e'(t)}{\epsilon(t)}-\frac{|\xi_t|-2^{-\frac72}+ce(t)}{\epsilon(t)^2}\epsilon'(t) \ge \frac2t,
	\end{align*}
	on the set $\{0\le b_{2}\le1\}$.
	
	\noindent An easy computation gives us
	\begin{align*}
		\frac{d}{dt}\frac1{|x_t|} = -\frac{x_t}{|x_t|^3}\cdot\left(\frac{\xi_t}{\langle\xi_t\rangle_\lambda}+O(\epsilon) \right), \ \frac{d}{dt}\frac1{|\xi_t|} = -\frac{\xi_t}{|\xi_t|^3}O(\epsilon).
	\end{align*}
	Using these identities, the required lower bounds for $b_{j}$, $j=3,4,5$ are obvious. Indeed, we have
	\begin{align*}
		\frac{d}{dt}b_{3}(t,x,\xi) = 2^{12} \frac{|x_t|^2|\xi_t|^2-(x_t\cdot\xi_t)^2}{|x_t|^3\langle\xi_t\rangle_\lambda}+\frac{O(\epsilon)}{t}\ge \frac2t ,
	\end{align*}
	since $-2^{-\frac32}\le\cos\angle(x_t,\xi_t)\le-\frac12$ on the set $\{0\le b_{3}\le1\}$. We also have for $b_{4}$
	\begin{align*}
		\frac{d}{dt}b_4(t,x_t,\xi_t) = \frac{|x_t|^2\langle\xi_t\rangle_\lambda-tx_t\cdot\xi_t}{t^2|x_t|\langle\xi_t\rangle_\lambda}-\frac{O(\epsilon)}{t} \ge\frac2t,
	\end{align*}
	on the set $\{0\le b_{4}\le1\}$. Now we see that
	\begin{align*}
		\frac{d}{dt}b_{5}(t,x_t,\xi_t) = 2^{10}\left( -\frac{|x_t|\langle\xi\rangle_\lambda+x_t\cdot\xi_t}{t^2}+\frac{|x_t|(x_t\cdot\xi_t)+\xi_t}{t}+\frac{O(\epsilon(t))}{t} \right) \ge\frac2t.
			\end{align*}
	\end{proof}
\begin{proof}[Proof of Lemma \ref{concl-damping}]
	In view of Lemma 10 and Lemma 11 of \cite{metataru}, we make only a slight modification of the construction of the damping terms $b_\lambda$. 
	The proof is based on the pointwise bounds for the phase space transform of the solution $u$ to the evolution equation
	\begin{align*}
		(D_t +\langle D_x\rangle_\lambda+a^w_\lambda(t,x,D)) u = -ib^w_\lambda(t,x,D)u, \ u(2^j) = \mathcal P_{j}u_0.
	\end{align*}
	At the initial time $s=2^j$, we already know that the symbol of $\mathcal P^+_{j}$ is supported on the region
	\begin{align*}
		D_s = \{ \frac14 <|\xi| <4, \ 2^{j-1}<|x|<2^{j+1}, \ \cos\angle(x,\xi) \ge 2^{-5} \}.
	\end{align*}
	Then we have the phase space transform of the initial data $\mathcal P_{j}u_0$ satisfies pointwise bound
	\begin{align*}
		|(T_{\frac1s}\mathcal P_{j}u_0)(x,\xi)| \lesssim \left(1+d_s((x,\xi),D_s)^2 \right)^{-N},
	\end{align*}
	for any $N>0$. Here $d_s$ is the rescaled distance function on the phase space given by
	\begin{align*}
		d_s((x,\xi),(y,\eta))^2 = s^{-1}|x-y|^2+s|\xi-\eta|^2. 
	\end{align*}
	As we have done at the beginning of the proof of the dispersive inequality \eqref{pointwise-decay}, we invoke the definition of the phase space kernel of the evolution operators to write in the form
	\begin{align*}
		\mathfrak S_{(k)}(t,s)[T_{\frac1s}\mathcal P_{j}u_0] = T_{\frac1t}u(t,x,\xi_t) = \int \mathfrak K_{(k)}(t,x,\xi_t,s,x_s,\xi)T_{\frac1s}\mathcal P_{j}u(s,x_s,\xi)\,dx_sd\xi.
	\end{align*}
	The application of the kernel bounds Theorem \ref{thm-bdd-kernel} yields
	\begin{align*}
		|(T_{\frac1t}u)(t,x,\xi_t)|& \lesssim \left(\frac ts\right)^{-\frac d4} \int \left( 1+t^{-1}\|x-x_t\|_\lambda^2 \right)^{-\frac{N}{2}}\left( 1+s\|\xi-\xi_s\|_{\lambda^{-1}}^2+(\Psi_\lambda(x_t,\xi_t)-\Psi_\lambda(x_s,\xi_s))^2 \right)^{-\frac N2} \\
		& \qquad\qquad\qquad\qquad \times \left(1+d_s((x,\xi),D_s)^2 \right)^{-N} \,dx_sd\xi.
	\end{align*}
	After an obvious integration on $\xi$ as we have done in the proof of \eqref{pointwise-decay}, we see that
	\begin{align*}
		|(T_{\frac1t}u)(t,x,\xi_t)|& \lesssim \left(\frac ts\right)^{-\frac d4} \int \left( 1+t^{-1}\|x-x_t\|_\lambda^2 \right)^{-\frac{N}{2}}\left( 1+s\|\xi-\xi_s\|_{\lambda^{-1}}^2\right)^{-\frac N2}\\ & \qquad \times \left(1+(\Psi_\lambda(x_t,\xi_t)-\Psi_\lambda(x_s,\xi_s))^2 \right)^{-\frac N2}\left(1+d_s((x,\xi),D_s)^2 \right)^{-N} \,dx_sd\xi \\
		& \lesssim \lambda \int t^{-\frac d2} \left( 1+t^{-1}\|x-x_t\|_\lambda^2 \right)^{-\frac{N}{2}} \mathcal W(t,x_t,\xi_t)\, dx_t,
	\end{align*}
	where we put
	\begin{align}\label{put-w}
		\mathcal W(t,x_t,\xi_t) = \left(\frac ts\right)^{\frac d4}\left(1+(\Psi_\lambda(x_t,\xi_t)-\Psi_\lambda(x_s,\xi_s))^2 \right)^{-\frac N2}\left(1+d_s((x_s,\xi_s),D_s)^2 \right)^{-N}.
	\end{align}
	Note that we can freely replace the variables $x_s$ in the integrand by $x_t$ with no harm, due to the Lipschitz continuity of the map $x_s\mapsto x_t$, whose Lipschitz constant is bounded by $\lesssim O(1)$. 
	We claim that the function $\mathcal W(t,x_t,\xi_t)$ safisties pointwise bound: for any arbitrarily large $N>0$,
	\begin{align}\label{bdd-claim-w}
	|\mathcal W(t,x_t,\xi_t)| \lesssim \left(\frac ts\right)^{\frac d4} \left(1+ t^\frac12 +t^{-2}|x_t|^2+|\xi_t|^2 \right)^{-\frac N4} \left( 1+ d_t((x_t,\xi_t), \textrm{supp}\, b_\lambda(t))^2 \right)^{\frac N4}	.
	\end{align}
	In other words, we allow an arbitrary growth on the factor $d_t((x_t,\xi_t), \textrm{supp} \,b_\lambda(t))$. When $(x_t,\xi_t)\in \textrm{supp }b_\lambda(t)$, where the operator $b^w_\lambda$ truly plays a role as a damping term, this term is identically zero, and hence is not problematic. However, even if $(x_t,\xi_t)\notin \textrm{supp }b_\lambda(t)$, in which case we deal with the main propagation regime, this term should not make any harm, since $\mathcal W(t,x_t,\xi_t)$ in \eqref{put-w} satisfies an obvious bound 
	\begin{align}\label{bdd-w-trivial}
	 |\mathcal W(t,x_t,\xi_t)|	\lesssim t^{\frac d4}s^{-\frac d4},
	\end{align}
	which implies that we can assume without any loss of generality that 
	\begin{align}\label{as-db-factor}
		d_t ((x_t,\xi_t),\textrm{supp }b_\lambda(t) ) \ll t^{-\frac14},
	\end{align}
	since $|x_t|\approx t$ and $|\xi_t|\approx1$, provided that $(x_t,\xi_t)\notin \textrm{supp }b_\lambda(t)$. Therefore, we conclude that the term $\left( 1+ d_t((x_t,\xi_t), \textrm{supp}\, b_\lambda(t))^2 \right)^{N/4}$ is bounded by $O(1)$ in any cases and it does not cause any loss. 
	
	The proof of the estimate \eqref{bdd-claim-w} follows in the identical manner as \cite{tataru4}, and hence we omit the details. Now we shall show that it truly implies that the estimates \eqref{outgoing-parametrix}, \eqref{finite-speed}, \eqref{frequency-localisation-re}, \eqref{error-estimate-damping}. In other words, we deduce that outside the main propagation regime $\{|x|\approx t, \ |\xi|\approx1\}$, their time-evolution result in the rapid time-decay. Therefore, the operators $b^w_\lambda$ satisfying the properties of Lemma \ref{property-damping} turn out to play a role as the damping term, which resolves a certain leakage outside the main propagation regime.
	
	In what follows, we deal with several cut-off operators arising from the parametrix $\mathcal S_{(k)}(t,s)\mathcal P_{j}$, which describe the region outside the main propagation regime. In order to establish the required rapid time-decay in the $L^2\rightarrow L^2$-bound for the operators, in view of the Minkowski's integral inequality it is enough to study the pointwise bound of the kernels of each cut-off operators. For simplicity we simply normalise the initial data $u_0$ so that $\|u_0\|_{L^2_x}=1$.

 We give the proof of the estimates \eqref{frequency-localisation-re} and \eqref{error-estimate-damping} with $\theta=1$. When we consider the parametrix as a wave-type operator, one can readily obtain the bound with $\theta=0$ and then establish the factor $\lambda$ by a simple use of interpolation.  The proof of \eqref{outgoing-parametrix} and \eqref{finite-speed} follows in the similar way.
\subsubsection*{Proof of \eqref{frequency-localisation-re}}
Now we deal with the proof of the frequency-localisation properties. To prove the first estimate of \eqref{frequency-localisation-re}, we fix the multi-index $\alpha\in\mathbb N^d_0$ with an arbitrarily large $|\alpha|=\tilde M>0$. After an application of the kernel bound \eqref{bdd-claim-w}, we simply write 
\begin{align*}
	|u(t,y)| & \lesssim \lambda\int_{\mathbb R^d}|x|^{\tilde M} (1+t^{-2}|x|^2)^{-N/4}(1+t^{-1}|x|^2)^{-M}\,dx \\
	& = \lambda\int_{|x|\le 2t}|x|^{\tilde M} (1+t^{-2}|x|^2)^{-N/4}(1+t^{-1}|x|^2)^{-M}\,dx \\
	& \qquad + \lambda\int_{|x|\ge2t}|x|^{\tilde M} (1+t^{-2}|x|^2)^{-N/4}(1+t^{-1}|x|^2)^{-M}\,dx.
\end{align*} 
By simply using the scaling $x\rightarrow tx$ and the integration on $x$, we easily obtain
\begin{align*}
	|u(t,y)| \lesssim \lambda t^{\tilde M+d}(1+t)^{-M/2},
\end{align*}
which is the desired bound provided that $M/2-\tilde M-d>0$ is an arbitrarily large. Indeed, we can choose $M>0$ so that the above inequality always holds. 
The proof of the second estimate of \eqref{frequency-localisation-re} follows in the essentially identical manner. In fact, we put the multi-index $\beta\in\mathbb N^d$ fixed with an arbitrarily large $|\beta|=\tilde N>0$. After the scaling $x\rightarrow tx$ we write
\begin{align*}
	|u(t,y)|& \lesssim \lambda \int_{\mathbb R^d}|x|^{\tilde M} |\xi|^{\tilde N}(1+t^{-2}|x|^2+|\xi|^2)^{-N/4}(1+t^{-1}|x|^2)^{-M}\,dx \\
	& \lesssim \lambda t^{\tilde M+d}\int_{\mathbb R^d}|\xi|^{\tilde N}(1+|\xi|^2)^{-N/8}|x|^{\tilde M}(1+|x|^2)^{-N/8}(1+t|x|^2)^{-M/2}\,dx.
\end{align*}
Then an obvious integration on $x$ yields the desired bound, provided that $\max(\tilde N,\tilde M+d-1)<N/8$, and $\tilde M+d\ll M/2$. Again, we can always choose $N$ and $M$ sufficiently large so that the required inequalities are always true.
\subsubsection*{Proof of \eqref{error-estimate-damping}} To prove the error estimates, we apply Proposition 20 of \cite{tataru4} to exploit the pointwise bound of the phase space kernel, whose symbol is supported in a certain region.
This gives rise to the following integral:
\begin{align*}
	& |T_{\frac1t}b^w_\lambda(t,x,D_x)u(t,x,\xi)| \\
	 & \lesssim |b_\lambda(t,x,\xi)|+(1+d_t((x,\xi),\textrm{supp }b_\lambda(t)))^{-M} \\ 
	 & \qquad\times \int (1+d_t((x,\xi),(y,\eta)))^{-M}(1+t^\frac12+t^{-2}|y|^2+|\eta|^2)^{-N/4}(1+d_t((y,\eta),\textrm{supp }b_\lambda(t)))^{N/4}\,dyd\eta .
\end{align*}
We may bound the term $(1+d_t((y,\eta),\textrm{supp }b_\lambda(t)))^{N/4}$ by $O(1)$ using the aforementioned remark \eqref{as-db-factor}. After the simple use of the scaling: $y\rightarrow t^\frac12y$, $x\rightarrow t^\frac12x$, $\eta\rightarrow t^\frac12\eta$, $\xi\rightarrow t^\frac12\xi$, the integral above is transformed as follows:
\begin{align*}
	\int\!\!\!\!\int (1+|x-y|)^{-M}(1+|\xi-\eta|)^{-M}(1+t^\frac12+t^{-1}|y|^2+t|\eta|^2)^{-N/4}\,dyd\eta . 
\end{align*}
We simply integrate on $y$ and $\eta$ and rescale the variables $t^\frac12x\rightarrow x$ and $t^\frac12\xi\rightarrow \xi$ to get
\begin{align*}
&	|T_{\frac1t}b^w_\lambda(t,x,D_x)u(t,x,\xi)|  \lesssim (1+t^\frac12+t^{-2}|x|^2+|\xi|^2)^{-N/4}(1+d_t((x,\xi),\textrm{supp }b_\lambda(t)))^{-M}.
\end{align*}
Then we apply the map $T^*_{\frac1t}$ and use the bound above to get
\begin{align*}
	|u(t,y)| & \lesssim \int |x|^{\tilde M}(1+t^{-2}|x|^2)^{-N/4}(1+t^{-1}|x|^2)^{-M}\,dx.
\end{align*}
Then the remaining task is exactly same as the proof of the estimates \eqref{frequency-localisation-re}. Finally, we are left to deal with the operator involving $D_tb^w_\lambda$, which is merely a repetitive task following the proof of Lemma 11 of \cite{metataru}, and hence we omit it. 
This completes the proof of all the error-type estimates, and hence completes the proof of the main properties of our outgoing parametrix Theorem \ref{main-thm-parametrix}.
\end{proof}
\section*{Appendix: Proof of \eqref{est-stri-local}} 
Now we give the proof of the frequency-localised Strichartz estimates.
	We first consider the non-endpoint case $p>2$. By the use of interpolation of the estimates between 
	\begin{align*}
		\|\mathcal K^\pm_{(k)} f\|_{L^\infty_x} & \lesssim \langle2^k\rangle^{\theta}(1+|t-s|)^{-\frac{d-1+\theta}{2}}\|f\|_{L^1_x}
	\end{align*}
	and
	\begin{align*}
		\|\mathcal K^\pm_{(k)} f\|_{L^2_x} & \lesssim \|f\|_{L^2_x},
	\end{align*}
	we obtain for $q\ge2$
	\begin{align*}
		\|\mathcal K^\pm_{(k)} f\|_{L^q_x} & \lesssim \langle2^k\rangle^{\theta(1-\frac2q)}(1+|t-s|)^{-\frac{d-1+\theta}{2}(1-\frac2q)}\|f\|_{L^{q'}_x}.
	\end{align*}
	Then we use the weak Young's inequality.
    
    Next, we deduce the space-time estimates: for $d\ge3$ and $0\le\theta\le1$,
\begin{align}\label{stri-pp'-bdd}
\|\mathcal K^\pm_{(k)} f\|_{L^p_tL^q_x}	& \lesssim \langle2^k\rangle^{(1-\frac2q)\theta}\|f\|_{L^{p'}_tL^{q'}_x},
\end{align}
where $(\sigma,p,q)$ satisfies the relations:
\begin{align*}
\frac2p+\frac{d-1+\theta}{q}= \frac{d-1+\theta}{2}. 
\end{align*}
We first prove the estimates
\begin{align}\label{stri-2-bdd}
\|\mathcal K^\pm_{(k)} f\|_{L^\infty_tL^2_x} & \lesssim \langle2^k\rangle^{(1-\frac2q)\frac\theta2}\|f\|_{L^{p'}_tL^{q'}_x},
\end{align}
which is equivalent to the following estimates via a typical approach of the $TT^*$-argument:
\begin{align*}
	\|\mathcal K^\pm_{(k)}(\cdot,t)^*\mathcal K^\pm_{(k)}(t,\cdot)\|_{L^{p'}_tL^{q'}_x\rightarrow L^p_tL^q_x} \lesssim \langle2^k\rangle^{(1-\frac2q)\theta},
\end{align*}
where $\mathcal K_\lambda(s,t)^*$ is the $L^2$-adjoint of $\mathcal K^\pm_{(k)}(t,s)$. This can be obtained by the use of the estimates \eqref{stri-pp'-bdd} and the following lemma: 
\begin{lem}[Lemma 13 of \cite{tataru4}]
	For the parametrix $\mathcal K^\pm_{(k)}$, we have
	\begin{align*}
		\| \mathcal K^\pm_{(k)}(s_1,t)^*\mathcal K^\pm_{(k)}(t,s_2)-\mathcal K^\pm_{(k)}(s_1,s_1+0)^*\mathcal K^\pm_{(k)}(s_1,s_2)\|_{L^2\rightarrow L^2} \lesssim (1+|s_1-s_2|)^{-N}, \ s_1<s_2<t.
	\end{align*}
\end{lem}
\noindent The proof is identical as the proof of Lemma 13 of \cite{tataru4}. 
Now we prove non-endpoint Strichartz estimates. We interpolate the above estimates \eqref{stri-2-bdd} and \eqref{stri-pp'-bdd} with $(p,q)=(p_2,q_2)$ to get
\begin{align*}
\|\mathcal K^\pm_{(k)} f\|_{L^{p_1}_tL^{q_1}_x} & \lesssim \langle2^k\rangle^{(1-\frac2{q_2})\frac\theta2(1-\eta)+(1-\frac2{q_2})\theta\eta}\|f\|_{L^{p_2'}_tL^{q_2'}_x},	
\end{align*}
where $\frac1{p_1}= \frac{1-\eta}{\infty}+\frac{\eta}{p_2}$, $\frac1{q_1}= \frac{1-\eta}{2}+\frac{\eta}{q_2}$ for $0\le\eta\le1$. Note that $q_1\le q_2$. Since $(1-\frac2{q_2})\eta=1-\frac2{q_1}$, we have
\begin{align*}
\|\mathcal K^\pm_{(k)} f\|_{L^{p_1}_tL^{q_1}_x} & \lesssim \langle2^k\rangle^{(1-\frac2{q_1})\frac\theta2+(1-\frac2{q_2})\frac\theta2}\|f\|_{L^{p_2'}_tL^{q_2'}_x},	
\end{align*}
which proves the $L^{p}_tL^{q}_x$-bound of $\mathcal K^\pm_{(k)} f$. Concerning the $X_k$-part, as the proof of Proposition 12 of \cite{tataru4}, it suffices to prove the estimates
\begin{align*}
	\|\mathbf 1_{|x|<R} \mathcal K^\pm_{(k)}(t,s)\|_{L^{p'}_tL^{q'}_x\rightarrow L^2_tL^2_x} \lesssim \langle2^k\rangle^{(1-\frac2q)\frac\theta2 }|R|^\frac12.
\end{align*} 
We may assume that $R\gg1$.
We split the parametrix $\mathcal K^\pm_{(k)}(t,s)$ into
\begin{align*}
	\mathcal K^\pm_{(k)}(t,s) = \mathbf 1_{\{|t-s|<2^{10}R\}}\mathcal K^\pm_{(k)}(t,s)+\mathbf 1_{\{|t-s|>2^{10}R\}}\mathcal K^\pm_{(k)}(t,s).
\end{align*}
Then for the first term on the right handside the use of the H\"older inequality and the Strichartz estimates yields
\begin{align*}
	\|\mathbf 1_{|x|<R} \mathbf 1_{\{|t-s|<2^{10}R\}}\mathcal K^\pm_{(k)}(t,s)f\|_{L^2_tL^2_x} & \lesssim |R|^\frac12 \|\mathcal K^\pm_{(k)} f\|_{L^\infty_tL^2_x} \\
	& \lesssim  |R|^\frac12 \langle2^k\rangle^{(1-\frac2q)\frac\theta2} \|f\|_{L^{p'}_tL^{q'}_x}.
\end{align*}
For the second term 
we use the error-type estimates and the Sobolev embedding to get
\begin{align*}
	\|\mathbf 1_{|x|<R}\mathbf 1_{\{|t-s|>2^{10}R\}}\mathcal K^\pm_{(k)}(t,s)f\|_{L^2_tL^2_x} & \lesssim \left\| (1+|\cdot|)^{-N}*\|f\|_{L^2_x} \right\|_{L^2_t} \\
	& \lesssim \left\| (1+|\cdot|)^{-N}*\|f\|_{L^{q'}_x} \right\|_{L^2_t} \\
	& \lesssim \|f\|_{L^{p'}_tL^{q'}_x},
\end{align*}
where we used the weak Young's inequality to get the last inequality. 
Now we are concerned with the endpoint issue and follow the strategy of \cite{keeltao,tataru4}. We need to prove the following estimates:
\begin{align}\label{est-stri-endpoint}
	\|\mathcal K^\pm_{(k)} f\|_{L^2_tL^q_x} & \lesssim \langle2^k\rangle^{(1-\frac2q)\frac\theta2}\|f\|_{L^2_tL^{q'}_x},
\end{align}	
where $q=\frac{2(d-1+\theta)}{d-3+\theta}$, $0\le\theta\le1$, and $d\ge3$. Here we exclude the endpoint $q=\infty$, $\theta=0$, $d=3$.
\noindent We prove this in the similar way as Theorem 4 of \cite{tataru}.
We note that the kernel $K^\pm_{(k)}$ of the parametrix $\mathcal K^\pm_{(k)}$ satisfies the bound
\begin{align}\label{kernel-bdd-diagonal}
	|K^\pm_{(k)}(t,x,s,y)| & \lesssim \langle2^k\rangle^{\theta}(1+|t-s|+|x-y|)^{-\frac{d-1+\theta}{2}}, \ 0\le\theta\le1,
\end{align}
since by our construction the kernel  of $\mathcal K^\pm_{(k)}$ is concentrated on the region $\{ |x-y|\approx|t-s| \}\cap\{|\xi|\approx1\}$. From now on we drop out the sign $\pm$ of the notation $\mathcal K^\pm_{(k)}$ and simply denote the parametrix by $\mathcal K_{(k)}$.
The first step is to decompose our parametrix $\mathcal K_{(k)}$ into the sum of the dyadic pieces:
\begin{align*}
	\mathcal K_{(k)} = \sum_{\mu\in 2^{\mathbb Z}}\mathcal K_{(k)}^\mu,
\end{align*}
each of which are supported on the region $|(t,x)-(s,y)|\approx \mu$. To construct such a dyadic decomposition we introduce the collection $\mathcal Q$ of all closed cubes of size $\mu$ and whose vertices are in $\mu(\mathbb Z^{d+1}\times\mathbb Z^{d+1})$ for some dyadic number $\mu>0$. We endow the collection $\mathcal Q$ with the ordering by inclusion. Then for any $Q,Q'\in\mathcal Q$, we must have $|Q\cap Q'|=0$ or $Q\subset Q'$ or $Q'\subset Q$. Now we denote the maximal cubes in $\mathcal Q$ by $\mathfrak Q$. Then it follows that
\begin{align*}
	\mathbb R^{d+1}\times\mathbb R^{d+1} = \bigcup_{Q\in\mathfrak Q}Q
\end{align*}
is an almost disjoint partition of $\mathbb R^{d+1}\times\mathbb R^{d+1}$. We make the label for the cubes in $\mathfrak Q$ as
\begin{align*}
	\mathfrak Q = \{ Q^\alpha_\mu\times\widetilde{Q}^\alpha_\mu : \mu\in2^{\mathbb Z}, \ \alpha\in I \},
\end{align*}
so that for each $\alpha$, the cubes $Q^\alpha_\mu$ and $\widetilde{Q}^\alpha_\mu$ have size $\mu$. Then we have the partition of unit into the characteristic functions
\begin{align*}
	1 = \sum_{\mu\in2^{\mathbb Z}}\sum_{\alpha}\chi_{Q^\alpha_\mu}\chi_{\tilde{Q}^\alpha_\mu}.
\end{align*}
Finally, we define the desired dyadic piece of the operator $\mathcal K_{(k)}$ as 
\begin{align*}
	\mathcal K_{(k)}^\mu = \sum_\alpha \chi_{Q^\alpha_\mu}\mathcal K_{(k)}\chi_{\tilde{Q}^\alpha_\mu}.
\end{align*}
Then the kernel of each $\mathcal K_{(k)}^\mu$ is supported at distance $\mu$ from the diagonal. The main estimate is the following, which is very similar to Lemma 3.2 of \cite{tataru}. See also Lemma 4.1 of \cite{keeltao}.
\begin{lem}\label{lem-endpoint-int}
Let $d\ge3$. The following estimate hold for $(q,\tilde q)$ in a neighborhood of $(\frac{2(d-1+\theta)}{d-3+\theta},\frac{2(d-1+\theta)}{d-3+\theta})$:
\begin{align}
\|\chi_{Q_\mu^\alpha}\mathcal K_{(k)} \chi_{\widetilde{Q}_\mu^\alpha}\|_{L^2_tL^{q'}_x\rightarrow L^2_tL^{\tilde q}_x} \lesssim \langle2^k\rangle^{(1-\frac2q)\frac\theta2+(1-\frac2{\tilde q})\frac\theta2}\mu^{-\mathfrak b(q,\tilde q)},	
\end{align}
	where
	\begin{align*}
		 \mathfrak b (q,\tilde q) = \frac{d-3+\theta}{2}-\frac{d-1+\theta}{2}\left( \frac1q+\frac1{\tilde q} \right).
	\end{align*}
\end{lem}
We prove the above estimates in the three cases:
	\begin{enumerate}
		\item $q=\tilde q=\infty$,
		\item $\tilde q=2$, $2\le q<\frac{2(d-1+\theta)}{d-3+\theta}$,
		\item $ q=2$, $2\le \tilde q<\frac{2(d-1+\theta)}{d-3+\theta}$.
	\end{enumerate}
	If $q=\tilde q=\infty$, we see that
	\begin{align*}
		\|\chi_{Q_\mu^\alpha}\mathcal K_{(k)} \chi_{\widetilde{Q}_\mu^\alpha} u\|_{L^2_tL^\infty_x} \lesssim \mu^{\frac{1}{2}}\|\chi_{Q_\mu^\alpha}\mathcal K_{(k)} \chi_{\widetilde{Q}_\mu^\alpha}u\|_{L^\infty_tL^\infty_x} \lesssim \mu^{\frac 12}\|\chi_{Q_\mu^\alpha} \mathcal K_{(k)} \chi_{\widetilde{Q}_\mu^\alpha}\|_{L^\infty_tL^\infty_x}\|u\|_{L^1_tL^1_x},
	\end{align*}
	and we simply use the kernel bound \eqref{kernel-bdd-diagonal} and the H\"older inequality for the $u$ in $t$, which yields the desired bound. \\
When $\tilde q=2$ and $2\le q<\frac{2(d-1+\theta)}{d-3+\theta}$, we choose $2<p$ so that $(p,q)$ is a Strichartz admissible pair. See \eqref{strichartz-admissible-pair}. Then we have
\begin{align*}
	\|\chi_{Q_\mu^\alpha}\mathcal K_{(k)} \chi_{\widetilde{Q}_\mu^\alpha} u\|_{L^2_tL^2_x} \lesssim \mu^{\frac12}\|\chi_{Q_\mu^\alpha}\mathcal K_{(k)} \chi_{\widetilde{Q}_\mu^\alpha} u\|_{L^\infty_tL^2_x} \lesssim \mu^{\frac12}{\langle 2^k\rangle}^{(1-\frac2q)\frac\theta2}\|\chi_{\tilde{Q}^\alpha_\mu}u\|_{L^{p'}_tL^{q'}_x} \lesssim \mu^{1-\frac1p}\langle2^k\rangle^{(1-\frac2q)\frac\theta2}\|u\|_{L^2_tL^{q'}_x},
\end{align*}
which is the desired bound, since $-1+\frac1p = \mathfrak b(2,q)$. The third case follows by simply interchanging the role of $q$ and $\tilde q$. This completes the proof of Lemma \ref{lem-endpoint-int}. \\
Now we turn our attention to the proof of the endpoint Strichartz estimates \eqref{est-stri-endpoint}. 
We shall make the use of the bilinear real interpolation argument. To do this, we first apply the standard $TT^*$-argument to set the bilinear operator $\mathfrak T_{(k)}^\mu$ given by
\begin{align*}
	\mathfrak T_{(k)}^\mu(u,v) = \int_{s<t} \langle \mathcal K^\mu_{(k)} u(t), \mathcal K^\mu_{(k)} v(s)\rangle_{L^2_x}\,dsdt,
\end{align*}
and we write the equivalent form of the estimate of Lemma \ref{lem-endpoint-int} as
\begin{align}
|\mathfrak T_{(k)}^\mu(u,v)| \lesssim \langle2^k\rangle^{(1-\frac2q)\frac\theta2+(1-\frac2{\tilde q})\frac\theta2}\mu^{-\mathfrak b(q,\tilde q)}\|u\|_{L^2_tL^{q'}_x}\|v\|_{L^2_tL^{\tilde q'}_x} 	.
\end{align}
Even though we cannot obtain any bound for the sum by putting $q=\tilde q =\frac{2(d-1+\theta)}{d-3+\theta}$:
\begin{align*}
	\sum_{\mu\in 2^{\mathbb Z}}|\mathfrak T^\mu_{(k)}(u,v)|,
\end{align*}
which would lead us to the desired endpoint estimates, we can improve the exponent $-\mathfrak b(q,\tilde q)$ of $\mu$ since we have on one hand a two-parameter family of estimates together with various exponential decay factors which holds for any points in an open neighborhood of $(\frac{2(d-1+\theta)}{d-3+\theta},\frac{2(d-1+\theta)}{d-3+\theta})$. The desired bound for the sum above readily follows in the identical manner as the bilinear interpolation argument given by \cite{keeltao}. We omit the details. This completes the proof of the estimates \eqref{est-stri-endpoint}.

\section*{Acknowledgment}
Funded by the Deutsche Forschungsgemeinschaft (DFG, German Research Foundation) -- IRTG 2235 -- Project-ID 282638148.
The first author thanks Daniel Tataru for a helpful discussion about this problem.
The authors also thank the anonymous referees for their very helpful reports, and Robert Schippa for his questions which helped us to correct some errors.
\bibliographystyle{amsplain}
\bibliography{biblio}

\end{document}